\documentclass[11pt]{article}
\usepackage{mathrsfs}
\usepackage{latexsym,amssymb,amsmath,amscd,amscd,amsthm,amsxtra,cancel,xypic}
\usepackage[utf8]{inputenc}
\usepackage[T1]{fontenc}
\usepackage{amssymb}
\usepackage{nicefrac,mathtools,enumitem}
\usepackage[colorinlistoftodos]{todonotes}
\usepackage{soul}
\usepackage{tikz}
\usepackage{xcolor}
\usepackage{geometry}
\def\dn{\mathrm{dn}}

\usepackage{times}
\usepackage{amsthm}
\usepackage{amsmath}
\usepackage{amsfonts}
\usepackage{amssymb}
\usepackage{color}
\usepackage{mathrsfs}
\usepackage{stmaryrd}
\SetSymbolFont{stmry}{bold}{U}{stmry}{m}{n}
\usepackage{pifont}
\usepackage{bm}

\usepackage{tikz-cd}
\usepackage{tikz}
\usepackage{graphicx}
\usepackage[percent]{overpic}
\usepackage{subcaption}

\newcommand{\be }{\begin{equation}}
	\newcommand{\ee }{\end{equation}}

\newtheorem{assumption}{Assumption}

\usepackage{thmtools}

\DeclareMathOperator{\re}{Re}

\declaretheoremstyle[spaceabove=0.25cm,spacebelow=0.25cm,notefont=\normalfont\bfseries, notebraces={(}{)}]{theorem}
\declaretheoremstyle[spaceabove=0.25cm,spacebelow=0.25cm,bodyfont=\normalfont,notefont=\normalfont\bfseries, notebraces={(}{)}]{noital}
\declaretheoremstyle[spaceabove=0.25cm,spacebelow=0.25cm,bodyfont=\normalfont\color{darkgreen},notefont=\normalfont\bfseries, notebraces={(}{)}]{green}
\declaretheoremstyle[spaceabove=0.25cm,spacebelow=0.25cm,bodyfont=\normalfont,notefont=\normalfont\bfseries,qed=$\qedsymbol$,notebraces={(}{)}]{proofstyle}

\declaretheorem[name=Theorem,numberwithin=section,style=theorem]{thm}

\declaretheorem[name=Proposition,sibling=thm,style=theorem]{pro}

\declaretheorem[name=Lemma,sibling=thm,style=theorem]{lem}

\declaretheorem[name=Remark,sibling=thm,style=theorem]{rmk}
\declaretheorem[name=Riemann-Hilbert Problem,sibling=thm,style=theorem]{RHP}

\definecolor{lightblue}{rgb}{0.68, 0.85, 0.9}
\definecolor{lightred}{rgb}{1.0, 0.8, 0.8}
\definecolor{lightgreen}{rgb}{0.8, 1.0, 0.8}
\definecolor{darkgreen}{rgb}{0.0, 0.5, 0.0}
\definecolor{color12lines}{rgb}{0.5, 0.0, 0.0}
\definecolor{color13lines}{rgb}{0.0, 0.5, 0.0}
\definecolor{color23lines}{rgb}{0.0, 0.0, 0.6}
\definecolor{darkgreen}{rgb}{0.0, 0.39, 0.0}
\definecolor{lightgray}{rgb}{0.75, 0.75, 0.75}

\tikzset{
	branchpoint/.style={orange, thick, mark=x, mark options={orange, line width=1.25pt}}
}

\tikzset{
	singularpoint/.style={blue, mark=*, mark options={blue, mark size=1.25pt}}
}

\tikzset{
	stokeslabel/.style={gray!95,font=\tiny}
}

\tikzset{
	cutlabel/.style={orange,font=\tiny}
}

\tikzset{
	sheetlabel/.style={font=\small}
}

\tikzset{
	branchcut/.style={orange,dashed,semithick}
}

\tikzset{
	disccolor/.style={gray!12}
}

\tikzset{
	wall/.style={black,thick}
}

\tikzset{
	path/.style={thick,rounded corners}
}

\tikzset{
	witharrow/.style={
		decoration={markings, mark=at position #1 with {\arrow[sloped]{Latex}}},
		postaction={decorate}
	}
}

\tikzset{
	antistokesmark/.style={semithick, black, radius=1.5pt, fill=yellow}
}

\tikzset{
	withbackgroundrectangle/.style={show background rectangle, background rectangle/.style={fill=gray!7}}
}

\numberwithin{equation}{section}

\usepackage[
hypertexnames=false,
colorlinks=true,
pdfstartview=FitV,
linkcolor=blue,
citecolor=blue,
urlcolor=blue
]{hyperref}

\title{\bf The Direct Scattering Problem for the Defocusing Nonlinear Schr\"odinger Equation with Step-Like Periodic Background}

\date{}

\begin{document}
	
	\maketitle

	\begin{center}
		Dinghao Zhu$^{\dagger, \ddagger}$ \footnote{dhzhu@mail.bnu.edu.cn}
		
		\bigskip
		\begin{minipage}{0.7\textwidth}
			\begin{small}
				\begin{enumerate}
					\item[$^{\dagger}$] {\it School of Mathematical Sciences, Beijing Normal University,\\ Beijing 100875, China}
					\item[$^{\ddagger}$] {\it Institut de Recherche en Math\'ematique et Physique, UCLouvain\\ Chemin du Cyclotron2,
						B-1348 Louvain-La-Neuve, Belgium}
				\end{enumerate}
			\end{small}
		\end{minipage}
	\end{center}
	
	\begin{abstract}
		We consider the defocusing nonlinear Schr\"odinger equation for step-like data connecting two, in general different, genus-one periodic states. Starting from the Jost solutions associated with the left and right backgrounds, we determine the corresponding scattering data and reorganize them through a scalar conjugation and a suitable symmetrization. The resulting formulation identifies the inverse problem with a full dark-soliton gas Riemann--Hilbert problem supplemented by a radiative jump on the real axis. In this way, the effects of the two periodic backgrounds are separated, and a precise link is established between step-like periodic scattering and the full-gas construction. Assuming the absence of discrete eigenvalues and appropriate regularity of the scattering data, we establish the existence and uniqueness of the associated Riemann--Hilbert problem.
	\end{abstract}
	\setcounter{tocdepth}{1}
	\tableofcontents
	\section{Introduction}
	We consider the defocusing nonlinear Schr\"{o}dinger (dNLS) equation
	\begin{equation}\label{eq:dNLS}
		i  u_t+ u_{xx}- 2(|u|^2  - 1)u=0,
	\end{equation}
	with initial step-like periodic wave data
	\begin{equation}\label{initial-periodic}
		\begin{cases}
			u(x,0) - u_0^{\ell}(x,0) \longrightarrow 0, & x\to -\infty, \\
			u(x,0) - u_0^{r}(x,0) \longrightarrow 0, & x\to +\infty,
		\end{cases}
	\end{equation}
	where, for $s\in\{\ell,r\}$, $u_0^s(x,0)$ denotes the initial profile of the genus-one finite-gap solution $u_0^s(x,t)$ given by \eqref{eq:u0}, which can be written in the form of a classical traveling-wave solution \eqref{eq:classical-trav} of the dNLS equation.
	\par 
	The dNLS equation admits a rich class of nondecaying solutions under finite-density boundary conditions. Besides dark solitons, it possesses periodic finite-gap solutions whose amplitudes can be expressed in terms of elliptic functions. Such solutions play an important role in the modulation theory and stability analysis of the dNLS equation; see, for example, \cite{BottmanDeconinckNivala2011,ForestLee1986,GallayHaragus2007}.  
	\par
	The inverse-scattering and Riemann--Hilbert analysis of the dNLS equation with nonzero or step-like plane-wave backgrounds has been developed in a number of works; as discussed in \cite{BiondiniFagerstromPrinari2016,DemontisPrinariVanderMeeVitale2013,FrommLenellsQuirchmayr2025}. Even for genus-zero asymptotic states, the long-time dynamics may generate modulated elliptic-wave regions, as shown through Whitham theory and nonlinear steepest descent in \cite{BiondiniKodama2006,Jenkins2015}. Periodic-background scattering for the dNLS equation was further studied in \cite{BiondiniZhang}. These results naturally lead to the question of how the scattering theory changes when the two spatial infinities are governed by different periodic waves.
	\par
	Related step-like finite-gap problems have also been investigated for the KdV equation in \cite{BoutetDeMonvelEgorovaTeschl2008,EgorovaGrunertTeschl2009,EgorovaTeschl2011,Zhu2026}. For the focusing NLS equation, a direct-scattering construction with two elliptic backgrounds and its connection with a full soliton-gas Riemann--Hilbert problem (RHP) was recently developed in \cite{GravaJenkinsZhangZhang2026}. These works provide the main background for the defocusing step-like elliptic problem considered here.
	\par 
	The dNLS equation admits the following Lax-pair representation \cite{Zakharov-Shabat-1973}:
	\begin{equation}\label{Laxpair}
		\Phi_x=\mathcal{L}\Phi, \qquad 
		i\Phi_t=\mathcal{B}\Phi,
	\end{equation}
	where the $2\times2$ matrices $\mathcal{L}$ and $\mathcal{B}$ are defined by
	\begin{equation*}
		\mathcal{L}=i\sigma_3(Q-kI),\qquad 
		\mathcal{B}=-2ik\mathcal{L}-(Q^2-I)\sigma_3+iQ_x.
	\end{equation*}
	Here
	\begin{equation*}
		Q=Q(x,t)=
		\begin{pmatrix}
			0&\overline{u(x,t)}\\
			u(x,t)&0
		\end{pmatrix},
		\qquad
		I=
		\begin{pmatrix}
			1&0\\
			0&1
		\end{pmatrix},
	\end{equation*}
	and $\sigma_1,\sigma_2,\sigma_3$ denote the Pauli matrices
	\begin{equation*}
		\sigma_1=
		\begin{pmatrix}
			0&1\\
			1&0
		\end{pmatrix},
		\qquad
		\sigma_2=
		\begin{pmatrix}
			0&-i\\
			i&0
		\end{pmatrix},
		\qquad
		\sigma_3=
		\begin{pmatrix}
			1&0\\
			0&-1
		\end{pmatrix}.
	\end{equation*}
	\par 
	For the dNLS equation with symmetric nonzero boundary conditions, the two-sheeted spectral surface can be uniformized by the Joukowski variable 
	\begin{equation*}
		k = k(z) := \frac{z+z^{-1}}{2},
	\end{equation*}
	introduced in the monograph of Faddeev and Takhtajan \cite{FaddeevTakhtajan1987}. This formulation was subsequently used and developed in the inverse-scattering analyses of vector, multicomponent, and scalar dNLS systems in \cite{PrinariAblowitzBiondini2006,PrinariBiondiniTrubatch2011,DemontisPrinariVanderMeeVitale2013}, and later in the long-time analysis of \cite{Jenkins2016}.
	\par 
	In this paper, we study the direct and inverse scattering problems for the dNLS equation with step-like periodic initial data. Starting from the scattering data associated with the two asymptotic elliptic backgrounds, we construct the corresponding RHP and, through a nontrivial transformation, reduce it to a full dark-soliton gas RHP \cite{YanGengChen2026} with an additional radiative component. This establishes a direct connection between step-like periodic scattering and the full-gas framework \cite{DyachenkoZakharovZakharov2016}, and provides a natural starting point for the analysis of the long-time behavior of the resulting solutions.
	\par 
	Hereafter, for \(n\in\mathbb N_0\) and \(1\leq p\leq\infty\), we denote by
	\begin{equation*}
		W^{n,p}(\mathbb R)
		:=
		\left\{
		f\in L^p(\mathbb R):
		\partial_x^j f\in L^p(\mathbb R),\ 1\leq j\leq n
		\right\}
	\end{equation*}
	the standard Sobolev space, and by \(W_{\mathrm{loc}}^{n,p}(\mathbb R)\) the corresponding local Sobolev space. We also set
	\begin{equation*}
		L^{p,s}(\Sigma)
		:=
		\left\{
		f:\langle z\rangle^s f\in L^p(\Sigma)
		\right\},
		\qquad
		\langle z\rangle=(1+|z|^2)^{1/2}.
	\end{equation*}
	On contours in the complex plane, the \(L^p\)-norm is taken with respect to arclength measure.
	\par 
	We now state the main results of the direct scattering analysis. 
	\begin{thm}[Direct scattering]\label{thm:direct-scattering}
		Suppose that for $t = 0$, the initial-value condition (\ref{initial-periodic}) satisfies
		$u(x)-u_0^\ell(x)\in L^1(\mathbb R^-)$,
		$u(x)-u_0^r(x)\in L^1(\mathbb R^+)$.
		Then the left and right Jost solutions exist uniquely and determine the scattering coefficients
		\begin{equation*}
			\begin{aligned}
				a(z)=&\frac{\det[\Psi_1^\ell(x;z),\Psi_2^r(x;z)]}{1-z^{-2}}, \quad z \in \mathbb{C}^+\setminus ( \Sigma_1^{\ell} \cup \Sigma_1^r ),
				&&b(z)=\frac{\det[\Psi_1^r(x;z),\Psi_1^\ell(x;z)]}{1-z^{-2}},
				\quad z\in\mathbb R\setminus\{0\},\\
				b_{1}(z)=&\frac{\det[\Psi_{1}^r(x;z),\Psi_{1}^\ell(x;z)]}{1-z^{-2}},
				\quad z\in\Sigma_1^r,
				&&b_{2}(z)=-\frac{\det[\Psi_{2}^r(x;z),\Psi_{2}^\ell(x;z)]}{1-z^{-2}},
				\quad z\in\Sigma_1^\ell.
			\end{aligned}
		\end{equation*}
		At every noncoincident band endpoint, $a(z)$, $b_1(z)$, and $b_2(z)$ have at most fourth-root singularities. Note that the Jost functions $\Psi^{\ell}$ and $\Psi^r$ are continuous for $z \to \pm 1$, the scattering coefficients $a(z)$ and $b(z)$ have simple poles at $z=\pm1$.
		\par 
		If, in addition, $u(x)-u_0^\ell(x)\in W^{1,1}(\mathbb R^-)$, $u(x)-u_0^r(x)\in W^{1,1}(\mathbb R^+)$, then the corresponding expansions are given by
		\begin{equation*}
			a(z)
			e^{
				-\frac{i}{2}(x_0^\ell-x_0^r)(z-z^{-1})}
			=
			1+\mathcal O(z^{-1}),
			\quad z\to\infty, \qquad
			a(z)
			e^{
				-\frac{i}{2}(x_0^\ell-x_0^r)(z-z^{-1})}
			=
			1+\mathcal O(z),
			\quad z\to0.
		\end{equation*}
		\par 
		If, furthermore, $u\in W_{\mathrm{loc}}^{4,1}(\mathbb R)$, $u-u_0^\ell\in W^{4,1}(\mathbb R^-)$, $u-u_0^r\in W^{4,1}(\mathbb R^+)$, then, for $z\in\mathbb R$,
		\begin{equation*}
			b(z)e^{-\frac{i}{2}(x_0^\ell-x_0^r)(z-z^{-1})}
			=
			\mathcal O(z^{-4}),
			\quad z\to\pm\infty, \qquad 
			b(z)e^{-\frac{i}{2}(x_0^\ell-x_0^r)(z-z^{-1})}
			=
			\mathcal O(z^{4}),
			\quad z\to 0.
		\end{equation*}
	\end{thm}
	Here, for $s\in{\ell,r}$, the Jost solution $\Psi^s$ and the contour $\Sigma_1^s$ are defined in \eqref{eq:direct_z_M_def} and \eqref{eq:Sigma12}, respectively. The real constant $x_0^s$ denotes the spatial phase shift of the corresponding genus-one periodic background $u_0^s(x)$.
	\par 
	Throughout this paper, we assume that no discrete eigenvalues are present. We also use the notation
	$
	\Sigma^\ell:=\Sigma_1^\ell\cup\Sigma_2^\ell$,
	$
	\Sigma^r:=\Sigma_1^r\cup\Sigma_2^r.
	$
	\begin{assumption}\label{ass:no-discrete-spectrum}
		The analytic scattering coefficient $a(z)$ has no zeros in $\mathbb C^+\setminus\left(\Sigma_1^\ell\cup\Sigma_1^r\right)$.
		$\Sigma^{\ell}$ and $\Sigma^r$ have no common endpoints.
	\end{assumption}
	\par 
	Under Assumption~\ref{ass:no-discrete-spectrum}, the scalar function $\gamma(z)$ is introduced through the multiplicative jump problem given explicitly in \eqref{eq:gamma}. We then define
	\begin{equation*}
		a_1(z)
		=
		\sqrt{\frac{a(z)}{\gamma(z)}}
		e^{-\frac{i}{4}(x_0^\ell-x_0^r)(z-z^{-1})},
		\quad 
		a_2(z)
		=
		\sqrt{a(z)\gamma(z)}
		e^{\frac{i}{4}(x_0^\ell-x_0^r)(z-z^{-1})},
		\qquad z\in\mathbb C^+.
	\end{equation*}
	The two band coefficients and the real-axis reflection coefficient are defined by
	\begin{equation*}
		\begin{aligned}
			r_1(z)
			&=
			\frac{a_{1-}(z)}{a_{2-}(z)}
			\frac{e^{-ix_0^r(z-z^{-1})}}
			{a_{2+}(z)a_{1-}(z)-ib_{2-}(z)},
			&&z\in\Sigma_1^\ell,\\
			r_2(z)
			&=
			\frac{a_{2-}(z)}{a_{1-}(z)}
			\frac{e^{ix_0^r(z-z^{-1})}}
			{a_{2-}(z)a_{1+}(z)+ib_{1-}(z)},
			&&z\in\Sigma_1^r,\\
			\rho(z)
			&=
			\frac{b(z)}
			{a_2(z)a_1^*(z)}
			e^{-ix_0^r(z-z^{-1})},
			&&z\in\mathbb R,
		\end{aligned}
	\end{equation*}
	where the superscript $*$ denotes Schwarz conjugation, namely,	$g^*(z):=\overline{g(\overline{z})}$.
	\par 
	Next we introduce the full soliton gas RHP on a background, i.e., in the case $\rho(z) \neq 0$.
	\begin{RHP}\label{RHP2}\
		\begin{enumerate}
			\item Analyticity: $  m(z)$ is analytic for $z\in\mathbb{C}\setminus(\mathbb{R}\cup\Sigma^{\ell}\cup\Sigma^r \cup\{0\})$.
			\item Asymptotic behaviors: $  m(z)= I +\mathcal{O}(z^{-1})$ as $z\to\infty$ and $  m(z)=\frac{\sigma_1}{z}+\mathcal{O}(1)$ as $z\to0$.
			\item Symmetry conditions: $  m(z)=\sigma_1\overline{ m(\bar{z})}\sigma_1=z^{-1} m(z^{-1})\sigma_1$.
			\item Jump conditions: 
			\begin{equation*}
				m_+(z;x,t)= m_-(z;x,t) V^{ m}(z;x,t),
			\end{equation*}
			where
			\begin{equation}\label{jump2}
				V^{ m}(z;x,t) =	\begin{cases}
					\begin{pmatrix}
						1 & 0 \\
						-2ir_1(z)e^{if(z;x,t)} & 1
					\end{pmatrix}, & z\in \Sigma_1^{\ell} \setminus \Sigma_1^r,
					\\
					\begin{pmatrix}
						1 & -2ir_2(z) e^{-if(z;x,t)} \\
						0 & 1
					\end{pmatrix}, & z\in \Sigma_1^r \setminus \Sigma_1^{\ell},
					\\
					\begin{pmatrix}
						\frac{1 - r_1(z) r_2(z)}{1 + r_1(z) r_2(z)} & \frac{-2ir_2(z)}{1 + r_1(z) r_2(z)} e^{-if(z;x,t)} \\
						\frac{-2ir_1(z)}{1 + r_1(z) r_2(z)} e^{if(z;x,t)} & \frac{1 - r_1(z) r_2(z)}{1 + r_1(z) r_2(z)}
					\end{pmatrix}, & z\in \Sigma_1^{\ell} \cap \Sigma_1^r,
					\\
					\begin{pmatrix}
						1 & 2ir_1^*(z)e^{-if(z;x,t)} \\
						0 & 1
					\end{pmatrix}, & z\in \Sigma_2^{\ell} \setminus \Sigma_2^r,
					\\
					\begin{pmatrix}
						1 & 0 \\
						2ir_2^*(z) e^{if(z;x,t)} & 1
					\end{pmatrix}, & z\in \Sigma_2^r \setminus \Sigma_2^{\ell},
					\\
					\begin{pmatrix}
						\frac{1 - r_1^*(z) r_2^*(z)}{1 + r_1^*(z) r_2^*(z)} & \frac{2ir_1^*(z)}{1 + r_1^*(z) r_2^*(z)} e^{-if(z;x,t)} \\
						\frac{2ir_2^*(z)}{1 + r_1^*(z) r_2^*(z)} e^{if(z;x,t)} & \frac{1 - r_1^*(z) r_2^*(z)}{1 + r_1^*(z) r_2^*(z)}
					\end{pmatrix}, & z\in \Sigma_2^{\ell} \cap \Sigma_2^r,
					\\
					\begin{pmatrix}
						1 - |\rho(z)|^2  & - \overline{\rho(z)} e^{-i f(z;x,t)}  \\
						\rho(z) e^{i f(z;x,t)}  &  1
					\end{pmatrix}, & z \in \mathbb{R},
				\end{cases}
			\end{equation}
			where 
			\begin{equation}\label{eq:f}
				f(z;x,t)=x(z-z^{-1})-t(z^2-z^{-2}).
			\end{equation}
		\end{enumerate}
	\end{RHP}
	
	\begin{thm}[Reconstruction]\label{thm:inverse-problem}
		Under the assumptions of Theorem \ref{thm:direct-scattering}, the symmetrized scattering coefficients $r_1(z)$ and $r_2(z)$ are positive. Furthermore, if
		\begin{equation*}
			\rho\in L^{2,2}(\mathbb R)\cap L^{1,2}(\mathbb R) \cap L^{\infty,2}(\mathbb{R}),
			\qquad
			r_1\in L^2(\Sigma_1^\ell) \cap L^{\infty}(\Sigma_1^\ell),
			\qquad
			r_2\in L^2(\Sigma_1^r) \cap L^{\infty}(\Sigma_1^r).
		\end{equation*}
		Then the full soliton gas RHP \ref{RHP2} is uniquely solvable for every $(x,t)\in\mathbb R^2$. The potential function is reconstructed by
		\begin{equation}\label{eq:reconstruction-intro}
			u(x,t)
			=
			\gamma(\infty)
			\lim_{z\to\infty}z\,m_{21}(z;x,t),
		\end{equation}
		where $\gamma(\infty)$ is defined by \eqref{eq:gamma-asym} and $ u(x,t) \in  C^2(\mathbb{R}) \times C^1(\mathbb{R}^+)$ is a classical solution of the dNLS equation.
	\end{thm}
	\begin{rmk}\label{rmk:full-gas-interpretation}
		RHP~\ref{RHP2} may be interpreted as a full dark-soliton gas RHP coupled to a radiative component. When $\rho\equiv0$, the jump on the real axis disappears and the problem is supported only on the elliptic spectral bands $\Sigma^\ell\cup\Sigma^r$. The resulting band jumps have the same algebraic structure as those arising from continuum limits of multisoliton data. If one of the two band coefficients $r_1$ and $r_2$ vanishes identically, these jumps reduce to the one-sided triangular jumps appearing in dark-soliton gas constructions such as \cite{BertolaWangYanZhu2026}. When $\rho\not\equiv0$, the additional jump on $\mathbb R$ incorporates the radiative part of the scattering data. Thus RHP~\ref{RHP2} can be regarded as a full dark-soliton gas RHP in the presence of nonzero radiation. Closely related interpretations for step-like finite-gap scattering problems have been given for the focusing NLS and KdV equations in \cite{GravaJenkinsZhangZhang2026,Zhu2026}; see also \cite{GravaJenkinsZhangZhangInverse2026} for an inverse-scattering problem involving an elliptic background and a full soliton gas.
		\par 
		Unlike a general full-gas RHP, the band data in the present setting are constrained by their direct-scattering origin. In particular, $r_1$ and $r_2$ satisfy the symmetries, sign conditions, and prescribed square-root endpoint behavior inherited from the step-like initial condition. Hence the class obtained here is more restrictive than an abstract admissible class of full-gas data. This is consistent with the constructions in \cite{BertolaWangYanZhu2026}, where the resulting potentials may approach finite-gap backgrounds only at the rate $\mathcal O(|x|^{-1})$ and therefore need not belong to the $L^1$ step-like class considered in Theorem~\ref{thm:direct-scattering}.
		\par 
		Consequently, the scattering data obtained in the present paper form a distinguished subclass of full-gas data, and we do not claim that every admissible full-gas datum lies in the range of the direct-scattering map considered here. The main point of RHP~\ref{RHP2} is that it is derived directly from prescribed step-like elliptic initial data, thereby providing a concrete bridge between classical direct scattering theory and the more general full-gas Riemann--Hilbert framework.
	\end{rmk}
	
	\section{Genus-one finite gap solution}
	In this section, we construct the genus-one finite-gap backgrounds associated with the left and right periodic initial data.
	\subsection{Spectral geometry and band configurations}\label{sec:sepctral geo}
	We begin by determining the spectral bands on $z$-plane associated with the left and right reflectionless periodic initial data. When no two band endpoints coincide, there are six possible configurations. 
	\par 
	For the dNLS equation, the periodic Zakharov--Shabat operator is self-adjoint, and its Floquet spectrum consists of real spectral bands; see \cite{BiondiniZhang}. Under the Joukowski transformation
	$
	k=\frac{1}{2}\left(z+z^{-1}\right),
	$
	which is standard in the inverse-scattering theory for the defocusing NLS equation with nonzero boundary conditions \cite{DemontisPrinariVanderMeeVitale2013,Jenkins2016}, the relevant real spectral intervals are represented by pairs of Schwarz-symmetric arcs of the unit circle. 
	\par 
	The same unit-circle spectral geometry also appears in recent RHP constructions of dark-soliton gases \cite{BertolaWangYanZhu2026}. Accordingly, each genus-one background considered here determines one arc in the upper unit semicircle together with its complex conjugate in the lower half-plane. 
	\par 
	To describe them, we first introduce some notation.
	\par 
	Let
	$
	\mathbb{T}:=\{z\in\mathbb{C}:|z|=1\}.
	$
	For two points $a,b\in\mathbb{T}$, we denote by $(a,b)_{\mathbb{T}}$ the open arc of the unit circle oriented counterclockwise from $a$ to $b$. For $s\in\{\ell,r\}$, write
	\[
	\eta_j^s=e^{ i \theta_j^s},
	\qquad
	0<\theta_1^s<\theta_2^s<\pi,
	\qquad j=1,2,
	\]
	and define the upper and lower spectral bands by
	\begin{equation}\label{eq:Sigma12}
		\Sigma_1^s
		:=
		(\eta_1^s,\eta_2^s)_{\mathbb{T}},
		\qquad
		\Sigma_2^s
		:=
		(\overline{\eta_2^s},
		\overline{\eta_1^s})_{\mathbb{T}}.
	\end{equation}

	Depending on the relative positions of the endpoints, the upper spectral bands are decomposed into
	\[
	\Sigma_1^\ell\setminus\Sigma_1^r,
	\qquad
	\Sigma_1^\ell\cap\Sigma_1^r,
	\qquad
	\Sigma_1^r\setminus\Sigma_1^\ell,
	\]
	with the corresponding decomposition on the lower half-plane obtained by complex conjugation. The six possible relative configurations of $\Sigma_1^\ell$ and $\Sigma_1^r$ are illustrated in Figure \ref{fig:six-cases}.

	\begin{figure}[htbp]
		\centering
		
		\begin{subfigure}[b]{0.31\textwidth}
			\centering
			\includegraphics[width=\linewidth]{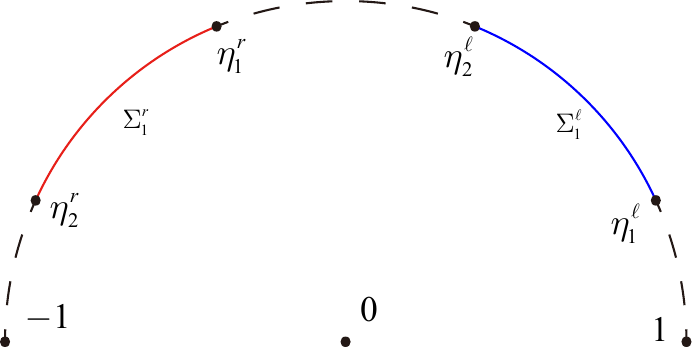}
			\caption{}
			\label{fig:cut1}
		\end{subfigure}
		\hfill
		\begin{subfigure}[b]{0.31\textwidth}
			\centering
			\includegraphics[width=\linewidth]{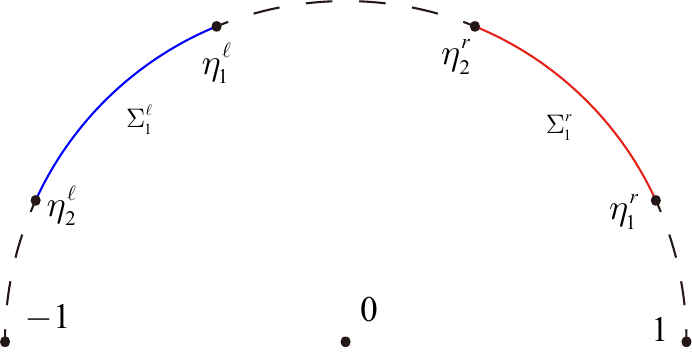}
			\caption{}
			\label{fig:cut2}
		\end{subfigure}
		\hfill
		\begin{subfigure}[b]{0.31\textwidth}
			\centering
			\includegraphics[width=\linewidth]{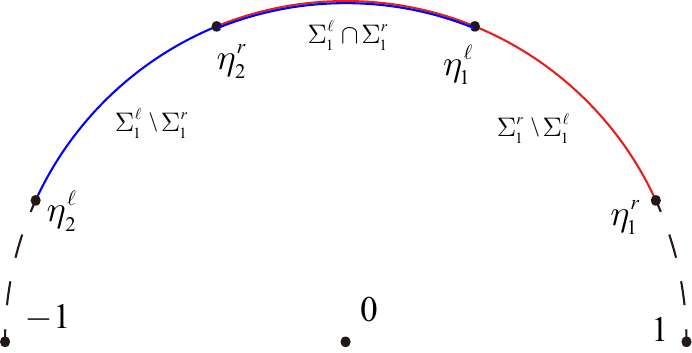}
			\caption{}
			\label{fig:cut3}
		\end{subfigure}
		
		\medskip
		
		\begin{subfigure}[b]{0.31\textwidth}
			\centering
			\includegraphics[width=\linewidth]{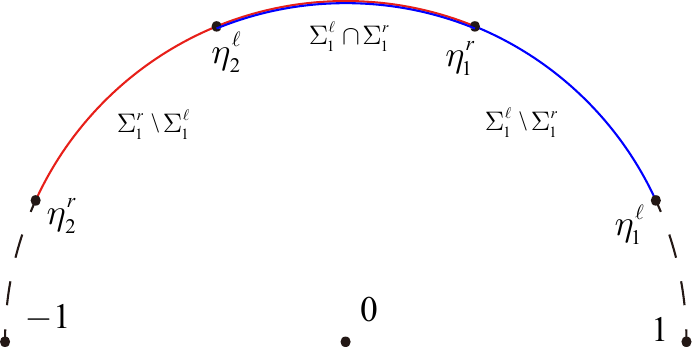}
			\caption{}
			\label{fig:cut4}
		\end{subfigure}
		\hfill
		\begin{subfigure}[b]{0.31\textwidth}
			\centering
			\includegraphics[width=\linewidth]{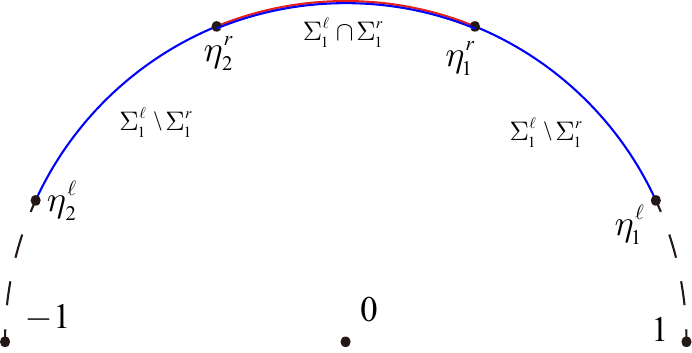}
			\caption{}
			\label{fig:cut5}
		\end{subfigure}
		\hfill
		\begin{subfigure}[b]{0.31\textwidth}
			\centering
			\includegraphics[width=\linewidth]{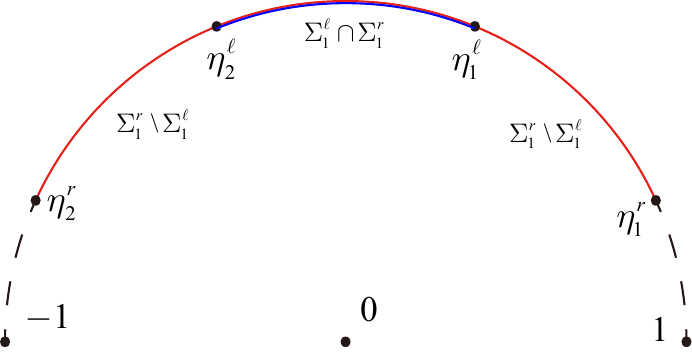}
			\caption{}
			\label{fig:cut6}
		\end{subfigure}
		
		\caption{The six possible configurations of the spectral bands on $\mathbb{T}\cap\mathbb{C}^+$. The case in the lower half-plane can be obtained by complex conjugation. Here, the spectral bands shown in red correspond to the right background wave, while those shown in blue correspond to the left background wave.
		}
		\label{fig:six-cases}
	\end{figure}

	\subsection{Riemann surfaces and the Joukowski transformation}
	In this section, we describe the Riemann surfaces arising naturally in the $z$- and $k$-planes and clarify their relation through the Joukowski transformation
	$
	k=\frac{1}{2}\left(z+\frac{1}{z}\right).
	$
	\par 
	In Subsection \ref{sec:sepctral geo}, we introduced the corresponding spectral curves. We now discuss their branch points, branch cuts, and sheet structures. We then describe how these objects are mapped into one another under the Joukowski transformation. This correspondence will also be used to relate the Abelian differentials and Abelian integrals defined on the two surfaces. For convenience, throughout this section we suppress the superscripts $r$ and $\ell$ on $\eta_j$, $j =1 ,2$.
	\par 
	Introduce the Riemann surface corresponding to the $z$-plane and $k$-plane:
	\begin{gather*}
		\mathcal{S}^{(z)} = \{ ( z , R ) \in \mathbb{C}^2 : R^2 = (z - \eta_1 )(z - \overline{\eta_1} )(z - \eta_2 )(z - \overline{\eta_2} ) \},  \\
		\mathcal{S}^{(k)} = \{ ( k , \mathcal{R} ) \in \mathbb{C}^2 : \mathcal{R}^2 = (k + 1 )(k - \re\eta_1 )(k - \re\eta_2 )(k - 1 ) \},
	\end{gather*}
	where $R(z)$ and $\mathcal{R}(k)$ denote the corresponding square roots on the sheets characterized by $R(z) \sim z^2$ and $ \mathcal{R}(k) \sim k^2 $ at infinity. The canonical homology bases on $ \mathcal{S}^{(z)} $ and $ \mathcal{S}^{(k)} $ is shown in Figure \ref{fig:homology-bases}. Then the quasi-momentum $dp$ and quasi-energy $dq$ can be expressed by the second kind of Abel differential on $\mathcal{S}^{(z)}$:
	\begin{equation*}
		dp = \frac{ z^4 + A z^3 + B z^2 + A z + 1 }{2 z^2 R(z)} dz, \qquad  dq = -\frac{ z^6 + A z^5 + C z^4 + D z^3 + C z^2 + A z + 1 }{z^3 R(z)} dz,
	\end{equation*}
	which can also be encoded on $\mathcal{S}^{(k)}$:
	\begin{equation*}
		dp = \frac{4k^2 + 2 A k + B - 2}{4 \mathcal{R}(k)} dk, \qquad dq = -\frac{ 8k^3 + 4 A k^2 + (2 C - 6)k + D - 2A }{2 \mathcal{R}(k)} dk.
	\end{equation*}
	The constants $A$, $B$, $C$ and $D$ are determined by imposing
	\begin{equation*}
		\begin{aligned}
			&dp = \left(\frac{1}{2} + \mathcal{O}(z^{-2})\right) dz, \quad dq = \left(-z + \mathcal{O}(z^{-2})\right) dz, \qquad z\to \infty_+, \\
			&dp = \left(\frac{1}{2z^2} + \mathcal{O}(1)\right) dz, \quad dq = \left(-\frac{1}{z^{3}} + \mathcal{O}(1)\right) dz, \qquad z\to 0_+,
		\end{aligned}
	\end{equation*}
	together with the $a$-circle normalization (see Figure \ref{fig:jump})
	\begin{equation*}
		\oint_{a} dp = \oint_{a} dq = 0.
	\end{equation*}
	A direct calculation gives the following expression for $A$, $B$, $C$ and $D$:
	\begin{gather*}
		A = -( \re\eta_1 + \re\eta_2 ), \quad B = 2 + 2( \re\eta_1 - \re\eta_2 ) - 2(1 + \re\eta_1) ( 1 - \re\eta_2 ) \frac{ E(m_1) }{ K(m_1) }, \\
		C = \frac{2\eta_1\eta_2(\eta_1+\eta_2)^2 - (\eta_1-\eta_2)^2(1+\eta_1^2\eta_2^2)}{8\eta_1^2\eta_2^2},  \\
		D = (\re\eta_1 + \re\eta_2) ( \re\eta_1 - \re\eta_2 ) - ( \re\eta_1 + \re\eta_2 ) ( 1 + \re\eta_1 ) ( 1 - \re\eta_2 )
		\left[ \frac{ (1 + \sqrt{m_1})^2 }{2}  \frac{ E(m) }{ K(m) }  + \frac{ 1 - m_1 }{2} \right],
	\end{gather*}
	where 
	\begin{equation*}
		m=\left(\frac{2\sqrt{\operatorname{Im}\eta_1\operatorname{Im}\eta_2}}
		{\lvert\eta_1-\eta_2^{-1}\rvert}\right)^2
		= \frac{4\sqrt{m_1}}{(1 + \sqrt{m_1})^2}, \qquad 
		m_1 = \frac{ (1 - \re\eta_1) (1 + \re\eta_2) }{ (1 + \re\eta_1) (1 - \re\eta_2) }.
	\end{equation*}
	Here we recall the definitions of the complete elliptic integrals of the first and second kind:
	\begin{equation*}
		K(m) = \int_{0}^{\frac{\pi}{2}} \frac{d \theta}{ \sqrt{ 1 - m \sin^2 \theta } } , \qquad 
		E(m) = \int_{0}^{\frac{\pi}{2}} \sqrt{ 1 - m \sin^2 \theta } d \theta.
	\end{equation*}
	By applying the Riemann bilinear relation, we have 
	\begin{equation}\label{eq:Omega12}
		\Omega_1 := -\oint_{b} dp = -\frac{ \pi | \eta_1 - \overline{\eta_2} | }{ K (m) }, \qquad 
		\Omega_2 := -\oint_{b} dq = v\frac{ \pi | \eta_1 - \overline{\eta_2} | }{ K (m) },
	\end{equation}
	where $ v = -\re (\eta_1 + \eta_2 ) $ is the velocity of the traveling wave.
	\par 
	We choose a base point $\eta_1$ on Riemann surface $ \mathcal{S}^{(z)} $ and $\re \eta_1 $ on $ \mathcal{S}^{(k)} $. Then define the Abel integrals by quasi-momentum and quasi-energy:
	\begin{equation*}
		\text{on the $z$-plane:}\quad
		p(z) = \int_{\eta_1}^{z} dp, \quad
		q(z) = \int_{\eta_1}^{z} dq,
		\qquad
		\text{on the $k$-plane:}\quad
		p(k) = \int_{\re\eta_1}^{k} dp, \quad
		q(k) = \int_{\re\eta_1}^{k} dq.
	\end{equation*}
	\par 
	Now we have the expansions in the $z$-plane
	\begin{gather*}
		p(z) = \frac{z}{2} + p_{\infty} + \frac{p_1^{(z)}}{z} + o(z^{-1}), \quad q(z) = -\frac{z^2}{2} + q_{\infty} + o(1), \qquad z \to \infty, \\
		p(z) = -\frac{1}{2z} - p_{\infty} - p_1^{(z)} z + o(z), \quad q(z) = \frac{1}{2z^2} - q_{\infty} + o(1), \qquad z \to 0,
	\end{gather*}
	together with the expansions in the $k$-plane
	\begin{equation}\label{eq:p(k)}
		p(k) = k + p_{\infty}^{(k)}  + \frac{p_1^{(k)}}{k} + o(k^{-1}), \quad q(k) = -2k^2 + q_{\infty}^{(k)}  + o(1), \qquad k \to \infty,
	\end{equation}
	where 
	\begin{equation*}
		p_1^{(z)} = (1 + \re\eta_1)(1 - \re\eta_2) \frac{E(m_1)}{K(m_1)} + \frac{1}{2} - \frac{(2 + \re\eta_1 - \re\eta_2)^2}{4}, \quad 
		p_1^{(k)} = \frac{p_1^{(z)} - \frac{1}{2} }{2}. 
	\end{equation*}
	By Joukowski transformation we further have
	\begin{equation*}
		p_{\infty} = p_{\infty}^{(k)}, \quad  q_{\infty} = q_{\infty}^{(k)} - 1.
	\end{equation*}
	Let  $\Omega_0\in\mathbb{R}$ be the phase constant fixing the spatial translation of the periodic wave. Using \eqref{eq:Omega12}, the dispersion information can be encoded by defining 
	\begin{equation*}
		\Omega = \Omega_0 + x\Omega_1 +t\Omega_2 = \Omega_1( x - x_0 + v t ),
	\end{equation*}
	where $ x_0 = -\frac{\Omega_0}{ \Omega_1} $.
	\par 
	To analyze the domains of analyticity of the Jost functions introduced in Section \ref{Secfast}, we introduce the following lemma.
	\begin{lem}\label{lemfoin}
		Away from the spectral bands, the imaginary part of the normalized quasi-momentum $p$ has signs 
		\begin{equation*}
			\operatorname{Im}p(z)>0,\qquad z\in\mathbb C^+\setminus\overline{\Sigma_1},
			\qquad
			\operatorname{Im}p(z)<0,\qquad z\in\mathbb C^-\setminus\overline{\Sigma_2},
		\end{equation*}
		where \(\overline{\Sigma_j}\) denotes the closure of the contour \(\Sigma_j\), $j =1, 2$, including its endpoints.
	\end{lem}
	
	\begin{proof}
		We prove the upper-half-plane inequality. Since $p$ is analytic in $\mathbb C^+\setminus\overline{\Sigma_1}$, its imaginary part is harmonic there. By the Schwarz symmetry of $p$ and the boundary relations on the spectral band, the boundary values of $p$ on $\mathbb R\setminus\{0\}$ and on both sides of $\Sigma_1$ are real. Hence $\operatorname{Im}p(z)=0$ on these parts of the boundary.
		\par 
		For $0<\varepsilon<1<R$, consider
		$
		D_{\varepsilon,R}:=\left\{z\in\mathbb C^+\setminus\overline{\Sigma_1}:\varepsilon<|z|<R\right\}.
		$
		The asymptotic expansions of $p$ at infinity and at the origin give
		\begin{equation*}
			\operatorname{Im}p(z)=\frac12\operatorname{Im}z+o(1),\quad z\to\infty, \qquad 
			\operatorname{Im}p(z)=\frac{\operatorname{Im}z}{2|z|^2}+o(|z|^{-1}),\quad z\to0.
		\end{equation*}
		Thus $\operatorname{Im}p$ is nonnegative on the two semicircular parts of $\partial D_{\varepsilon,R}$ when $\varepsilon$ is sufficiently small and $R$ is sufficiently large. The minimum principle therefore yields
		$
		\operatorname{Im}p(z)\geq0$, $z\in\mathbb C^+\setminus\overline{\Sigma_1}.
		$
		Since $\operatorname{Im}p$ is positive near infinity, it is not identically zero. The strong minimum principle then gives
		$
		\operatorname{Im}p(z)>0$,  $z\in\mathbb C^+\setminus\overline{\Sigma_1}.
		$
		The lower-half-plane inequality follows from Schwarz symmetry.
	\end{proof}

	\subsection{Riemann--Hilbert problem formulation}\label{subrpf}
	In this subsection, we construct the reflectionless genus-one periodic background used throughout the subsequent scattering analysis. We first formulate the corresponding model Riemann--Hilbert problem on the $z$-plane and then transform it to the $k$-plane. The transformed problem can be solved explicitly in terms of the genus-one Abel map and the Riemann theta function. Finally, by relating the resulting matrix solution to the dNLS Lax pair, we recover a globally smooth, nonsingular genus-one periodic solution of the defocusing NLS equation.
	\par 
	For simplicity, throughout this subsection we continue to omit the superscripts $\ell$ and $r$ associated with the left and right backgrounds, respectively.
	\par 
	We begin by formulating the following RHP.
	\begin{RHP}\label{RH5}\
		\begin{enumerate}
			\item Analyticity: $ Y(z)$ is analytic for
			$z\in\mathbb{C}\setminus( \Sigma_1 \cup \Sigma_2 \cup \Sigma_0 \cup\{0\})$.
			\item Asymptotic behaviors: $ Y(z)= I +\mathcal{O}(z^{-1})$ as $z\to\infty$ and $ Y(z)=\frac{\sigma_1}{z}+\mathcal{O}(1)$ as $z\to0$.
			\item Symmetry: $ Y(z)=\sigma_1\overline{Y(\bar{z})}\sigma_1=z^{-1}Y(z^{-1})\sigma_1$.
			\item Jump matrices: 
			\begin{equation*}
				Y_+(z;x,t)=Y_-(z;x,t)\begin{cases}
					\begin{pmatrix}
						0 & -i  \\
						-i & 0
					\end{pmatrix}, & z\in \Sigma_1 ,
					\\
					\begin{pmatrix}
						0 & i \\
						i  & 0
					\end{pmatrix}, & z\in \Sigma_2,
					\\
					e^{i \Omega \sigma_3}, & z\in \Sigma_0.
				\end{cases}
			\end{equation*}
			\item $Y(z)$ admits fourth-root behavior at the endpoints.
		\end{enumerate}	
	\end{RHP}
	Here, $\Sigma_j$, $j =1,2$, are defined in \eqref{eq:Sigma12}, and $\Sigma_0 = (\eta_{2}, \overline{\eta_2})_{\mathbb{T}}$.
	This genus-one RHP \ref{RH5} can be solved on $z$-plane directly, see Appendix A of \cite{BertolaWangYanZhu2026}. Here we briefly introduce the necessary notation.
	As illustrated in Figure~\ref{fig:homology-bases}, we introduce the Abel map associated with the genus-one Riemann surface $\mathcal{S}^{(z)}$. Let $a$ and $b$ denote the canonical homology cycles, and let
	\begin{equation*}
		\omega_z=\frac{c_z\,dz}{R(z)}, \qquad c_z = \frac{i|\eta_1 - \eta_2^{-1}|}{4K(m)},
	\end{equation*}
	be the normalized holomorphic differential satisfying
	\begin{equation*}
		\oint_a\omega_z=1,
		\qquad
		\tau_z=\oint_b\omega_z    = \frac{iK(1-m)}{K(m)}.
	\end{equation*}
	The Abel map is defined by
	\begin{equation*}
		J_z(z)=\int_{1}^{z}\omega_z,
	\end{equation*}
	where, the base point $1$ is chosen to agree with the base point on the $k$-plane introduced below. One can also obtain the symmetry $ J_z(z) + J_z(z^{-1}) = 0$. These formulas are introduced in preparation for the subsequent analysis of the endpoint singularities.
	\par 
	For convenience here, we will use the transformation $ F(k) = \frac{1}{ \sqrt{1 - z(k)^{-2}} } Y(z(k))  $ in \cite{BertolaWangYanZhu2026} to obtain the following RHP on $k$-plane.
	\begin{RHP}\label{RHp2}\
		\begin{enumerate}
			\item Analyticity: $ F(k)$ is analytic for $k\in\mathbb{C}\setminus \left( (-1,\re \eta_{2}) \cup (\re \eta_{1},1) \right)$ with jump conditions:
			\begin{equation*}
				F_+(k)=F_-(k)\begin{cases}
					\begin{pmatrix}
						0 & -ie^{i \Omega } \\
						-ie^{-i \Omega } & 0
					\end{pmatrix}, & k\in (-1,\re \eta_{2}),
					\\
					\begin{pmatrix}
						0 & -i  \\
						-i  & 0
					\end{pmatrix}, & k\in (\re \eta_{1},1). 
				\end{cases}
			\end{equation*}
			\item Asymptotic behaviors: $ F(k)= I +\mathcal{O}(k^{-1})$ as $k\to\infty$.
			\item Symmetry conditions: $ F(k)=\sigma_1\overline{F(\bar{k})}\sigma_1$.
			\item $F(k)$ admits fourth-root behavior at the endpoints.
		\end{enumerate}
	\end{RHP}
	Now we solve the genus-one RHP \ref{RHp2}. 
	\par 
	Recall the function $\mathcal{R}(k) := \sqrt{(k + 1 )(k - \re\eta_1 )(k - \re\eta_2 )(k - 1 )}$ with $\mathcal{R}(k) \sim k^{2}$, as $k \to \infty$. We define the modified Abel differentials of the first kind as $\omega = \frac{c dk}{\mathcal{R}(k)} $ with $c := \frac{ i\sqrt{ (1 + \re \eta_1) (1 - \re \eta_2) } }{ 4 K(m_1)  }$. The homology basis for $\mathcal{S}^{(k)}$ is shown in Figure \ref{fhsk1}.  The period ratio is defined by $ \tau := \oint_{b_1} \omega = i \frac{ K( 1 - m_1 ) }{K(m_1)}$.
	\begin{figure}[htbp]
		\centering
		\begin{subfigure}[b]{0.45\textwidth}
			\centering
			\includegraphics[width=3cm]{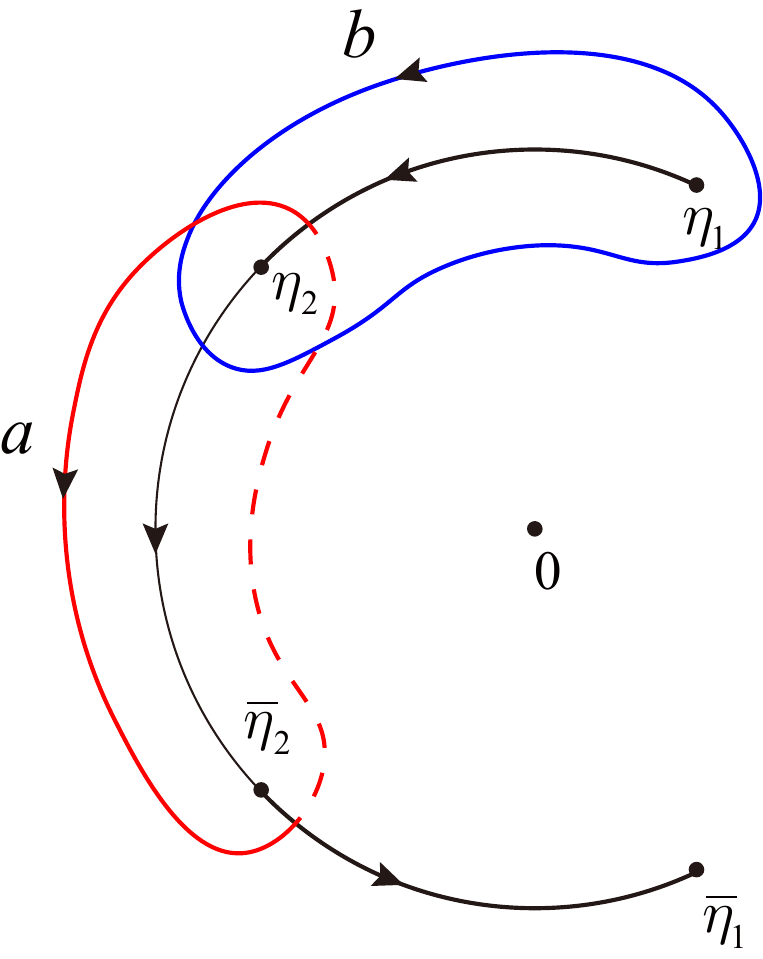}
			\caption{The canonical homology basis cycles on $\mathcal{S}^{(z)}$.}
			\label{fig:jump}
		\end{subfigure}
		\hfill
		\begin{subfigure}[b]{0.45\textwidth}
			\centering
			\includegraphics[width=\textwidth]{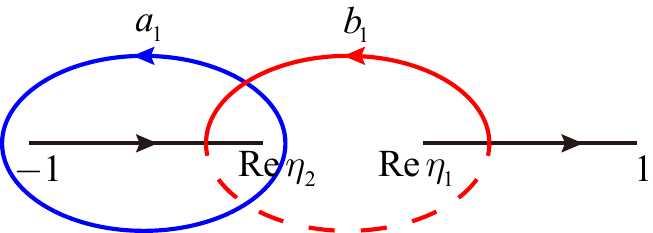}
			\caption{The homology basis for $\mathcal{S}^{(k)}$.}
			\label{fhsk1}
		\end{subfigure}
		\caption{The canonical homology bases on $\mathcal{S}^{(z)}$ and $\mathcal{S}^{(k)}$.}
		\label{fig:homology-bases}
	\end{figure}
	\par 
	Define the function $h(k) :=  \left[ \frac{ ( k - 1 ) ( k - \re \eta_2 ) }{ ( k + 1 ) ( k - \re \eta_1 ) }  \right]^{\frac{1}{4}} $, which satisfies
	\begin{equation*}
		\frac{h_+(k) + h_+(k)^{-1} }{2} = i \frac{ h_-(k) - h_-(k)^{-1} }{2}, \quad k \in (-1 , \re \eta_2) \cup ( \re \eta_1, 1 ).
	\end{equation*}
	\par 
	Consider the Abel-Jacobi map 
	\begin{equation}\label{Abelmap}
		J(k):= \int_{1}^{k} \omega,
	\end{equation}
	and the Riemann theta function
	\begin{equation*}
		\theta_3(k) := \theta_3(k; \tau) = \sum_{n \in \mathbb{Z}} \exp \left( \pi i n^2 \tau  + 2\pi i n k \right).
	\end{equation*}
	The solution of RHP \ref{RHp2} is then given by
	\begin{equation}\label{solF}
		F(k) = 
		\frac{\theta_3(0)}{\theta_3(\frac{\Omega}{2\pi})}
		\begin{pmatrix}
			\frac{h(k) + h(k)^{-1}}{2} \frac{\theta_3( -\frac{\Omega}{2\pi} + J(k) - J(\infty) )}
			{\theta_3( J(k) - J(\infty) )} & 
			\frac{h(k) - h(k)^{-1}}{-2} \frac{\theta_3( \frac{\Omega}{2\pi} + J(k) + J(\infty) )}
			{\theta_3( J(k) + J(\infty) )}   \\
			\frac{h(k) - h(k)^{-1}}{-2} \frac{\theta_3( -\frac{\Omega}{2\pi} + J(k) + J(\infty) )}
			{\theta_3( J(k) + J(\infty) )} & 
			\frac{h(k) + h(k)^{-1}}{2} \frac{\theta_3( \frac{\Omega}{2\pi} + J(k) - J(\infty) )}
			{\theta_3( J(k) - J(\infty) )}
		\end{pmatrix}.
	\end{equation}
	\begin{rmk}\label{rem:relation-Abel-maps}
		The Abel maps constructed directly on the $z$-plane and on the $k$-plane are related by the Landen transformation. More precisely,
		\begin{equation*}
			m=\frac{4\sqrt{m_1}}{(1+\sqrt{m_1})^2},
		\end{equation*}
		and the corresponding normalization constants coincide:
		\begin{equation*}
			c_z=c
			=
			\frac{
				i\sqrt{(1+\operatorname{Re}\eta_1)
					(1-\operatorname{Re}\eta_2)}
			}{
				4K(m_1)
			}.
		\end{equation*}
		Under the Joukowski transformation $k=\frac{1}{2}\left(z+z^{-1}\right)$, one has $ \omega = 2 \omega_z $. Consequently, the corresponding period ratios satisfy
		$
		\tau=2\tau_z.
		$
		\par 
		Since the base point is chosen to be $1$ in both constructions, the Abel map on the $k$-plane, pulled back to the $z$-plane, is related to the Abel map constructed directly on the $z$-plane by
		\begin{equation*}
			\mathcal J(z)
			:=
			J\left(\frac{z+z^{-1}}{2}\right)
			=
			2J_z(z),
			\quad
			\operatorname{mod}\bigl(\mathbb Z+\tau\mathbb Z\bigr).
		\end{equation*}
		Thus, the two constructions describe the same genus-one geometry with different period lattices: the direct $z$-plane construction uses $(J_z(z),\tau_z)$, whereas the pullback of the $k$-plane construction uses $(2J_z(z),2\tau_z)$.
		\par 
		We further obtain the endpoint values
		\begin{gather}
			J_z(\eta_1) = -\frac{1}{4},\quad J_z(\eta_2) = -\frac{1}{4} - \frac{\tau_z}{2},\quad J_z(\overline{\eta_1}) = \frac{1}{4},\quad J_z(\overline{\eta_2}) = \frac{1}{4} - \frac{\tau_z}{2}, 
			\\
			J(1) = 0,\quad J(\operatorname{Re}\eta_1) = -\frac{1}{2},\quad J(\operatorname{Re}\eta_2) = -\frac{1}{2} + \frac{\tau}{2},\quad J(-1) = \frac{\tau}{2}.
		\end{gather}
		Note that the interval $(\operatorname{Re}\eta_2,\operatorname{Re}\eta_1)$ is not a branch cut in the $k$-plane, and hence the Abel map has no upper and lower boundary values there as a function of the planar variable $k$. Instead, each $k\in(\operatorname{Re}\eta_2,\operatorname{Re}\eta_1)$ has two lifts to the two sheets of $\mathcal S^{(k)}$. Under the Joukowski transformation, the positive and negative sides of the corresponding band in the $z$-plane are mapped to these two lifts, respectively. Thus, the two values obtained from the $z$-plane should be interpreted as Abel integrals along different lifted paths on $\mathcal S^{(k)}$.
		\par 
		Taking the value at $\eta_2$ as an example, we obtain
		\begin{equation*}
			2J_z(\eta_2) = 2\int_{1}^{\eta_1} \omega_{z } + 2 \int_{\eta_1}^{\eta_2} \omega_{z+} 
			= -\frac{1}{2} - \tau_z
			= \int_{ 1 }^{\re \eta_1} \omega_+ + \left.\int_{ \re \eta_1 }^{\re \eta_2} \omega\right|_{\text{lower sheet}}
			= -\frac{1}{2} - \frac{\tau}{2}
			= J(\re \eta_2) - \tau.
		\end{equation*}
	\end{rmk}
	\par 
	Next, we connect the solution of RHP \ref{RH5} with the Lax pair \eqref{Laxpair}.
	\begin{pro}\label{prop:periodic-background}
		The matrix 
		\begin{equation*}
			F_0(x,t;k) = e^{i ((x - x_0 )p_{\infty}^{(k)}+t (q_{\infty}^{(k)}-1))\sigma_3}
			F(x,t;k) e^{-i[(x-x_0) p(k) + tq(k) ]\sigma_3},
		\end{equation*}
		solves the Lax pair \eqref{Laxpair} with potential 
		\begin{equation}\label{eq:u0}
			u_0(x,t) = \frac{ 2 - \re\eta_1 + \re\eta_2 }{2} 
			\frac{ \theta_3(0) \theta_3( \frac{\Omega}{2\pi}  + 2J(\infty) ) }{ \theta_3( \frac{\Omega}{2\pi} ) \theta_3( 2 J(\infty) )  }
			e^{-2i( (x - x_0 ) p_{\infty}^{(k)} + t (q_{\infty}^{(k)} - 1 ))},
		\end{equation}
		whose modulus can be expressed in terms of Jacobi elliptic functions:
		\begin{equation}\label{eq:modu0}
			|u_0(x,t)|^2 = \rho_1-(\rho_1-\rho_3)\operatorname{dn}^2\left(
			\sqrt{\rho_1-\rho_3}\left(x-(\operatorname{Re}\eta_1+\operatorname{Re}\eta_2)t-x_0\right)
			-K(m_1);m_1
			\right),
		\end{equation}
		where $\rho_1 = \frac{ (2 + \re\eta_1 - \re\eta_2)^2 }{4} $ and $ \rho_3 = \frac{ ( \re\eta_1 + \re\eta_2)^2 }{4} $.
	\end{pro}
	\begin{proof}
		We first observe that the jump matrices of \(F_0\) are independent of \(x\) and \(t\). The jump conditions of $F_0$ are given by
		\begin{equation*}
			F_{0+}(k) =  F_{0-}(k) (-i \sigma_1), \quad  k \in (-1, \re\eta_2) \cup ( \re\eta_1 , 1 ).
		\end{equation*}
		Hence the two matrices
		$
		(F_0)_xF_0^{-1}$,
		$
		(F_0)_tF_0^{-1}
		$
		have no jumps across the contour. The possible singularities at the branch points are removable by the local growth condition. Therefore both matrices are entire functions of \(k\).
		
		Assume that, as \(k\to\infty\),
		\begin{equation*}
			F=I+\frac{F^{(1)}}{k}
			+\frac{F^{(2)}}{k^2}+\mathcal{O}(k^{-3}), \qquad 
			F^{-1}=I-\frac{F^{(1)}}{k}
			+\frac{(F^{(1)})^2-F^{(2)}}{k^2}+\mathcal{O}(k^{-3}),
		\end{equation*}
		then
		\[
		F\sigma_3F^{-1}
		=
		\sigma_3-\frac{[\sigma_3,F^{(1)}]}{k}
		+\frac{[\sigma_3,F^{(1)}]F^{(1)}
			-[\sigma_3,F^{(2)}]}{k^2}
		+\mathcal{O}(k^{-3}).
		\]
		Using \eqref{eq:p(k)}, we get
		\[
		\begin{aligned}
			(F_0)_xF_0^{-1}
			&=ip_{\infty}^{(k)}\sigma_3 +  e^{i ((x - x_0 )p_{\infty}^{(k)}+t (q_{\infty}^{(k)}-1))\sigma_3} 
			\left( F_xF^{-1} - i p(k) F \sigma_3 F^{-1} \right) 
			e^{-i ((x - x_0 )p_{\infty}^{(k)}+t (q_{\infty}^{(k)}-1))\sigma_3}  \\
			&=
			-ik\sigma_3
			+ie^{i ((x - x_0 )p_{\infty}^{(k)}+t (q_{\infty}^{(k)}-1))\sigma_3} [\sigma_3,F^{(1)}]
			e^{-i ((x - x_0 )p_{\infty}^{(k)}+t (q_{\infty}^{(k)}-1))\sigma_3}+\mathcal{O}(k^{-1}).
		\end{aligned}
		\]
		Since the left-hand side is entire and grows at most linearly at infinity, Liouville's theorem gives
		\[
		(F_0)_xF_0^{-1} =
		-ik\sigma_3+ie^{i ((x - x_0 )p_{\infty}^{(k)}+t (q_{\infty}^{(k)}-1))\sigma_3} [\sigma_3,F^{(1)}]
		e^{-i ((x - x_0 )p_{\infty}^{(k)}+t (q_{\infty}^{(k)}-1))\sigma_3}.
		\]
		Define
		\begin{equation}\label{eq:Q_0}
			Q_0(x,t)=
			\begin{pmatrix}
				0& 2 F^{(1)}_{12} e^{2i ((x - x_0 )p_{\infty}^{(k)}+t (q_{\infty}^{(k)}-1))} \\
				2 F^{(1)}_{21} e^{-2i ((x - x_0 )p_{\infty}^{(k)}+t (q_{\infty}^{(k)}-1))} &0
			\end{pmatrix}.
		\end{equation}
		Then
		\[
		e^{i ((x - x_0 )p_{\infty}^{(k)}+t (q_{\infty}^{(k)}-1))\sigma_3} [\sigma_3,F^{(1)}]
		e^{-i ((x - x_0 )p_{\infty}^{(k)}+t (q_{\infty}^{(k)}-1))\sigma_3}=\sigma_3 Q_0,
		\]
		and therefore
		\[
		(F_0)_xF_0^{-1}
		=
		-ik\sigma_3+i\sigma_3Q_0
		=
		i\sigma_3(Q_0-kI).
		\]
		Thus
		\[
		(F_0)_x=\mathcal L_0F_0,
		\qquad
		\mathcal L_0=i\sigma_3(Q_0-kI).
		\]
		
		It remains to check the time part. From the coefficient of \(k^{-1}\) in the preceding \(x\)-expansion one obtains
		\begin{equation}\label{Y123}
			(F^{(1)})_x
			=
			i\left(
			[\sigma_3,F^{(1)}](F^{(1)} - p_{\infty}^{(k)} I)
			-[\sigma_3,F^{(2)}]
			+p_1^{(k)}\sigma_3
			\right).
		\end{equation}
		The off-diagonal entries in \eqref{Y123} give
		\begin{equation}\label{Y123y}
			F^{(1)}_{12x} = 2i (F^{(1)}_{12} F^{(1)}_{22} - F^{(2)}_{12} - F_{12}^{(1)} p_{\infty}^{(k)}),
			\qquad
			F^{(1)}_{21x} = 2i (F^{(2)}_{21} - F^{(1)}_{21} F^{(1)}_{11} + F_{21}^{(1)}  p_{\infty}^{(k)}  ),
		\end{equation}
		and from definition \eqref{eq:Q_0}, we have
		\begin{equation}\label{eq:Q1221}
			\begin{aligned}
				(Q_0)_{12,x} &= 2 e^{2i ( (x-x_0) p_{\infty}^{(k)} + t ( q_{\infty}^{(k)} - 1 ) )} \left( F_{12,x}^{(1)} + 2i p_{\infty}^{(k)} F_{12}^{(1)} \right), \\
				(Q_0)_{21,x} &= 2 e^{-2i ( (x-x_0) p_{\infty}^{(k)} + t ( q_{\infty}^{(k)} - 1 ) )} \left( F_{21,x}^{(1)} - 2i p_{\infty}^{(k)} F_{21}^{(1)} \right).
			\end{aligned}
		\end{equation}
		Now, by \eqref{eq:p(k)}, we have
		\[
		\begin{aligned}
			(F_0)_tF_0^{-1}
			=&
			i(q_{\infty}^{(k)} - 1)\sigma_3
			+e^{i ((x - x_0 )p_{\infty}^{(k)}+t (q_{\infty}^{(k)}-1))\sigma_3} \left(
			F_t F^{-1}  - i q(k) F \sigma_3 F^{-1}\right)
			e^{-i ((x - x_0 )p_{\infty}^{(k)}+t (q_{\infty}^{(k)}-1))\sigma_3}  \\
			=&
			2ik^2\sigma_3 -i\sigma_3 
			+
			e^{i ((x - x_0 )p_{\infty}^{(k)}+t (q_{\infty}^{(k)}-1))\sigma_3} \\
			&\times\left(-2ik[\sigma_3,F^{(1)}]  
			+2i\left(
			[\sigma_3,F^{(1)}]F^{(1)}
			-[\sigma_3,F^{(2)}]
			\right)\right)
			e^{-i ((x - x_0 )p_{\infty}^{(k)}+t (q_{\infty}^{(k)}-1))\sigma_3}
			+\mathcal{O}(k^{-1}).
		\end{aligned}
		\]
		Again, the left-hand side is entire and grows at most quadratically, so the negative powers vanish. Hence
		\[
		\begin{aligned}
			(F_0)_tF_0^{-1}
			=&
			2ik^2\sigma_3 -i\sigma_3 
			+
			e^{i ((x - x_0 )p_{\infty}^{(k)}+t (q_{\infty}^{(k)}-1))\sigma_3} \\
			&\times\left(-2ik[\sigma_3,F^{(1)}]  
			+2i\left(
			[\sigma_3,F^{(1)}]F^{(1)}
			-[\sigma_3,F^{(2)}]
			\right)\right)
			e^{-i ((x - x_0 )p_{\infty}^{(k)}+t (q_{\infty}^{(k)}-1))\sigma_3}.
		\end{aligned}
		\]
		A direct calculation gives
		\begin{equation}\label{eq:F1F2}
			2[\sigma_3,F^{(1)}]F^{(1)}
			-2[\sigma_3,F^{(2)}]
			=
			2\begin{pmatrix}
				2 F^{(1)}_{12} F^{(1)}_{21} & 2 F^{(1)}_{12} F^{(1)}_{22}  - 2 F^{(2)}_{12}\\
				2 F^{(2)}_{21} -2 F^{(1)}_{21}  F^{(1)}_{11} &-2 F^{(1)}_{12} F^{(1)}_{21}
			\end{pmatrix}.
		\end{equation}
		Using \eqref{Y123y}, \eqref{eq:Q1221} and \eqref{eq:F1F2}, we obtain
		\[
		\begin{aligned}
			e^{i\left((x-x_0)p_{\infty}^{(k)}
				+t\left(q_{\infty}^{(k)}-1\right)\right)\sigma_3}
			2i\left(
			[\sigma_3,F^{(1)}]F^{(1)}
			-[\sigma_3,F^{(2)}]
			\right)
			e^{-i\left((x-x_0)p_{\infty}^{(k)}
				+t\left(q_{\infty}^{(k)}-1\right)\right)\sigma_3}
			=
			iQ_0^2\sigma_3+(Q_0)_x.
		\end{aligned}
		\]
		Consequently,
		\[
		i(F_0)_tF_0^{-1}
		=
		-2k^2\sigma_3
		+2k\sigma_3Q_0
		-(Q_0^2-I)\sigma_3
		+i(Q_0)_x.
		\]
		Since
		\[
		-2ik\mathcal L_0
		=
		-2ik\,i\sigma_3(Q_0-kI)
		=
		-2k^2\sigma_3+2k\sigma_3Q_0,
		\]
		we finally obtain
		\[
		i(F_0)_t
		=
		\left[
		-2ik\mathcal L_0
		-(Q_0^2-I)\sigma_3
		+i(Q_0)_x
		\right]F_0.
		\]
		This is precisely the time part of the dNLS Lax pair. Hence \(F_0\) solves the Lax pair with potential \(Q_0\). In particular,
		$u_0=2F^{(1)}_{21} e^{-2i ((x - x_0 )p_{\infty}^{(k)}+t (q_{\infty}^{(k)}-1))} = 2F^{(1)}_{21} e^{-2i ((x - x_0 )p_{\infty}+t q_{\infty})}$.
		For the derivation of \eqref{eq:modu0}, the reader is referred to \cite{Jenkins2015,WangYan2025}.
	\end{proof}

	\begin{rmk}
		Recall that $v=-\re(\eta_1+\eta_2)$ and define $ \varpi_0=2\left(q_{\infty}^{(k)}-1-vp_{\infty}^{(k)}\right)+\frac{v^2}{4} $. Then the potential $u_0$ in \eqref{eq:u0} can be rewritten exactly in the classical travelling-wave form
		\begin{equation}\label{eq:classical-trav}
			u_0^{\rm cl}(x,t)
			=
			e^{-i\varpi_0t}
			e^{-i\left(\frac{v}{2}x+\frac{v^2}{4}t\right)}
			\psi(x+vt-x_0),
		\end{equation}
		for a function $\psi$ which satisfies \eqref{eq:classical-profile-ode}. For an introduction to classical traveling wave solutions, readers are referred to Appendix~\ref{subsec:classical-elliptic}. Hence the potential reconstructed from the RHP \ref{RHp2} coincides with the classical travelling-wave after the corresponding parameters are identified. 
	\end{rmk}

	\begin{rmk}
		A natural question is why the RHP \ref{RH5} associated with the ``free'' solution is formulated with the inversion symmetry
		$
		Y(z)=z^{-1}Y(z^{-1})\sigma_1.
		$
		A reader familiar with \cite{Jenkins2016} may wonder what would happen if the corresponding ``free'' solution instead satisfied the symmetry with a negative sign. In fact, starting from $Y(z)$, we may introduce the transformation
		$
		\widetilde{Y}(z)=\sigma_3Y(z)\sigma_3.
		$
		Then the symmetry of $\widetilde{Y}$ becomes
		\[
		\widetilde{Y}(z)
		=
		-z^{-1}\widetilde{Y}(z^{-1})\sigma_1.
		\]
		Therefore $\widetilde{Y}$ satisfies the jump relation
		$
		\widetilde{Y}_+(z)
		=
		\widetilde{Y}_-(z)\widetilde{V}(z)$,
		$
		\widetilde{V}(z)=\sigma_3V^Y(z)\sigma_3,
		$ where $V^Y(z)$ is the jump matrices of RHP \ref{RH5}
		Thus, the sign in the symmetry cannot be changed independently of the jump matrices and the normalization at the origin.
		\par 
		Indeed, if the background eigenfunction is defined by
		\[
		\Psi_0(x;z)
		=
		e^{i(x-x_0)p_\infty\sigma_3}
		Y(z;x,0)
		e^{-i(x-x_0)p(z)\sigma_3},
		\]
		then the eigenfunction constructed from $\widetilde{Y}$ satisfies
		$
		\widetilde{\Psi}_0(x;z)
		=
		\sigma_3\Psi_0(x;z)\sigma_3.
		$
		If $\Psi_0$ solves
		$
		\partial_x\Psi_0
		=
		i\sigma_3\bigl(Q_0-k(z)I\bigr)\Psi_0,
		$
		then $\widetilde{\Psi}_0$ solves the same type of Lax equation with
		\[
		\widetilde{Q}_0
		=
		\sigma_3Q_0\sigma_3 = -Q_0 = \begin{pmatrix}
			0 & -\overline{u_0}\\
			-u_0 & 0
		\end{pmatrix}.
		\]
		Therefore, under this particular transformation, the negative inversion symmetry corresponds to a global phase shift of $\pi$ in the background solution. This is also consistent with the backgrounds $+1$ and $-1$ considered in \cite{Jenkins2016}.
	\end{rmk}

	\section{From exact periodic steps to asymptotically periodic data}\label{Secfast}
	So far, we have established a connection between RHP \ref{RHp2} on the $k$-plane and the Lax pair \eqref{Laxpair}. We now return to the $z$-plane for $t = 0$ and use the solution of RHP \ref{RH5} to relate the problem to asymptotically step-like initial backgrounds. We then analyze the associated Jost solutions and carry out the standard direct scattering procedure.
	\par 
	Recall that the spacial part of the Lax pair corresponding to parameter $z$ is
	\begin{equation}\label{eq:direct_z_lax}
		\Phi_x=i\sigma_3\left(Q(x)- \frac{ z + z^{-1} }{2}  I \right)\Phi,
	\end{equation}
	where
	$
	Q(x)=
	\begin{pmatrix}
		0 & \overline{u(x)} \\
		u(x) & 0
	\end{pmatrix}.
	$
	We have already obtained the solution $F(k)$ of RHP~\ref{RHp2} in the  $k$-plane. Therefore, by applying the transformation $ \sqrt{1 - z(k)^{-2}} F(k(z)) =  Y(z)  $, we obtain the solution $Y(z)$ of RHP~\ref{RH5} in the $z$-plane.
	\par 
	In Subsection~\ref{subrpf}, we considered only a genus-one RHP associated  with a single periodic background, since the only spectral bands chosen  on the unit circle in the $z$-plane were
	$
	\Sigma_1\cup \Sigma_2.
	$
	In what follows, we consider two sets of spectral bands,
	$
	\Sigma^{\ell}
	$
	and
	$
	\Sigma^r,
	$
	as illustrated in Figure~\ref{fig:six-cases}. These two sets correspond,  respectively, to the periodic-wave backgrounds approached by the initial  data as $x\to-\infty$ and $x\to+\infty$. Accordingly, in the remainder  of the paper, we will attach the superscripts $\ell$ and $r$ to certain  parameters and functions whenever it is necessary to distinguish whether  they are associated with the left or right asymptotic background.
	\par 
	For $s\in\{\ell,r\}$, let $u_0^s(x)$ be the background solution as $x\to\infty^s$, where $ \infty^{\ell} = -\infty $ and $ \infty^{r} = + \infty $.  Define
	\begin{equation*}
		Q_0^s(x)=
		\begin{pmatrix}
			0 & \overline{u_0^s(x)}\\
			u_0^s(x) & 0
		\end{pmatrix},
		\qquad
		\Delta Q^s(x)=Q(x)-Q_0^s(x).
	\end{equation*}
	\par 
	For $t = 0 $, the background fundamental matrix is
	\begin{equation}\label{eq:direct_z_background}
		\Psi_0^s(x;z)
		= e^{i  (x - x_0^s )p_{\infty}^s\sigma_3}
		Y^s(z;x,0)
		e^{-i(x-x_0^s)p^s(z)\sigma_3}.
	\end{equation}
	By construction, $\Psi_0^s(x;z)$ solves \eqref{eq:direct_z_lax} with
	$Q(x)$ replaced by $Q_0^s(x)$.
	\par 
	We define the Jost solution $\Psi^s(x;z)$ as the solution of \eqref{eq:direct_z_lax} satisfying
	\begin{equation*}
		\Psi^s(x;z)
		= e^{i  (x - x_0^s )p_{\infty}^s\sigma_3}
		Y^s(z;x,0)\left(I+o(1)\right)
		e^{-i(x-x_0^s)p^s(z)\sigma_3},
		\qquad x\to\infty^s .
	\end{equation*}
	Equivalently, we introduce $M^s(x;z)$ by
	\begin{equation}\label{eq:direct_z_M_def}
		\Psi^s(x;z)
		= e^{i  (x - x_0^s )p_{\infty}^s\sigma_3}
		Y^s(z;x,0)M^s(x;z)
		e^{-i(x-x_0^s)p^s(z)\sigma_3}.
	\end{equation}
	Then the normalization condition is
	\begin{equation}\label{eq:direct_z_M_norm}
		M^s(x;z)\to I,
		\qquad x\to\infty^s .
	\end{equation}
	
	\begin{lem}\label{lem:direct_z_M_equation}
		Suppose that $\Psi^s(x;z)$ is a solution of \eqref{eq:direct_z_lax}, and let $M^s(x;z)$ be defined by \eqref{eq:direct_z_M_def}. Then $M^s(x;z)$ satisfies
		\begin{equation}\label{eq:direct_z_M_ODE}
			\partial_x M^s
			={}
			-i p^s(z)[\sigma_3,M^s]
			+
			Y^s(z;x,0)^{-1}
			e^{-i(x-x_0^s)p_\infty^s\sigma_3}
			i\sigma_3\Delta Q^s(x)
			e^{i(x-x_0^s)p_\infty^s\sigma_3}
			Y^s(z;x,0)M^s .
		\end{equation}
		Together with the normalization condition \eqref{eq:direct_z_M_norm}, this differential equation is equivalent to the Volterra integral equation
		\begin{align}
			M^s(x;z)
			={}&I+
			\int_{\infty^s}^{x}
			e^{-i(x-y)p^s(z)\sigma_3}
			Y^s(z;y,0)^{-1}
			e^{-i(y-x_0^s)p_\infty^s\sigma_3}
			\nonumber\\
			&\qquad\times
			i\sigma_3\Delta Q^s(y)
			e^{i(y-x_0^s)p_\infty^s\sigma_3}
			Y^s(z;y,0)M^s(y;z)
			e^{i(x-y)p^s(z)\sigma_3}
			\,dy .
			\label{eq:direct_z_Volterra}
		\end{align}
	\end{lem}
	
	\begin{proof}
		For brevity, set
		\begin{equation}\label{def:H}
			H^s(x;z)
			=
			e^{i(x-x_0^s)p_\infty^s\sigma_3}
			Y^s(z;x,0).
		\end{equation}
		Then the background solution and the Jost solution can be written as
		\[
		\Psi_0^s(x;z)
		=
		H^s(x;z)
		e^{-i(x-x_0^s)p^s(z)\sigma_3},
		\qquad
		\Psi^s(x;z)
		=
		H^s(x;z)M^s(x;z)
		e^{-i(x-x_0^s)p^s(z)\sigma_3}.
		\]
		Since $\Psi_0^s$ solves the background equation, we have
		$
		\partial_x\Psi_0^s
		=
		i\sigma_3
		\bigl(Q_0^s(x)-k(z)I\bigr)\Psi_0^s.
		$
		Differentiating the expression for $\Psi_0^s$ gives
		\[
		\left(
		\partial_x H^s
		-i p^s(z)H^s\sigma_3
		\right)
		e^{-i(x-x_0^s)p^s(z)\sigma_3}
		=
		i\sigma_3
		\bigl(Q_0^s(x)-k(z)I\bigr)
		H^s
		e^{-i(x-x_0^s)p^s(z)\sigma_3}.
		\]
		Consequently,
		\begin{equation}\label{eq:Hx_background}
			\partial_x H^s
			=
			i\sigma_3
			\bigl(Q_0^s(x)-k(z)I\bigr)H^s
			+
			i p^s(z)H^s\sigma_3.
		\end{equation}
		
		On the other hand, differentiating
		$
		\Psi^s
		=
		H^sM^s
		e^{-i(x-x_0^s)p^s(z)\sigma_3}
		$
		and using \eqref{eq:Hx_background}, we obtain
		\begin{equation}\label{eq:PQS}
			\partial_x\Psi^s
			={}
			i\sigma_3
			\bigl(Q_0^s(x)-k(z)I\bigr)\Psi^s+
			H^s
			\left(
			\partial_xM^s
			+i p^s(z)[\sigma_3,M^s]
			\right)
			e^{-i(x-x_0^s)p^s(z)\sigma_3}.
		\end{equation}
		Since $\Psi^s$ also satisfies
		$
		\partial_x\Psi^s
		=
		i\sigma_3
		\bigl(Q(x)-k(z)I\bigr)\Psi^s$,
		with
		$Q(x)=Q_0^s(x)+\Delta Q^s(x),
		$
		comparing this and \eqref{eq:PQS} yields
		\[
		H^s
		\left(
		\partial_xM^s
		+i p^s(z)[\sigma_3,M^s]
		\right)
		=
		i\sigma_3\Delta Q^s(x)H^sM^s,
		\]
		which is
		\begin{equation}\label{eq:paer}
			\partial_xM^s
			=
			-i p^s(z)[\sigma_3,M^s]
			+
			(H^s)^{-1}
			i\sigma_3\Delta Q^s(x)
			H^sM^s.
		\end{equation}
		Since
		$
		(H^s)^{-1}
		=
		Y^s(z;x,0)^{-1}
		e^{-i(x-x_0^s)p_\infty^s\sigma_3},
		$
		we see that \eqref{eq:paer} is precisely \eqref{eq:direct_z_M_ODE}.
		
		To derive the integral equation, conjugate \eqref{eq:direct_z_M_ODE} by $e^{ixp^s(z)\sigma_3}$. We obtain
		\begin{align*}
			\frac{d}{dx}
			\left(
			e^{ixp^s(z)\sigma_3}
			M^s(x;z)
			e^{-ixp^s(z)\sigma_3}
			\right)
			={}&
			e^{ixp^s(z)\sigma_3}
			Y^s(z;x,0)^{-1}
			e^{-i(x-x_0^s)p_\infty^s\sigma_3}
			\\
			&\times
			i\sigma_3\Delta Q^s(x)
			e^{i(x-x_0^s)p_\infty^s\sigma_3}
			Y^s(z;x,0)M^s(x;z)
			e^{-ixp^s(z)\sigma_3}.
		\end{align*}
		Integrating from $\infty^s$ to $x$ and using $M^s(x;z)\to I$ as $x\to\infty^s$ gives \eqref{eq:direct_z_Volterra}.
	\end{proof}
	
	\begin{rmk}
		The formula \eqref{eq:direct_z_Volterra} separates the known background from the perturbation. The background at $x\to\infty^s$ is contained in $Y^s(z;x,0)$ and $p^s(z)$, whereas the perturbation appears only through $\Delta Q^s(x)$. In particular, if $\Delta Q^s\equiv0$, then
		\eqref{eq:direct_z_Volterra} gives
		$
		M^s(x;z)\equiv I.
		$
		Thus in the unperturbed case,
		$
		\Psi^s(x;z)=\Psi_0^s(x;z).
		$
	\end{rmk}

	\begin{pro}\label{pro:M-large-z}
		Suppose that
		$
		\Delta Q^\ell\in L^1(\mathbb{R}^-)$,
		$\Delta Q^r\in L^1(\mathbb{R}^+).
		$
		Then, for every fixed $x\in\mathbb{R}$ and
		$s\in\{\ell,r\}$,
		\[
		M^s(x;z)=I+o(1),
		\qquad z\to\infty,
		\]
		where the limit is understood columnwise in the corresponding domains of analyticity.
		
		Moreover, $M^s$ satisfies the symmetry
		\begin{equation}\label{eq:symMs}
			M^s(x;z)
			=
			\sigma_1M^s(x;z^{-1})\sigma_1,
		\end{equation}
		for all $z$ for which both sides are defined. 
		\par 
		Consequently,
		\[
		M^s(x;z)=I+o(1),
		\qquad z\to0,
		\]
		again columnwise in the corresponding domains.
	\end{pro}
	\par 
	The proof of Proposition \ref{pro:M-large-z} is deferred to Appendix~\ref{app:proMl}. The following proposition summarizes the basic properties of the Jost solutions that will be used in the subsequent scattering analysis.
	\begin{pro}\label{pro3.3}
		Suppose $u(x)-u_0^\ell(x)\in L^1(\mathbb{R}^-)$ and
		$u(x)-u_0^r(x)\in L^1(\mathbb{R}^+)$. Then
		$\Psi^s(x;z)$, $s\in\{\ell,r\}$, have the following properties: 
		
		\begin{enumerate}
			\item \label{prop:pro3.3-property-1}
			For any $x_*\in\mathbb{R}$, the Volterra equations determine uniquely the left and right Jost solutions satisfying
			\[
			\Psi^\ell(\cdot;z)\in
			L^\infty((-\infty,x_*]),
			\qquad
			\Psi^r(\cdot;z)\in
			L^\infty([x_*,\infty)),
			\]
			for $z\in\mathbb{R}\setminus\{0\}$ and for either boundary value on the interior of the corresponding spectral bands.
			\par 
			For each fixed $x\in\mathbb{R}$, the analytic continuation of the Jost solutions is understood columnwise. More precisely,
			\begin{equation*}
				\begin{aligned}
					&\Psi_1^\ell(x;z) \text{ is analytic in }
					\mathbb{C}^+\setminus\Sigma_1^\ell,
					\qquad
					\Psi_2^r(x;z) \text{ is analytic in }
					\mathbb{C}^+\setminus\Sigma_1^r, \\
					&\Psi_2^\ell(x;z) \text{ is analytic in }
					\mathbb{C}^-\setminus\Sigma_2^\ell,
					\qquad
					\Psi_1^r(x;z) \text{ is analytic in }
					\mathbb{C}^-\setminus\Sigma_2^r.
				\end{aligned}
			\end{equation*}
			These analytic columns possess continuous boundary values on the corresponding contours away from $z=0$ and the endpoints of the spectral bands.
			\item \label{prop:pro3.3-property-2}
			For $s\in\{\ell,r\}$, the boundary values of
			$\Psi^s(x;z)$ on the spectral bands satisfy
			\begin{equation}\label{eq:Psijump}
				\Psi_+^s(x;z)
				=
				\Psi_-^s(x;z)
				\begin{cases}
					-i\sigma_1,
					& z\in\Sigma_1^s,
					\\[1ex]
					i\sigma_1,
					& z\in\Sigma_2^s.
				\end{cases}
			\end{equation}
			\item \label{prop:pro3.3-property-3}
			Let $n\geq2$ and assume that
			\[
			u-u_0^\ell\in W^{n,1}(\mathbb{R}^-),
			\qquad
			u-u_0^r\in W^{n,1}(\mathbb{R}^+).
			\]
			Then, for each fixed $x\in\mathbb{R}$, the Jost solutions admit asymptotic expansions at both $z=\infty$ and $z=0$.
			
			As $z\to\infty$,
			\begin{equation*}
				\Psi^s(x;z)
				e^{i(x-x_0^s)k(z)\sigma_3}
				=
				I+
				\sum_{j=1}^{n-1}
				\frac{C_j^s(x)}{z^j}
				+
				\mathcal{O}(z^{-n}),
				\qquad s\in\{\ell,r\},
			\end{equation*}
			where the expansion is understood columnwise in the corresponding domains of analyticity. In particular,
			\begin{equation*}
				C_1^s(x)
				=
				\begin{pmatrix}
					A^s(x) & \overline{u(x)}
					\\
					u(x) & -A^s(x)
				\end{pmatrix},
				\qquad
				\frac{d}{dx}A^s(x)
				=
				i|u(x)|^2.
			\end{equation*}
			
			The behavior at the origin follows from the symmetry
			\begin{equation}\label{eq:invsym}
				\Psi^s(x;z)
				=
				z^{-1}\Psi^s(x;z^{-1})\sigma_1.
			\end{equation}
			Since $k(z^{-1})=k(z)$, one obtains, as $z\to0$,
			\begin{equation*}
				\Psi^s(x;z)
				e^{-i(x-x_0^s)k(z)\sigma_3}
				=
				\frac{\sigma_1}{z}
				+
				\sum_{j=1}^{n-1}
				z^{j-1}C_j^s(x)\sigma_1
				+
				\mathcal{O}(z^{n-1}).
			\end{equation*}
			In particular,
			\begin{equation*}
				\Psi^s(x;z)
				e^{-i(x-x_0^s)k(z)\sigma_3}
				=
				\frac{\sigma_1}{z}
				+
				\begin{pmatrix}
					\overline{u(x)} & A^s(x)
					\\
					-A^s(x) & u(x)
				\end{pmatrix}
				+
				\mathcal{O}(z),
				\qquad z\to0.
			\end{equation*}
			
			\item \label{prop:pro3.3-property-4}
			Suppose, in addition, that
			\[
			u-u_0^\ell\in L^{1,1}(\mathbb{R}^-),
			\qquad
			u-u_0^r\in L^{1,1}(\mathbb{R}^+).
			\]
			Fix $x_*\in\mathbb{R}$. Then there exists a constant $\mathcal{C}>0$, independent of $z$, such that,
			\begin{equation*}
				\begin{aligned}
					\|\Psi^\ell(x;z)\|
					&\leq
					\sup_{y<x_*}\|\Psi_0^\ell(y;z)\|
					\exp\left\{
					\mathcal{C}(1+|x_*|)
					\|u-u_0^\ell\|_
					{L^{1,1}((-\infty,x_*])}
					\right\}, \\
					\|\Psi^r(x;z)\|
					&\leq
					\sup_{y>x_*}\|\Psi_0^r(y;z)\|
					\exp\left\{
					\mathcal{C}(1+|x_*|)
					\|u-u_0^r\|_
					{L^{1,1}([x_*,\infty))}
					\right\}.
				\end{aligned}
			\end{equation*}
			Consequently, the perturbation does not increase the order of the endpoint singularities inherited from the periodic backgrounds. More precisely, for
			$\zeta\in
			\{
			\eta_1^s,\eta_2^s,
			\overline{\eta_1^s},\overline{\eta_2^s}
			\}$,
			$s\in\{\ell,r\}$, 
			one has
			\begin{equation}\label{eq:Psisin}
				\Psi^s(x;z)
				=
				\mathcal{O}\left(
				|z-\zeta|^{-1/4}
				\right),
				\qquad z\to\zeta.
			\end{equation}
			In particular, the Jost solutions have at most fourth-root singularities at the endpoints of the corresponding spectral bands. Moreover, if we assume that
			\begin{equation*}
				u-u_0^\ell\in L^{1,2}(\mathbb{R}^-), \qquad  u-u_0^r\in L^{1,2}(\mathbb{R}^+),
			\end{equation*}
			then, after the singular scalar factors inherited from the periodic backgrounds are removed, the boundary values of the Jost solutions admit finite limits at the endpoints:
			\begin{equation}\label{eq:hhpsi}
				\begin{aligned}
					h_\pm^s(k(z))^{-1}\Psi_\pm^s(x;z)&=\widetilde{\Psi}^s(x;\zeta)+\mathcal O\left(|z-\zeta|^{1/2}\right),\qquad \zeta\in\{\eta_1^s,\overline{\eta_1^s}\},\\
					h_\pm^s(k(z))\Psi_\pm^s(x;z)&=\widetilde{\Psi}^s(x;\zeta)+\mathcal O\left(|z-\zeta|^{1/2}\right),\qquad \zeta\in\{\eta_2^s,\overline{\eta_2^s}\},
				\end{aligned}
			\end{equation}
			where $ \widetilde{\Psi}^s(x;z) $ is defined by \eqref{eq:regularized-Psi-definitions}. For the columns admitting analytic continuation, the same estimates hold in the corresponding slit neighborhoods of the endpoints. The function $h^s(k(z))$ denotes the fourth-root function $h(k)$ introduced in Subsection \ref{subrpf}, with $\eta_j$ replaced by $\eta_j^s$; equivalently, its explicit $z$-plane expression is given in \eqref{eq:h-k}.
		\end{enumerate}
	\end{pro}
	The proof of the Proposition \ref{pro3.3} is deferred to the Appendix \ref{AppA}.
	\par 
	So far, we have established various properties of the Jost functions. We now proceed to derive the scattering matrix.
	\par 
	Define the scattering matrix $S(z)$ by
	\begin{equation}\label{eq:scat}
		\Psi^\ell(x;z)=\Psi^r(x;z)S(z), \quad z\in\mathbb{R},
		\qquad 
		\Psi_\pm^\ell(x;z)=\Psi^r_\pm(x;z)S_\pm(z), \quad z\in \Sigma^r \cap \Sigma^\ell,
	\end{equation}
	where we recall that $\Sigma^r:=\Sigma^r_1 \cup \Sigma^r_2$ and $ \Sigma^{\ell}:= \Sigma^\ell_1 \cup \Sigma^\ell_2 $.
	\par 
	Recall \eqref{eq:invsym} and the symmetry of $Y(z)$ in RHP \ref{RH5}, we have
	\begin{equation}\label{eq:sym}
		\Psi^s (z) = \sigma_1  \Psi^{s*} ( z )  \sigma_1 = z^{-1} \Psi^s(z^{-1}) \sigma_1.
	\end{equation}
	\par 
	Therefore $S(z)$ satisfies
	\begin{equation}\label{Syms}
		S(z) = \sigma_1  S^* ( z )  \sigma_1 = \sigma_1  S ( z^{-1} )  \sigma_1.
	\end{equation}
	By using the symmetry of $S(z)$, one can write
	\begin{equation*}
		S(z) = \begin{pmatrix} a(z) & b^*(z) \\ b(z) & a^*(z) \end{pmatrix}, \quad  z \in \mathbb{R},
		\qquad
		S_{\pm}(z) = \begin{pmatrix} a_{\pm}(z) & b_{2\pm}(z) \\ b_{1\pm}(z) & a^*_{\pm}(z) \end{pmatrix}, \qquad z \in \Sigma^r_1 \cap \Sigma^\ell_1.
	\end{equation*}
	More specifically, we can obtain
	\begin{equation*}
		\begin{aligned}
			\Psi_1^{\ell}(z) &= a(z)\Psi_1^r(z) + b_1(z)\Psi_2^r(z), \qquad z \in \Sigma_1^r \setminus \Sigma_1^{\ell}, \\
			\Psi_2^{\ell}(z) &= b_1^*(z)\Psi_1^r(z) + a^*(z)\Psi_2^r(z), \qquad z \in \Sigma_2^r \setminus \Sigma_2^{\ell}, \\
			\Psi_2^r(z) &= -b_2(z)\Psi_1^{\ell}(z) + a(z)\Psi_2^{\ell}(z), \qquad z \in \Sigma_1^{\ell} \setminus \Sigma_1^r, \\
			\Psi_1^r(z) &= a^*(z)\Psi_1^{\ell}(z) - b_2^*(z)\Psi_2^{\ell}(z), \qquad z \in \Sigma_2^{\ell} \setminus \Sigma_2^r.
		\end{aligned}
	\end{equation*}
	Thus
	\begin{equation}\label{eq:defab}
		\begin{aligned}
			a(z) &= \frac{\det [\Psi_1^{\ell}(x; z), \Psi_2^r(x; z)]}{ 1 - z^{-2} }, \quad z \in \mathbb{C}^+\setminus ( \Sigma_1^{\ell} \cup \Sigma_1^r ),
			& b(z) &= \frac{\det [\Psi_{1}^r(x; z), \Psi_{1}^{\ell}(x; z)]}{ 1 - z^{-2} }, \quad z \in \mathbb{R}\setminus \{0\}, \\[1ex]
			b_{1\pm}(z) &= \frac{\det [\Psi_{1\pm}^r(x; z), \Psi_{1\pm}^{\ell}(x; z)]}{ 1 - z^{-2} }, \quad z\in \Sigma_1^r , 
			&  b_{2\pm}(z) &= -\frac{\det [\Psi_{2\pm}^r(x; z), \Psi_{2\pm}^{\ell}(x; z)]}{ 1 - z^{-2} }, \quad z\in \Sigma_1^{\ell} ,
		\end{aligned}
	\end{equation}
	where the function $a(z)$ admits a continuous extension to the boundary.
	\begin{rmk}\label{b1b2}
		Here, for clarity, we explicitly write the subscripts ``$\pm$'' to remind the reader that these definitions are given in terms of boundary values. The function $a(z)$ also possesses its own boundary values $a_{\pm}(z)$ on
		$
		\Sigma_1^{\ell}\cup\Sigma_1^r.
		$
		and the boundary value $a_+(z)$ on $\mathbb{R}$.
		In what follows, for simplicity, we sometimes omit the subscripts ``$\pm$'' indicating the boundary values.
		\par 
		Although, according to the definitions, $b_1(z)$ could also be defined on 
		$\Sigma_2^{\ell}$ and $b_2(z)$ could likewise be defined on 
		$\Sigma_2^{r}$, so that, formally, one would have $b_2^*(z)=b_1(z)$, and the notation $b_2$ might therefore appear redundant, we shall still distinguish between these two quantities. This distinction emphasizes that $b_1$ contains only the scattering information generated by the right background wave, whereas $b_2$ contains only that generated by the left background wave. It also corresponds naturally to the subsequent factorization of $a(z)$ into $a_1(z)$ and $a_2(z)$.
	\end{rmk}
	\par

	\begin{pro}\label{prop:scattering-endpoint-behavior}
		Assume that
		$
		u(x)-u_0^\ell(x)\in L^{1,2}(\mathbb R^-)$ and 		
		$u(x)-u_0^r(x)\in L^{1,2}(\mathbb R^+).
		$
		Then the scattering coefficients have the following properties.
		\begin{enumerate}
			\item At the upper-half-plane endpoints,
			\begin{equation}\label{eq:a-endpoint-behavior}
				a(z)=\mathcal O\left(|z-\eta_j^s|^{-1/4}\right),
				\qquad
				z\to\eta_j^s,
				\qquad
				s\in\{\ell,r\},
				\quad j=1,2.
			\end{equation}
			Moreover, regardless of whether $\eta_j^\ell\in\Sigma_1^r$ and whether $\eta_j^r\in\Sigma_1^\ell$, one has
			\begin{equation}\label{eq:b12-endpoint-behavior}
				b_{1\pm}(z)
				=\mathcal O\left(|z-\eta_j^r|^{-1/4}\right),
				z\to\eta_j^r, \qquad
				b_{2\pm}(z)
				=\mathcal O\left(|z-\eta_j^\ell|^{-1/4}\right),
				z\to\eta_j^\ell,
				\qquad j=1,2.
			\end{equation}
			
			\item The scattering coefficients satisfy the symmetries
			\begin{equation}\label{eq:absym}
				a(z)=a^*(z^{-1}),
				\qquad
				b(z)=b^*(z^{-1}).
			\end{equation}
			In particular, their boundary values satisfy
			\begin{equation*}
				a_\pm(z)=\overline{a_\mp(z)},
				\quad z\in\Sigma_1^\ell\cup\Sigma_1^r, \qquad 
				b_{1\pm}(z)=\overline{b_{1\mp}(z)},
				\quad z\in\Sigma_1^r, \qquad 
				b_{2\pm}(z)=\overline{b_{2\mp}(z)},
				\quad z\in\Sigma_1^\ell.
			\end{equation*}
			
			\item If, in addition,
			$u(x)-u_0^\ell(x)\in W^{1,1}(\mathbb R^-)$
			and
			$u(x)-u_0^r(x)\in W^{1,1}(\mathbb R^+)$,
			then
			\begin{equation}\label{eq:a-infinity-zero-asymptotics}
				a(z)e^{-\frac{i}{2}(x_0^\ell-x_0^r)(z-z^{-1})}
				=1+\mathcal O(z^{-1}),
				\quad z\to\infty, \qquad 
				a(z)e^{-\frac{i}{2}(x_0^\ell-x_0^r)(z-z^{-1})}
				=1+\mathcal O(z),
				\quad z\to0.
			\end{equation}
			If, furthermore,
			$u(x) \in W^{4,1}_{\rm loc}(\mathbb{R})$,
			$u(x)-u_0^\ell(x)\in W^{4,1}(\mathbb R^-)$
			and
			$u(x)-u_0^r(x)\in W^{4,1}(\mathbb R^+)$,
			then for $z\in\mathbb R$,
			\begin{equation}\label{eq:b-infinity-zero-asymptotics}
				b(z)e^{-\frac{i}{2}(x_0^\ell-x_0^r)(z-z^{-1})}
				=\mathcal O(z^{-4}),
				\quad z\to\pm\infty, \qquad 
				b(z)e^{-\frac{i}{2}(x_0^\ell-x_0^r)(z-z^{-1})}
				=\mathcal O(z^4),
				\quad z\to0.
			\end{equation}
		\end{enumerate}

	\end{pro}
	
	\begin{proof}
		The endpoint behavior of $a(z)$ \eqref{eq:a-endpoint-behavior} and that of $b_1(z)$ and $b_2(z)$ \eqref{eq:b12-endpoint-behavior} follows directly from \eqref{eq:Psisin} and the definitions in \eqref{eq:defab}.
		\par 
		By the definitions in \eqref{eq:defab} and the symmetries in \eqref{eq:sym}, we obtain the stated symmetry relations \eqref{eq:absym}.

		When $z$ lies on a spectral band, its the boundary values satisfy
		\begin{equation*}
			a_{\pm }(z) =
			\lim_{s \to z_{\pm}} a(s)
			=
			\lim_{s \to z_{\pm}} \overline{a(\bar{s}^{-1})}
			=
			\lim_{\bar{s}^{-1} \to z_{\mp}} \overline{a(\bar{s}^{-1})}
			= \overline{a_{\mp}(z)}, \quad z\in \Sigma_1^{\ell} \cup \Sigma_1^r, 
		\end{equation*}
		and for $z\in\Sigma_1^r$, using the definition of $b_{1\pm}$ and symmetry \eqref{eq:sym}, we obtain
		\begin{equation*}
			\begin{aligned}
				b_{1\pm}(z)
				=&\frac{\det\left[\Psi_{1\pm}^r(x;z),\Psi_{1\pm}^{\ell}(x;z)\right]}{1-z^{-2}}
				=\frac{\det\left[z^{-1}\sigma_1\overline{\Psi_{1\mp}^r(x;z)},z^{-1}\sigma_1\overline{\Psi_{1\mp}^{\ell}(x;z)}\right]}{1-z^{-2}} \\
				=& \overline{\frac{\det\left[\Psi_{1\mp}^r(x;z),\Psi_{1\mp}^{\ell}(x;z)\right]}{1-z^{-2}}}=\overline{b_{1\mp}(z)}. 
			\end{aligned}
		\end{equation*}
		The case of $\Sigma_1^{\ell}$ can be obtained similarly.
		\par 
		Finally, we prove the asymptotic formulas. The first-order large-$z$ expansion obtained from the Volterra equations gives
		\begin{equation*}
			\Psi^s(x;z)e^{i(x-x_0^s)k(z)\sigma_3}
			=I+\mathcal{O}(z^{-1}),
			\qquad z\to\infty.
		\end{equation*}
		Substitution into the determinant formula for $a(z)$ yields
		\begin{equation*}
			\begin{aligned}
				a(z)
				=e^{i(x_0^\ell-x_0^r)k(z)}
				\frac{\det\left[e_1+\mathcal{O}(z^{-1}),e_2+\mathcal{O}(z^{-1})\right]}{1-z^{-2}}
				=e^{i(x_0^\ell-x_0^r)k(z)}
				\left(1+\mathcal{O}(z^{-1})\right),
				\qquad z\to\infty.
			\end{aligned}
		\end{equation*}
		This proof \eqref{eq:a-infinity-zero-asymptotics} at $z=\infty$. The behavior $z=0$ follows from the symmetry \eqref{eq:absym}.
		\par 
		The decay estimate \eqref{eq:b-infinity-zero-asymptotics} for $b(z)$ requires the cancellation of the first three terms in the large-$z$ expansion of the determinant defining $b(z)$. This calculation is carried out separately in Appendix \ref{app:b-large-z}.
	\end{proof}
	\par 
	Next, we will derive the transformation relations among the scattering coefficients.
	\begin{pro}\label{pro3.6}
		Suppose that
		$	u(x)-u_0^\ell(x)\in L^1(\mathbb{R}^-)$,
		and 
		$u(x)-u_0^r(x)\in L^1(\mathbb{R}^+)$.
		Then the scattering coefficients satisfy the following relations.
		
		\begin{enumerate}
			
			\item Away from the common spectral bands, one has
			\begin{align}
				a_\pm(z)
				&=
				\mp i b_{2\mp}(z), \quad 
				z\in\Sigma_1^\ell\setminus\Sigma_1^r,
				\quad
				&&a_\pm(z)
				=
				\pm i b_{1\mp}(z),  \quad 
				z\in\Sigma_1^r\setminus\Sigma_1^\ell,
				\label{eq:a-b2-transformation}
				\\
				a_\pm^*(z)
				&=
				\pm i b_{2\mp}^*(z), \quad 
				z\in\Sigma_2^\ell\setminus\Sigma_2^r,
				\quad 
				&&a_\pm^*(z)
				=
				\mp i b_{1\mp}^*(z), \quad 
				z\in\Sigma_2^r\setminus\Sigma_2^\ell.
				\label{eq:astar-b2-transformation}
			\end{align}
			Consequently,
			\begin{equation}\label{eq:abrel}
				\frac{a_+(z)}{a_-(z)}
				=
				\begin{cases}
					-\dfrac{b_{2-}(z)}{b_{2+}(z)},
					&
					z\in\Sigma_1^\ell\setminus\Sigma_1^r,
					\\[2ex]
					-\dfrac{b_{1-}(z)}{b_{1+}(z)},
					&
					z\in\Sigma_1^r\setminus\Sigma_1^\ell,
				\end{cases}
				\qquad \qquad 
				\frac{a_+^*(z)}{a_-^*(z)}
				=
				\begin{cases}
					-\dfrac{b_{2-}^*(z)}{b_{2+}^*(z)},
					&
					z\in\Sigma_2^\ell\setminus\Sigma_2^r,
					\\[2ex]
					-\dfrac{b_{1-}^*(z)}{b_{1+}^*(z)},
					&
					z\in\Sigma_2^r\setminus\Sigma_2^\ell.
				\end{cases}
			\end{equation}
			
			\item For
			$z\in\Sigma_1^r\cap\Sigma_1^\ell$,
			\begin{equation}\label{b2-b1-}
				\begin{pmatrix}
					a_+(z) & b_{2+}(z)
					\\
					b_{1+}(z) & a_+^*(z)
				\end{pmatrix}
				=
				\begin{pmatrix}
					a_-^*(z) & b_{1-}(z)
					\\
					b_{2-}(z) & a_-(z)
				\end{pmatrix}.
			\end{equation}
			The corresponding relation on
			$\Sigma_2^r\cap\Sigma_2^\ell$
			follows from the symmetry \eqref{Syms}.
		\end{enumerate}
	\end{pro}
	
	\begin{proof}
		We first consider the parts of the spectral bands which do not overlap. For $z\in\Sigma_1^r\setminus\Sigma_1^\ell$, only the right Jost solution has a jump. Thus the jump relations of the Jost solutions \eqref{eq:Psijump} give
		\begin{align*}
			a_{\pm}(z)
			=
			\frac{
				\det[
				\Psi_1^\ell(x;z),
				\Psi_{2\pm}^r(x;z)
				]
			}{
				1-z^{-2}
			}
			=
			\mp i
			\frac{
				\det[
				\Psi_1^\ell(x;z),
				\Psi_{1\mp}^r(x;z)
				]
			}{
				1-z^{-2}
			}
			=
			\pm i b_{1\mp}(z).
		\end{align*}
		The remaining cases in \eqref{eq:a-b2-transformation} and \eqref{eq:astar-b2-transformation} can be obtained similarly. The quotient identities follow immediately by dividing the corresponding boundary-value relations.
		\par 
		We next consider
		$z\in\Sigma_1^r\cap\Sigma_1^\ell$.
		Using the scattering relation and the jump relations of the
		Jost solutions, we obtain
		\begin{align*}
			\Psi_+^\ell(x;z)
			=
			\Psi_-^\ell(x;z)(-i\sigma_1)
			=
			\Psi_-^r(x;z)S_-(z)(-i\sigma_1)
			=
			\Psi_+^r(x;z)
			(i\sigma_1)S_-(z)(-i\sigma_1)
			=
			\Psi_+^r(x;z)
			\sigma_1S_-(z)\sigma_1.
		\end{align*}
		On the other hand, $\Psi_+^\ell(x;z)
		=
		\Psi_+^r(x;z)S_+(z)$. Hence $S_+(z)=\sigma_1S_-(z)\sigma_1$, which gives \eqref{b2-b1-}.
	\end{proof}
	\par 
	Define the $2\times2$ matrix-valued function
	\begin{equation}\label{eq:def-M-direct}
		\Phi(x;z)
		=
		\begin{cases}
			\displaystyle
			\left(
			\frac{\Psi_1^\ell(x;z)}{a(z)},
			\Psi_2^r(x;z)
			\right)
			\exp\left[
			\frac{i}{2}
			(x-x_0^r)(z-z^{-1})\sigma_3
			\right],
			&
			z\in
			\mathbb{C}^+
			\setminus
			\left(
			\Sigma_1^\ell\cup\Sigma_1^r
			\right),
			\\[3ex]
			\displaystyle
			\left(
			\Psi_1^r(x;z),
			\frac{\Psi_2^\ell(x;z)}{a^*(z)}
			\right)
			\exp\left[
			\frac{i}{2}
			(x-x_0^r)(z-z^{-1})\sigma_3
			\right],
			&
			z\in
			\mathbb{C}^-
			\setminus
			\left(
			\Sigma_2^\ell\cup\Sigma_2^r
			\right).
		\end{cases}
	\end{equation}
	
	Then the matrix-valued function $\Phi(x;z)$ satisfies the following RHP.
	
	\begin{RHP}\label{RHP:direct-M}
		The matrix function $\Phi(x;z)$ satisfies the following properties.
		
		\begin{enumerate}
			
			\item
			$\Phi(x;z)$ is analytic for
			$
			z\in
			\mathbb{C}
			\setminus
			\left(
			\mathbb{R}
			\cup
			\Sigma^\ell
			\cup
			\Sigma^r
			\cup
			\{0\}
			\right)$,
			with at most fourth-root singularities at the endpoints of the
			spectral bands.
			
			\item
			The boundary values of $\Phi(x;z)$ satisfy
			\begin{equation*}
				\Phi_+(x;z)
				=
				\Phi_-(x;z)V^{(\Phi)}(x;z),
			\end{equation*}
			where
			\begin{equation*}
				V^{(\Phi)}(x;z)
				=
				\begin{cases}
					
					\begin{pmatrix}
						\frac{a_-(z)}{a_+(z)}
						&
						-i
						e^{-i(x-x_0^r)(z-z^{-1})}
						\\[2ex]
						0
						&
						\frac{a_+(z)}{a_-(z)}
					\end{pmatrix},
					&
					z\in
					\Sigma_1^r\setminus\Sigma_1^\ell,
					
					\\[6ex]
					
					\begin{pmatrix}
						-\frac{i b_{2-}(z)}{a_+(z)}
						&
						-i
						e^{-i(x-x_0^r)(z-z^{-1})}
						\\[2ex]
						-\frac{i}{a_+(z)a_-(z)}
						e^{i(x-x_0^r)(z-z^{-1})}
						&
						\frac{i b_{1-}(z)}{a_-(z)}
					\end{pmatrix},
					&
					z\in
					\Sigma_1^r\cap\Sigma_1^\ell,
					
					\\[6ex]
					
					\begin{pmatrix}
						1
						&
						0
						\\[2ex]
						-\frac{i}{a_+(z)a_-(z)}
						e^{i(x-x_0^r)(z-z^{-1})}
						&
						1
					\end{pmatrix},
					&
					z\in
					\Sigma_1^\ell\setminus\Sigma_1^r,
					
					\\[6ex]
					
					\begin{pmatrix}
						\frac{1}{a(z)a^*(z)}
						&
						-\frac{b^*(z)}{a^*(z)}
						e^{-i(x-x_0^r)(z-z^{-1})}
						\\[2ex]
						\frac{b(z)}{a(z)}
						e^{i(x-x_0^r)(z-z^{-1})}
						&
						1
					\end{pmatrix},
					&
					z\in\mathbb{R},
					
					\\[6ex]
					
					\begin{pmatrix}
						1
						&
						\frac{i}{a_+^*(z)a_-^*(z)}
						e^{-i(x-x_0^r)(z-z^{-1})}
						\\[2ex]
						0
						&
						1
					\end{pmatrix},
					&
					z\in
					\Sigma_2^\ell\setminus\Sigma_2^r,
					
					\\[6ex]
					
					\begin{pmatrix}
						-\frac{i b_{1-}^*(z)}{a_-^*(z)}
						&
						\frac{i}{a_+^*(z)a_-^*(z)}
						e^{-i(x-x_0^r)(z-z^{-1})}
						\\[2ex]
						i
						e^{i(x-x_0^r)(z-z^{-1})}
						&
						\frac{i b_{2-}^*(z)}{a_+^*(z)}
					\end{pmatrix},
					&
					z\in
					\Sigma_2^r\cap\Sigma_2^\ell,
					
					\\[6ex]
					
					\begin{pmatrix}
						\frac{a_+^*(z)}{a_-^*(z)}
						&
						0
						\\[2ex]
						i
						e^{i(x-x_0^r)(z-z^{-1})}
						&
						\frac{a_-^*(z)}{a_+^*(z)}
					\end{pmatrix},
					&
					z\in
					\Sigma_2^r\setminus\Sigma_2^\ell.
					
				\end{cases}
			\end{equation*}
			
			\item
			As $z\to\infty$,
			\[
			\Phi(x;z)
			=
			I+\mathcal{O}(z^{-1}).
			\]
			As $z\to0$,
			\[
			\Phi(x;z)
			=
			\frac{\sigma_1}{z}
			+
			\mathcal{O}(1).
			\]
			
			\item
			The matrix $\Phi(x;z)$ satisfies the symmetry relations
			\[
			\Phi(x;z)
			=
			\sigma_1\Phi^*(x;z)\sigma_1
			=
			z^{-1}\Phi(x;z^{-1})\sigma_1.
			\]
			
		\end{enumerate}
	\end{RHP}
	
	\begin{proof}
		We calculate the jump matrices directly from \eqref{eq:def-M-direct} and recall the jump relations \eqref{eq:Psijump}.
		\par 
		First, let $z\in\Sigma_1^r\setminus\Sigma_1^\ell$. The function $\Psi_1^\ell$ has no jump on this part of the contour, whereas $\Psi_{2+}^r=-i\Psi_{1-}^r$. By the scattering relation \eqref{eq:scat}, we have
		$
		\Psi_{1-}^r
		=
		\frac{\Psi_1^\ell-b_{1-}\Psi_{2-}^r}{a_-}.
		$
		Hence
		\begin{align*}
			\Psi_{2+}^r
			e^{-\frac{i}{2}(x-x_0^r)(z-z^{-1})}
			={}
			-i
			e^{-i(x-x_0^r)(z-z^{-1})}
			\frac{\Psi_1^\ell}{a_-}
			e^{\frac{i}{2}(x-x_0^r)(z-z^{-1})}
			+
			\frac{i b_{1-}}{a_-}
			\Psi_{2-}^r
			e^{-\frac{i}{2}(x-x_0^r)(z-z^{-1})}.
		\end{align*}
		Since Proposition~\ref{pro3.6} gives $a_+=i b_{1-}$, the jump matrix is
		\[
		\begin{pmatrix}
			\frac{a_-}{a_+}
			&
			-i e^{-i(x-x_0^r)(z-z^{-1})}
			\\[2ex]
			0
			&
			\frac{a_+}{a_-}
		\end{pmatrix}.
		\]
		
		Next, let $z\in\Sigma_1^r\cap\Sigma_1^\ell$. We have $\Psi_{1+}^\ell=-i\Psi_{2-}^\ell$, $\Psi_{2+}^r=-i\Psi_{1-}^r$. The scattering relations \eqref{eq:scat} imply
		\[
		a_-\Psi_{2-}^\ell
		=
		b_{2-}\Psi_{1-}^\ell+\Psi_{2-}^r, \quad 
		a_-\Psi_{1-}^r
		=
		\Psi_{1-}^\ell-b_{1-}\Psi_{2-}^r.
		\]
		Therefore,
		\begin{equation*}
			\frac{\Psi_{1+}^\ell}{a_+}
			e^{\frac{i}{2}(x-x_0^r)(z-z^{-1})}
			={}
			-\frac{i b_{2-}}{a_+}
			\frac{\Psi_{1-}^\ell}{a_-}
			e^{\frac{i}{2}(x-x_0^r)(z-z^{-1})}-
			\frac{i}{a_+a_-}
			e^{i(x-x_0^r)(z-z^{-1})}
			\Psi_{2-}^r
			e^{-\frac{i}{2}(x-x_0^r)(z-z^{-1})},
		\end{equation*}
		and
		\begin{equation*}
			\Psi_{2+}^r
			e^{-\frac{i}{2}(x-x_0^r)(z-z^{-1})}
			={}
			-i
			e^{-i(x-x_0^r)(z-z^{-1})}
			\frac{\Psi_{1-}^\ell}{a_-}
			e^{\frac{i}{2}(x-x_0^r)(z-z^{-1})}+
			\frac{i b_{1-}}{a_-}
			\Psi_{2-}^r
			e^{-\frac{i}{2}(x-x_0^r)(z-z^{-1})}.
		\end{equation*}
		This gives the jump matrix on
		$\Sigma_1^r\cap\Sigma_1^\ell$.
		
		For $z\in\Sigma_1^\ell\setminus\Sigma_1^r$, the function $\Psi_2^r$ has no jump, while $\Psi_{1+}^\ell=-i\Psi_{2-}^\ell$. Using
		$
		a_-\Psi_{2-}^\ell
		=
		\Psi_2^r+b_{2-}\Psi_{1-}^\ell,
		$
		we find
		\begin{equation*}
			\frac{\Psi_{1+}^\ell}{a_+}
			e^{\frac{i}{2}(x-x_0^r)(z-z^{-1})}
			={}
			-\frac{i b_{2-}}{a_+}
			\frac{\Psi_{1-}^\ell}{a_-}
			e^{\frac{i}{2}(x-x_0^r)(z-z^{-1})}-
			\frac{i}{a_+a_-}
			e^{i(x-x_0^r)(z-z^{-1})}
			\Psi_2^r
			e^{-\frac{i}{2}(x-x_0^r)(z-z^{-1})}.
		\end{equation*}
		By Proposition~\ref{pro3.6}, $a_+=-i b_{2-}$. Thus the jump matrix is
		\[
		\begin{pmatrix}
			1
			&
			0
			\\[2ex]
			-\frac{i}{a_+a_-}
			e^{i(x-x_0^r)(z-z^{-1})}
			&
			1
		\end{pmatrix}.
		\]
		
		On the real axis, the scattering relation is
		$
		\Psi^\ell
		=
		\Psi^r
		\begin{pmatrix}
			a & b^*
			\\
			b & a^*
		\end{pmatrix}.
		$
		Thus
		$
		\Psi_1^\ell=a\Psi_1^r+b\Psi_2^r$ 
		and 
		$\Psi_2^\ell=b^*\Psi_1^r+a^*\Psi_2^r.
		$
		Solving the second relation for $\Psi_2^r$ and using $aa^*-bb^*=1$, we obtain
		\[
		\frac{\Psi_1^\ell}{a}
		=
		\frac{1}{aa^*}\Psi_1^r
		+
		\frac{b}{a}\frac{\Psi_2^\ell}{a^*}, \quad 
		\Psi_2^r
		=
		-\frac{b^*}{a^*}\Psi_1^r
		+
		\frac{\Psi_2^\ell}{a^*}.
		\]
		This gives the stated jump matrix on $\mathbb{R}$. The jump matrices in the lower half-plane can be derived similarly.
		\par 
		The normalization and symmetry properties follow from the corresponding properties of the Jost solutions \eqref{eq:sym} and the scattering data \eqref{eq:absym}.
	\end{proof}
	\par 
	So far, through the direct scattering analysis, we have obtained the scattering data $a(z)$, $b(z)$, $b_1(z)$, and $b_2(z)$, and constructed RHP~\ref{RHP:direct-M}. Our next goal is to use this RHP as a starting point to construct the full-soliton-gas RHP \ref{RHP2}, thereby revealing the connection between the RHP associated with periodic initial data and the full-soliton-gas RHP. The full-soliton-gas RHP exhibits a remarkably symmetric structure arising from the two background waves at the left and right spatial infinities. This symmetry also suggests that the contributions of the two backgrounds should be completely separated. Consequently, at least at a formal level, the present full-gas RHP can be conveniently reduced to the ``half-gas RHP'', i.e., the setting with a periodic-wave background on one side and a plane-wave background on the other; see \cite{BertolaWangYanZhu2026}.
	
	\par 
	As mentioned in Remark~\ref{b1b2}, we regard $b_1(z)$ as containing only the information associated with the right background wave, whereas $b_2(z)$ contains only the information associated with the left background wave. This observation motivates us to decompose $a(z)$ into two more general scattering coefficients, $a_1(z)$ and $a_2(z)$, so that 
	\begin{equation}\label{proaa}
		a_{1+}(z) = a_{1-}(z), \quad z \in \Sigma_1^{\ell}\setminus\Sigma_1^r, \quad \text{and}  \quad a_{2+}(z) = a_{2-}(z), \quad z \in \Sigma_1^r\setminus\Sigma_1^{\ell}.
	\end{equation}
	Additionally, to symmetrize the RHP \ref{RHP:direct-M}, we seek the following factorization of $a(z)$:
	\begin{equation*}
		a(z) = a_1(z) a_2(z),
	\end{equation*}
	where 
	\begin{equation}\label{eq:defa1a2}
		a_1(z) = \sqrt{a(z)/\gamma(z)} e^{ -\frac{i}{4}(x_0^{\ell} - x_0^r )(z - z^{-1}) }, \quad 
		a_2(z) = \sqrt{a(z)\gamma(z)} e^{ \frac{i}{4}(x_0^{\ell} - x_0^r )(z - z^{-1}) }, \quad z \in \mathbb{C}^+.
	\end{equation}
	Here, $\gamma(z)$ is a function to be determined.
	\par 
	Then we introduce the transformation 
	$
	m(z;x) = \gamma(\infty)^{\frac{\sigma_3}{2}} \Phi(z;x) \begin{cases}
		a_1(z)^{\sigma_3}, & z \in \mathbb{C}^+, \\
		a_1^*(z)^{-\sigma_3}, & z \in \mathbb{C}^-.
	\end{cases} 
	$
	This gives a symmetric ansatz
	\begin{equation}\label{RHPm}
		m(z;x) = \begin{cases}
			\displaystyle
			\gamma(\infty)^{\frac{\sigma_3}{2}}
			\left(
			\frac{\Psi_1^\ell(x;z)}{a_2(z)},
			\frac{\Psi_2^r(x;z)}{a_1(z)}
			\right)
			\exp\left[
			\frac{i}{2}
			(x- x_0^r  )(z-z^{-1})\sigma_3
			\right],
			&
			z\in
			\mathbb{C}^+
			\setminus
			\left(
			\Sigma_1^\ell\cup\Sigma_1^r
			\right),
			\\[3ex]
			\displaystyle
			\gamma(\infty)^{\frac{\sigma_3}{2}}
			\left(
			\frac{\Psi_1^r(x;z)}{a^*_1(z)},
			\frac{\Psi_2^\ell(x;z)}{a^*_2(z)}
			\right)
			\exp\left[
			\frac{i}{2}
			(x- x_0^r  )(z-z^{-1})\sigma_3
			\right],
			&
			z\in
			\mathbb{C}^-
			\setminus
			\left(
			\Sigma_2^\ell\cup\Sigma_2^r
			\right).
		\end{cases}
	\end{equation}
	\par 
	In order for the matrix $m(z;x)$ in \eqref{RHPm} to satisfy the full soliton gas RHP \ref{RHP2}, it remains to determine the function $\gamma(z)$. We now present the construction of $\gamma(z)$.
	\par 
	In order to establish property \eqref{proaa}, by using \eqref{eq:abrel}, we need
	\begin{equation}\label{eq:gamju1}
		\frac{\gamma_+(z)}{\gamma_-(z)} = -\frac{b_{2-}(z)}{b_{2+}(z)}, \quad z \in \Sigma_1^{\ell} \setminus \Sigma_1^r, \qquad 
		\frac{\gamma_+(z)}{\gamma_-(z)} = -\frac{b_{1+}(z)}{b_{1-}(z)}, \quad z \in \Sigma_1^r \setminus \Sigma_1^{\ell}.
	\end{equation}
	For $z \in \Sigma_1^r \cap \Sigma_1^{\ell}$, the jump matrix of $m(z;x)$ is given by
	\begin{equation}\label{eq:V^m}
		V^m(z)  = \begin{pmatrix}
			-
			\frac{
				i b_{2-}(z)
			}{
				a_{2+}(z) a_{1-}(z)
			}
			&
			\frac{-i}{a_{1-}(z) a_{1+}(z)}
			e^{-i(x-x_0^r)(z-z^{-1})}
			\\[2ex]
			-
			\frac{
				i a_{1-}(z) a_{1+}(z)
			}{
				a_+(z)a_-(z)
			}
			e^{i(x-x_0^r)(z-z^{-1})} 
			&
			\frac{
				i
				b_{1-}(z)
			}{
				a_{2-}(z) a_{1+}(z)
			}
		\end{pmatrix}.
	\end{equation}
	This gives
	\begin{equation}\label{eq:gamju2}
		\frac{\gamma_+(z)}{\gamma_-(z)} = \frac{ a_{2+}(z) a_{1-}(z) }{ a_{2-}(z) a_{1+}(z) } 
		= -\frac{ b_{2-}(z)  }{ b_{1-}(z) }
		= -\frac{b_{1+}(z)}{b_{1-}(z)}, \qquad z \in \Sigma_1^r \cap \Sigma_1^{\ell},
	\end{equation}
	where the second equality follows from the requirement that the jump matrix satisfies $V^{m}_{11}(z)=V^{m}_{22}(z)$, while the third equality follows from \eqref{b2-b1-}.
	\par 
	For $ z \in \mathbb{R} $, the jump matrix of $m(z;x)$ is given by
	\begin{equation}\label{eq:VmR}
		V^m(z) = \begin{pmatrix}
			\frac{a^*_{1-}(z) a_{1+}(z)}{|a(z)|^2}
			&
			-
			\frac{b^*(z)}{a^*(z)}
			e^{-i(x-x_0^r)(z-z^{-1})} \frac{a^*_{1-}(z)}{a_{1+}(z)}
			\\[2ex]
			\frac{b(z)}{a(z)}
			e^{i(x-x_0^r)(z-z^{-1})} \frac{a_{1+}(z)}{a^*_{1-}(z)}
			&
			\frac{1}{a^*_{1-}(z) a_{1+}(z)}
		\end{pmatrix},
	\end{equation}
	where the $(2,2)$-entry of the matrix is equal to $1$, namely, $\frac{1}{a^*_{1-}(z) a_{1+}(z)} = 1$. It follows that
	\begin{equation*}
		\frac{\gamma_+(z)}{\gamma_-(z)} =  \gamma_+(z) \gamma^*_-(z) = |a_+(z)|^2 = |a(z)|^2, \qquad z\in\mathbb{R}, 
	\end{equation*}
	where we have also used the assumption that $ \gamma^*(z) = \gamma^{-1}(z) $.
	\par 
	By symmetry, we obtain the jump condition for $\gamma(z)$ in the lower half-plane as
	\begin{equation}\label{eq:gamju3}
		\frac{\gamma_+(z)}{\gamma_-(z)} = -\frac{b_{2+}^*(z)}{b_{2-}^*(z)}, \quad z \in \Sigma_2^{\ell} \setminus \Sigma_2^r, \qquad 
		\frac{\gamma_+(z)}{\gamma_-(z)} = -\frac{b_{1-}^*(z)}{b_{1+}^*(z)}, \quad z \in \Sigma_2^r \setminus \Sigma_2^{\ell}, \qquad 
		\frac{\gamma_+(z)}{\gamma_-(z)} = -\frac{b_{1-}^*(z)}{b_{1+}^*(z)}, \quad z \in \Sigma_2^r \cap \Sigma_2^{\ell}.
	\end{equation}
	Still, we restrict boundary value $|\gamma(z)| < \infty$, as $z \to \infty$. The Sokhotski-Plemelj formula then gives the following expression for $\gamma(z)$:
	\begin{align}
		\log \gamma(z)
		=&
		\frac{1}{2\pi i}
		\Bigg[
		\int_{\Sigma_1^\ell\setminus\Sigma_1^r}
		\log\left(-\frac{b_{2-}(s)}{b_{2+}(s)}\right) \delta(s,z)\,ds
		+
		\int_{\Sigma_1^r\cap\Sigma_1^\ell}
		\log\left(-\frac{b_{1+}(s)}{b_{1-}(s)}\right) \delta(s,z)\,ds 	\notag\\
		&+
		\int_{\Sigma_1^r\setminus\Sigma_1^\ell}
		\log\left(-\frac{b_{1+}(s)}{b_{1-}(s)}\right) \delta(s,z)\,ds
		+
		\int_{\mathbb{R}}
		\log (|a(s)|^2) \delta(s,z)\,ds
		+
		\int_{\Sigma_2^\ell\setminus\Sigma_2^r}
		\log\left(-\frac{b_{2+}^*(s)}{b_{2-}^*(s)}\right) \delta(s,z)\,ds	\notag\\
		&+
		\int_{\Sigma_2^r\cap\Sigma_2^\ell}
		\log\left(-\frac{b_{1-}^*(s)}{b_{1+}^*(s)}\right) \delta(s,z)\,ds
		+
		\int_{\Sigma_2^r\setminus\Sigma_2^\ell}
		\log\left(-\frac{b_{1-}^*(s)}{b_{1+}^*(s)}\right) \delta(s,z)\,ds
		\Bigg],
		\label{eq:gamma}
	\end{align}
	where $\delta(s,z) = \frac{1}{s - z} - \frac{1}{2s}$ and $\log\gamma(\infty) \in i \mathbb{R}$. One can also obtain the symmetry
	\begin{equation}\label{eq:gammsy1}
		\gamma(z) = \gamma^*(z^{-1}) = (\gamma^*(z))^{-1}, \quad z\in\mathbb{C}\setminus(\mathbb{R} \cup \Sigma^\ell \cup \Sigma^r ), \qquad 
		\gamma_{\pm}(z) = \gamma^*_{\mp}(z^{-1}), \quad z\in \Sigma^\ell \cup \Sigma^r.
	\end{equation}
	\par 
	In summary, we have the following proposition concerning the function $\gamma(z)$.
	\begin{pro}\label{prop:gamma-properties}
		The function $\gamma(z)$, defined by \eqref{eq:gamma}, is analytic in $\mathbb{C}\setminus(\mathbb{R}\cup\Sigma^\ell\cup\Sigma^r)$ and satisfies the symmetry relation \eqref{eq:gammsy1}, the jump conditions \eqref{eq:gamju1}, \eqref{eq:gamju2}, and \eqref{eq:gamju3}, as well as the asymptotic condition
		\begin{equation}\label{eq:gamma-asym}
			\gamma(z) = \gamma(\infty) + \mathcal{O}(z^{-1}), \quad z\to\infty, \qquad
			\gamma(z) = \gamma(\infty)^{-1} + \mathcal{O}(z), \quad  z\to 0,
		\end{equation}
		where by symmetry, we also see that $\log \gamma(\infty) \in i\mathbb{R} $.
		Moreover, if $ u(x)-u_0^\ell(x)\in L^{1,2}(\mathbb{R}^-)$, $u(x)-u_0^r(x)\in L^{1,2}(\mathbb{R}^+)$, then near the endpoints,
		\begin{equation}\label{eq:endpoint-gamma}
			\begin{aligned}
				\gamma(z)&=(z-\zeta)^{-1/4} \left( e^{C_\zeta} + \mathcal{O}(| z - \zeta |^{1/2}) \right) ,
				&&\zeta\in\left\{
				\eta_1^\ell,\eta_2^\ell,
				\overline{\eta_1^r},\overline{\eta_2^r}
				\right\},\\
				\gamma(z)&=(z-\zeta)^{1/4} \left( e^{C_\zeta} + \mathcal{O}(| z - \zeta |^{1/2}) \right),
				&&\zeta\in\left\{
				\eta_1^r,\eta_2^r,
				\overline{\eta_1^\ell},\overline{\eta_2^\ell}
				\right\},
			\end{aligned}
		\end{equation}
		where $ C_\zeta $ is a constant associated with the endpoint $\zeta$.
	\end{pro}
	\begin{proof}
		Here, we only prove the endpoint behavior of $\gamma(z)$, since the remaining properties follow readily from its explicit form and the construction.
		\par 	
		Recall that the fourth-root factor associated with the background $s\in\{\ell,r\}$ is
		\begin{equation*}
			h^s(k(z)) = \left(\frac{z-1}{z+1}\right)^{1/2} \left[ \frac{ (z-\eta_2^s)(z-\overline{\eta_2^s}) }{ (z-\eta_1^s)(z-\overline{\eta_1^s}) } \right]^{1/4}.
		\end{equation*}
		Its boundary values on $\Sigma_1^s$ satisfy
		\begin{equation}\label{eq:h-upper-jump}
			h_+^s(k(z))=i h_-^s(k(z)).
		\end{equation}
		\par 
		We give the details at $\eta_2^\ell$; the other endpoints are treated in the same way. There are two possible configurations.
		
		First suppose that $\eta_2^\ell$ is the terminal endpoint of $\Sigma_1^\ell\setminus\Sigma_1^r$ and that $\eta_2^\ell$ is separated from $\Sigma_1^\ell\cap\Sigma_1^r$ (see Figures \ref{fig:cut1}, \ref{fig:cut2}, and \ref{fig:cut3}.). Since the right Jost solution is single-valued near $\eta_2^\ell$, the determinant definition of $b_2$ gives
		\begin{equation*}
			\frac{b_{2-}(z)}{b_{2+}(z)} = \frac{ h^\ell_+(k(z)) }{ h^\ell_-(k(z)) } 
			\frac{ \det\left[ \Psi_2^r(x;z), h^\ell_-(k(z))\Psi_{2-}^\ell(x;z) \right] }{ \det\left[ \Psi_2^r(x;z), h^\ell_+(k(z))\Psi_{2+}^\ell(x;z) \right] }. 
		\end{equation*}
		Using \eqref{eq:h-upper-jump} and \eqref{eq:hhpsi}, we obtain
		\begin{equation*}
			\ \frac{b_{2-}(z)}{b_{2+}(z)} = i \frac{ \det\left[ \Psi_2^r(x;\eta_2^\ell), \widetilde{\Psi}_2^\ell(x;\eta_2^\ell) \right] }{ \det\left[ \Psi_2^r(x;\eta_2^\ell), \widetilde{\Psi}_2^\ell(x;\eta_2^\ell) \right] }
			+ \mathcal{O}(|z-\eta_2^{\ell}|^{1/2}) = i + \mathcal{O}(|z-\eta_2^{\ell}|^{1/2}), \quad z \to \eta_2^\ell.
		\end{equation*}
		\par 
		Consequently, choosing the continuous branch of the logarithm along the contour, one has
		\begin{equation*}
			\log\left( -\frac{b_{2-}(z)}{b_{2+}(z)} \right) 
			= \log \left( -i( 1 +  \mathcal{O}(|z-\eta_2^{\ell}|^{1/2}) ) \right)
			=-\frac{i\pi}{2} + \mathcal{O}(|z-\eta_2^{\ell}|^{1/2}), \quad z \to \eta_2^\ell. 
		\end{equation*}
		Since $\eta_2^\ell$ is the terminal endpoint of the counterclockwise oriented arc $\Sigma_1^\ell$, the corresponding Cauchy integral satisfies
		\begin{equation*}
			\frac{1}{2\pi i} \int_{\Sigma_1^\ell\setminus\Sigma_1^r} \log\left( -\frac{b_{2-}(s)}{b_{2+}(s)} \right) \delta(s,z)\,ds = -\frac14\log(z-\eta_2^\ell) + 
			C_{\eta_2^\ell} + \mathcal{O}( |z - \eta_2^\ell|^{1/2}),
		\end{equation*}
		where $ C_{\eta_2^\ell} $ is a constant associated with $\eta_2^\ell$.
		It follows that
		\begin{equation*}
			\gamma(z) = (z-\eta_2^\ell)^{-1/4} \left(  e^{C_{\eta_2^\ell}} + \mathcal{O}( |z - \eta_2^\ell|^{1/2} ) \right), \quad  z \to \eta_2^\ell.
		\end{equation*}
		\par 
		We next consider the case in which $\eta_2^\ell \in \Sigma_1^r$. Since the left and right endpoints do not coincide, $\eta_2^\ell$ is locally the terminal endpoint of $\Sigma_1^\ell\cap\Sigma_1^r$ and the initial endpoint of $\Sigma_1^r\setminus\Sigma_1^\ell$ (see Figures \ref{fig:cut4}, \ref{fig:cut5}, and \ref{fig:cut6}).
		
		Approaching $\eta_2^\ell$ through $\Sigma_1^r\setminus\Sigma_1^\ell$, the left Jost column is single-valued. Hence by \eqref{eq:h-upper-jump} and \eqref{eq:hhpsi},
		\begin{equation}\label{eq:O1.2}
			\frac{b_{1+}(z)}{b_{1-}(z)} = 
			\frac{ \det\left[  \Psi_{1+}^r(x; \eta_2^\ell ), \widetilde{\Psi}_{1}^\ell(x; \eta_2^\ell  ) \right] }
			{ \det\left[ \Psi_{1-}^r(x; \eta_2^\ell ), \widetilde{\Psi}_{1}^\ell(x; \eta_2^\ell ) \right] }
			+\mathcal{O}(|z - \eta_2^\ell|^{1/2}).
		\end{equation}
		\par 
		Approaching through the overlap, both boundary values of the left Jost column occur. By \eqref{eq:h-upper-jump} and \eqref{eq:hhpsi},
		\begin{equation}\label{eq:b1-overlap-limit}
			\frac{b_{1+}(z)}{b_{1-}(z)}  
			= -i \frac{ \det\left[ \Psi_{1+}^r(x;\eta_2^\ell), \widetilde{\Psi}_{1}^\ell(x;\eta_2^\ell) \right] }{ \det\left[ \Psi_{1-}^r(x;\eta_2^\ell), \widetilde{\Psi}_{1}^\ell(x;\eta_2^\ell) \right] } +\mathcal{O}(|z - \eta_2^\ell|^{1/2}). 
		\end{equation}
		Thus the two limiting densities differ by the constant phase $-i$. Consequently, the piecewise-defined function
		\begin{equation*}
			\mathfrak{g}(z) = \begin{cases}
				\displaystyle \log\left( -\frac{b_{1+}(z)}{b_{1-}(z)} \right), & z\in\Sigma_1^r\cap\Sigma_1^\ell, \\ 
				\displaystyle  -\frac{i\pi}{2} + \log\left( -\frac{b_{1+}(z)}{b_{1-}(z)} \right) , & z\in\Sigma_1^r\setminus\Sigma_1^\ell, 
			\end{cases}
		\end{equation*}
		is continuous at $\eta_2^\ell$. The singular terms in $\log\gamma(z)$ can therefore be rewritten, up to a function analytic near $\eta_2^\ell$, as
		\begin{equation*}
			\begin{aligned}
				& \frac{1}{2\pi i} \int_{\Sigma_1^r\cap\Sigma_1^\ell} \log\left( -\frac{b_{1+}(s)}{b_{1-}(s)} \right) \delta(s,z)\,ds 
				+ \frac{1}{2\pi i} \int_{\Sigma_1^r\setminus\Sigma_1^\ell} \log\left( -\frac{b_{1+}(s)}{b_{1-}(s)} \right) \delta(s,z)\,ds \\ 
				=& \frac{1}{2\pi i} \int_{\Sigma_1^r} \mathfrak{g}(s)\delta(s,z)\,ds + \frac14 \log\left( \frac{z-\eta_2^r}{z-\eta_2^\ell} \right) + C_{\eta_2^\ell} + \mathcal{O}(|z-\eta_2^\ell|^{1/2}). 
			\end{aligned}
		\end{equation*}
		It follows from \eqref{eq:O1.2} and \eqref{eq:b1-overlap-limit} that $\mathfrak{g}$ is H\"older continuous of order $1/2$ at $\eta_2^\ell$. Thus the Cauchy integral containing $\mathfrak{g}$ has bounded non-tangential limits at $\eta_2^\ell$, whereas the explicit logarithm contributes $-\frac14\log(z-\eta_2^\ell)$.
		Therefore we also obtain the first expression in \eqref{eq:endpoint-gamma}.
		\par 
		Now, from the jump conditions for $\gamma(z)$, one can observe that the jump matrix of $m(z;x)$ already has the structure of the matrix in \eqref{jump2}. It remains to determine the explicit expressions for $r_1(z)$ and $r_2(z)$, thereby showing that the jump matrix of $m(z;x)$ is precisely the one given in \eqref{jump2}.
		\par 
		We begin by considering the jump matrix $V^{ m}$ on the overlap region $z\in\Sigma_1^r\cap\Sigma_1^\ell$. It can be seen that $ 1 + V^{ m}_{11} = \frac{ 2 }{ 1 + r_1(z) r_2(z) } $. Thus $  V^{ m}_{21} = - i r_1(z) ( 1 +  V^{ m}_{11} ) e^{i f(z)}  $, where $ f(z) $ is defined by \eqref{eq:f}. Therefore, combining this with \eqref{eq:V^m}, we obtain
		\begin{align}
			r_1(z)
			&=
			\frac{a_{1-}(z)}{a_{2-}(z)}
			\frac{
				e^{- i x_0^r (z - z^{-1} )}
			}{
				a_{2+}(z)a_{1-}(z)
				- i b_{2-}(z) },
			\qquad
			z\in\Sigma_1^\ell,
			\\[6pt]
			r_2(z)
			&=
			\frac{a_{2-}(z)}{a_{1-}(z)}
			\frac{
				e^{  i x_0^r (z - z^{-1} )}
			}{
				a_{2-}(z)a_{1+}(z)
				+ i b_{1-}(z) },
			\qquad
			z\in\Sigma_1^r,
		\end{align}
		where the second expression can be obtained similarly.
		\par 
		For $ z \in \mathbb{R} $, by comparing $ V^{ m}_{21} $ with \eqref{eq:VmR}, we obtain 
		\begin{equation*}
			\rho(z) = \frac{ b(z) }{ a_2(z) a_1^*(z)  } e^{-i x_0^r (z - z^{-1})}, \quad z \in \mathbb{R}.
		\end{equation*}
	\end{proof}
	\par 
	We now summarize the properties of the newly introduced scattering coefficients $a_1(z)$, $a_2(z)$, $r_1(z)$, $r_2(z)$, and $\rho(z)$.

	\begin{pro}[Symmetrized scattering coefficients]
		\label{prop:new-scattering-coefficients}
		Assume that Assumption~\ref{ass:no-discrete-spectrum} holds and that
		$u(x)-u_0^\ell(x)\in L^{1,2}(\mathbb R^-)$,
		and
		$u(x)-u_0^r(x)\in L^{1,2}(\mathbb R^+)$.
		Then the following properties hold.
		
		\begin{enumerate}
			\item The scattering coefficient $a(z)$ admits the factorization $a(z)=a_1(z)a_2(z)$ defined by \eqref{eq:defa1a2}. After removing the corresponding removable jumps, $a_1(z)$ and $a_2(z)$ are analytic and nonvanishing in $\mathbb C^+\setminus\Sigma_1^r$ and $\mathbb C^+\setminus\Sigma_1^\ell$ with property \eqref{proaa}, respectively. On the common spectral band,
			\begin{equation}\label{a1a2-tra-b1b2}
				\frac{a_{2+}(z)a_{1-}(z)}
				{a_{2-}(z)a_{1+}(z)}
				=
				-\frac{b_{2-}(z)}{b_{1-}(z)},
				\qquad
				z\in\Sigma_1^\ell\cap\Sigma_1^r.
			\end{equation}
			Moreover, they satisfy the symmetry
			\begin{equation}\label{eq:a1a2-sym}
				a_j(z)=a_j^*(z^{-1}),
				\qquad
				j=1,2.
			\end{equation}
			If, in addition $u(x)-u_0^\ell(x)\in W^{1,1}(\mathbb R^-)$ and $u(x)-u_0^r(x)\in W^{1,1}(\mathbb R^+)$, then
			\begin{equation*}
				\begin{aligned}
					&a_1(z)
					=
					\gamma(\infty)^{-1/2}
					+\mathcal O(z^{-1}),
					&&a_2(z) e^{-\frac{i}{2}(x_0^\ell - x_0^r)(z - z^{-1})}
					=
					\gamma(\infty)^{1/2}
					+\mathcal O(z^{-1}),
					&&&z\to\infty.  \\
					&a_1(z)
					=
					\gamma(\infty)^{1/2}
					+\mathcal O(z),
					&&a_2(z) e^{-\frac{i}{2}(x_0^\ell - x_0^r)(z - z^{-1})}
					=
					\gamma(\infty)^{-1/2}
					+\mathcal O(z),
					&&&z\to 0.
				\end{aligned}
			\end{equation*}
			\par 
			At the endpoints of $\Sigma_1^r\cup \Sigma_1^\ell$, the behavior is given by
			\begin{equation*}
				\begin{aligned}
					&a_1(z)=\mathcal O\left(|z-\zeta|^{-1/4}\right),
					\qquad
					a_2(z)=\mathcal O(1),
					&& z\to\zeta,\quad \zeta\in\partial\Sigma_1^r,\\
					& a_2(z)=\mathcal O\left(|z-\zeta|^{-1/4}\right),
					\qquad
					a_1(z)=\mathcal O(1),
					&& z\to\zeta,\quad \zeta\in\partial\Sigma_1^\ell.
				\end{aligned}
			\end{equation*}
			The corresponding behavior at the lower-half-plane endpoints follows from Schwarz symmetry.

			\item If, in addition, we assume that the denominators appearing in the definitions of $r_1(z)$ and $r_2(z)$ are nondegenerate at the corresponding endpoints, namely,
			\begin{equation*}
				\begin{aligned}
					\lim_{\substack{z\to\eta_j^\ell\\ z\in\Sigma_1^\ell}}
					(z-\eta_j^\ell)^{1/4}
					\left(
					a_{2+}(z)a_{1-}(z)-ib_{2-}(z)
					\right)
					&\neq0,
					\qquad j=1,2,\\
					\lim_{\substack{z\to\eta_j^r\\ z\in\Sigma_1^r}}
					(z-\eta_j^r)^{1/4}
					\left(
					a_{2-}(z)a_{1+}(z)+ib_{1-}(z)
					\right)
					&\neq0,
					\qquad j=1,2.
				\end{aligned}
			\end{equation*}
			Then the functions $r_1(z)$ and $r_2(z)$ have endpoint behavior
			\begin{equation*}
				\begin{aligned}
					r_1(z)
					&=
					\kappa_j^{\ell} (z-\eta_j^\ell)^{1/2}\bigl(1+o(1)\bigr),
					&&z\to\eta_j^\ell,\quad \kappa_j^{\ell}\neq0,\\
					r_2(z)
					&=
					\kappa_j^r (z-\eta_j^r)^{1/2}\bigl(1+o(1)\bigr),
					&&z\to\eta_j^r,\quad \kappa_j^r\neq0,
				\end{aligned}
				\qquad j=1,2,
			\end{equation*}
			Furthermore,
			\begin{equation*}
				r_1(z)>0,
				\quad
				z\in
				\Sigma_1^\ell
				\setminus
				\{\eta_1^\ell,\eta_2^\ell\},
				\qquad
				r_2(z)>0,
				\quad
				z\in
				\Sigma_1^r
				\setminus
				\{\eta_1^r,\eta_2^r\}.
			\end{equation*}
			On the nonoverlapping parts of the spectral bands, the expressions for $r_1(z)$ and $r_2(z)$ reduce to
			\begin{equation}\label{eq:r1r2setm}
				r_1(z)
				=
				\frac{
					e^{-ix_0^r(z-z^{-1})}
				}{
					2a_{2+}(z)a_{2-}(z)
				},
				\quad 
				z\in
				\Sigma_1^\ell\setminus\Sigma_1^r,
				\qquad
				r_2(z)
				=
				\frac{
					e^{ix_0^r(z-z^{-1})}
				}{
					2a_{1+}(z)a_{1-}(z)
				},
				\quad 
				z\in
				\Sigma_1^r\setminus\Sigma_1^\ell.
			\end{equation}
			
			\item The real-axis reflection coefficient satisfies the symmetry
			\begin{equation}\label{eq:rho-sym}
				\rho(z^{-1})
				=
				\overline{\rho(z)},
				\qquad
				z\in\mathbb R\setminus\{0\},
			\end{equation}
			and its values at $z=\pm1$ are
			$
			\rho(1)=-1
			$ and  $
			\rho(-1)=1.
			$
			If, in addition,
			\begin{equation*}
				u(x)\in W_{\mathrm{loc}}^{4,1}(\mathbb R),
				\qquad 
				u(x)-u_0^\ell(x)\in W^{4,1}(\mathbb R^-),
				\qquad
				u(x)-u_0^r(x)\in W^{4,1}(\mathbb R^+),
			\end{equation*}
			then
			\begin{equation}\label{eq:rho-aysm}
				\rho(z)
				=
				\mathcal O(z^{-4}),
				\quad
				z\to\pm\infty,
				\qquad
				\rho(z)
				=
				\mathcal O(z^4),
				\quad
				z\to0,
				\quad
				z\in\mathbb R.
			\end{equation}
			Consequently,
			\begin{equation*}
				r_1\in L^2(\Sigma_1^\ell),
				\qquad
				r_2\in L^2(\Sigma_1^r),
				\qquad
				\rho\in
				L^{2,2}(\mathbb R)
				\cap
				L^{1,2}(\mathbb R).
			\end{equation*}
		\end{enumerate}
	\end{pro}
	\begin{proof}
		We prove the three statements separately.
		
		The relation \eqref{a1a2-tra-b1b2} has already been obtained in \eqref{eq:gamju2} during the construction of $\gamma(z)$. The symmetry properties of $a_1(z)$ and $a_2(z)$ follow directly from the symmetry of $a(z)$ in \eqref{eq:absym}, their definitions in \eqref{eq:defa1a2}, and the symmetry of $\gamma(z)$ in \eqref{eq:gammsy1}. Their asymptotic behavior follows by substituting the asymptotic expansions of $a(z)$ and $\gamma(z)$ given in \eqref{eq:a-infinity-zero-asymptotics} and \eqref{eq:gamma-asym}, respectively, into \eqref{eq:defa1a2}. Finally, their endpoint behavior follows directly from the endpoint estimates for $a(z)$ in \eqref{eq:a-endpoint-behavior} and the endpoint behavior of $\gamma(z)$ in \eqref{eq:endpoint-gamma}.
		\par 
		We next prove the assertions concerning $r_1$ and $r_2$. We only prove the assertion for $r_1$ near $\eta_j^\ell$, since the proof for $r_2$ near $\eta_j^r$ is analogous. By the definitions of $a_1$ and $a_2$, we have
		\begin{equation*}
			\frac{a_{1-}(z)}{a_{2-}(z)}
			=
			\frac{e^{-\frac{i}{2}(x_0^\ell-x_0^r)(z-z^{-1})}}{\gamma_-(z)}
			=
			(z-\eta_j^\ell)^{1/4}
			\frac{ e^{-\frac{i}{2}(x_0^\ell-x_0^r)(z-z^{-1})}
			}{
				e^{C_{\eta_j^\ell}} + \mathcal{O}(|z - \eta_j^\ell|^{1/2})
			}.
		\end{equation*}
		\par 
		Set
		\begin{equation*}
			d_j^\ell
			=
			\lim_{\substack{z\to\eta_j^\ell\\ z\in\Sigma_1^\ell}}
			(z-\eta_j^\ell)^{1/4}
			\left(
			a_{2+}(z)a_{1-}(z)-ib_{2-}(z)
			\right), \quad 
			d_j^r
			=
			\lim_{\substack{z\to\eta_j^r\\ z\in\Sigma_1^r}}
			(z-\eta_j^r)^{1/4}
			\left(
			a_{2-}(z)a_{1+}(z)+ib_{1-}(z)
			\right)
		\end{equation*}
		By assumption, $d_j^\ell\neq0$, and hence
		\begin{equation*}
			r_1(z)
			=
			\frac{a_{1-}(z)}{a_{2-}(z)}
			\frac{
				e^{-ix_0^r(z-z^{-1})}
			}{
				a_{2+}(z)a_{1-}(z)-ib_{2-}(z)
			}
			=
			\kappa_j^{\ell}
			(z-\eta_j^\ell)^{1/2}
			\left(
			1+o(1)
			\right),
			\qquad
			z\to\eta_j^\ell,
		\end{equation*}
		where
		$
		\kappa_j^{\ell}
		=
		\frac{ \exp\left(-\frac{i}{2}(x_0^\ell+x_0^r)
			(\eta_j^\ell-(\eta_j^\ell)^{-1}) - C_{\eta_j^\ell} \right)
		}{
			d_j^\ell
		}
		\neq0.
		$
		\par 
		Similarly, we have the endpoint behavior of $r_2(z)$ with
		$
		\kappa_j^r
		=
		\frac{
			\exp\left(C_{\eta_j^r}  +  \frac{i}{2}(x_0^\ell+x_0^r)
			(\eta_j^r-(\eta_j^r)^{-1}) \right)
		}{
			d_j^r
		}
		\neq0.
		$
		\par 
		Next we only prove positivity of $r_1(z)$ and $r_2(z)$ on the overlap band. The remaining non-overlapping cases are covered by the same argument. For $z\in\Sigma_1^\ell\cap\Sigma_1^r$, we write
		\begin{equation*}
			r_1(z)
			=
			\frac{e^{-ix_0^r(z-z^{-1})}}
			{a_{2+}(z)a_{2-}(z)}
			\frac{a_{2+}(z)a_{1-}(z)}
			{a_{2+}(z)a_{1-}(z)-ib_{2-}(z)}.
		\end{equation*}
		Since $z^{-1}=\overline z$ on the unit circle and
		$a_{2+}(z)=\overline{a_{2-}(z)}$, we have
		$
		\frac{e^{-ix_0^r(z-z^{-1})}}
		{a_{2+}(z)a_{2-}(z)}
		=
		\frac{e^{-ix_0^r(z-\overline z)}}
		{|a_{2-}(z)|^2}
		>0.
		$
		On the other hand, using the definitions of $a_1$ and $a_2$, the jump relation for $\gamma$, and the scattering identities on the common spectral band, we obtain
		\begin{equation*}
			\left(
			\frac{a_{2+}(z)a_{1-}(z)}
			{ib_{2-}(z)}
			\right)^2
			=
			\frac{a_+(z)a_-(z)}
			{b_{1-}(z)b_{2-}(z)}
			=
			\frac{1+b_{1-}(z)b_{2-}(z)}
			{b_{1-}(z)b_{2-}(z)}
			=
			1+\frac{1}{|b_{1-}(z)|^2}>1.
		\end{equation*}
		Therefore, $\frac{a_{2+}(z)a_{1-}(z)}
		{ib_{2-}(z)}
		\in(-\infty,-1)\cup(1,\infty)$,
		and hence
		\begin{equation*}
			\frac{a_{2+}(z)a_{1-}(z)}
			{a_{2+}(z)a_{1-}(z)-ib_{2-}(z)}
			=
			\frac{
				\dfrac{a_{2+}(z)a_{1-}(z)}{ib_{2-}(z)}
			}{
				\dfrac{a_{2+}(z)a_{1-}(z)}{ib_{2-}(z)}-1
			}
			>0.
		\end{equation*}
		Consequently,
		$
		r_1(z)>0$,
		$	z\in\Sigma_1^\ell\cap\Sigma_1^r$.
		The same calculation gives $r_2(z)>0$ on the overlap band.
		\par 
		For $z\in\Sigma_1^\ell\setminus\Sigma_1^r$, the transformation relation \eqref{eq:a-b2-transformation}, the identity $a=a_1a_2$, and the removable jump of $a_1$ give
		\begin{equation*}
			a_{2+}(z)a_{1-}(z)-ib_{2-}(z)
			=
			2a_{2+}(z)a_{1-}(z),
		\end{equation*}
		which leads to the first expression in \eqref{eq:r1r2setm}, while the second one follows similarly from \eqref{eq:a-b2-transformation}.
		\par 
		The symmetry of $\rho$ in \eqref{eq:rho-sym} follows directly from the symmetries of $a$ and $b$ in \eqref{eq:absym}. The asymptotic behavior of $\rho$ in \eqref{eq:rho-aysm} follows from the asymptotic expansions of $a$ and $b$ in \eqref{eq:a-infinity-zero-asymptotics} and \eqref{eq:b-infinity-zero-asymptotics}, respectively.
		\par 
		Finally, let $z_0\in\{1,-1\}$. Since the Jost solutions are continuous at $z=z_0$, their symmetry \eqref{eq:sym} gives
		$
		\Psi^s(x;z_0)
		=
		z_0\Psi^s(x;z_0)\sigma_1.
		$
		Using the determinant definitions in \eqref{eq:defab}, we obtain
		\begin{equation*}
			\lim_{\substack{z\to z_0\\ z\in\mathbb R}}
			\frac{b(z)}{a(z)}
			=
			\frac{
				\det\left[
				\Psi_1^r(x;z_0),
				\Psi_1^\ell(x;z_0)
				\right]
			}{
				\det\left[
				\Psi_1^\ell(x;z_0),
				\Psi_2^r(x;z_0)
				\right]
			}
			=
			\frac{
				-\det\left[
				\Psi_1^\ell(x;z_0),
				\Psi_1^r(x;z_0)
				\right]
			}{
				z_0
				\det\left[
				\Psi_1^\ell(x;z_0),
				\Psi_1^r(x;z_0)
				\right]
			}
			=
			-z_0.
		\end{equation*}
		Therefore, by using symmetry \eqref{eq:a1a2-sym},
		$
		\rho(z_0) = \frac{b(z_0)}{a(z_0)}
		\frac{a_1(z_0)}{a_1^*(z_0)}
		e^{-ix_0^r(z_0-z_0^{-1})}
		=-z_0.
		$
	\end{proof}
	
	\begin{rmk}
		The positivity of $r_1$ and $r_2$, and hence their reality on the upper-half-plane spectral bands, ensures that the band jump matrices are compatible with the symmetry
		$
		m(z)=z^{-1}m(z^{-1})\sigma_1,
		$
		once the lower-band data are defined by Schwarz conjugation. In addition, $r_1,r_2>0$ implies $1+r_1r_2>0$ on the overlap of the two bands, so that the corresponding jump matrix is nonsingular. The endpoint relations $\rho(1)=-1$ and $\rho(-1)=1$, and hence $|\rho(\pm1)|=1$, are also consistent with the generic finite-density scattering behavior of the defocusing NLS equation. Indeed, in the long-time analysis, the stationary points reach the spectral endpoints $z=\pm1$ along the transition lines $x/t=\pm2$, and the unit-modulus values of the reflection coefficient constitute the characteristic critical behavior near these regions; see, for example, \cite{Jenkins2015,WangFan2023,FanLiYangZhang2026}.
	\end{rmk}
	
	\section{Time evolution and reconstruction}\label{sec:time-evolution}
	
	In the preceding sections, the direct scattering data $r_1$, $r_2$, and $\rho$ were constructed from the initial potential at $t=0$. We now regard them as the initial scattering data and incorporate their time evolution into the oscillatory factors of the jump matrices through the phase $f(z;x,t)=x(z-z^{-1})-t(z^2-z^{-2})$,
	which is determined by the time part of the Lax pair. Thus, RHP~\ref{RHP2} describes the evolution of the initial scattering data for arbitrary $t$. In this section, we prove Theorem~\ref{thm:inverse-problem} by establishing the unique solvability of RHP~\ref{RHP2}, verifying the regularity of the reconstructed potential, and showing that it satisfies the dNLS equation.
	\begin{proof}
		We prove the solvability of RHP~\ref{RHP2} by the standard singular-integral argument \cite{Zhou1989RHP,GravaJenkinsZhangZhang2026,Zhu2026}.
		\par 
		Fix $(x,t)$ and set
		\begin{equation*}
			\Gamma
			:=
			\mathbb R\cup\Sigma^\ell\cup\Sigma^r,
			\qquad
			V(z):=V^m(z;x,t),
			\qquad
			w(z):=V(z)-I.
		\end{equation*}
		To take account of the prescribed singularity at the origin, define
		\begin{equation*}
			m^{(0)}(z)
			:=
			I+\frac{\sigma_1}{z},
			\qquad
			n(z)
			:=
			m(z)-m^{(0)}(z).
		\end{equation*}
		Then $n$ is analytic in $\mathbb C\setminus\Gamma$ and satisfies
		$n(z)=\mathcal O(z^{-1})$ as $z\to\infty$. Let
		\begin{equation*}
			(C_\Gamma h)(z)
			:=
			\frac{1}{2\pi\mathrm i}
			\int_\Gamma
			\frac{h(s)}{s-z}\,ds,
			\qquad
			z\in\mathbb C\setminus\Gamma,
		\end{equation*}
		be the Cauchy transform on the oriented contour $\Gamma$, and denote by $C_+h$ and $C_-h$ its non-tangential boundary values from the positive and negative  sides of $\Gamma$, respectively. These operators satisfy the Sokhotski--Plemelj relation $C_+-C_-=I$ on $L^2(\Gamma)$. Set $\nu:=n_-\in L^2(\Gamma)$, where $n_-$ is the boundary value of $n$ from the negative side of $\Gamma$. The jump relation $m_+=m_-V$ is equivalent to
		\begin{equation*}
			n_+-n_-
			=
			\bigl(m^{(0)}+n_-\bigr)w
			=
			\bigl(m^{(0)}+\nu\bigr)w.
		\end{equation*}
		Therefore, by the Sokhotski--Plemelj formula and the normalization of $n$ at infinity,
		\begin{equation*}
			n(z)
			=
			C_\Gamma
			\left[
			\bigl(m^{(0)}+\nu\bigr)w
			\right](z).
		\end{equation*}
		Taking the negative boundary value gives
		\begin{equation*}
			(I-\mathcal C_w)\nu
			=
			C_-[m^{(0)}w],
			\qquad
			\mathcal C_wh
			:=
			C_-[hw].
		\end{equation*}
		Here $\mathcal C_w:L^2(\Gamma)\to L^2(\Gamma)$ is bounded because $C_-$ is bounded on $L^2(\Gamma)$ and $w\in L^\infty(\Gamma)$.
		\par 
		It remains to prove that the associated homogeneous problem has only the zero solution. If $(I-\mathcal C_w)\nu=0$, then $N(z):=C_\Gamma[\nu w](z)$ solves the homogeneous RHP \ref{Homogeneous RHP} and satisfies $N(z)=\mathcal O(z^{-1})$, $z\to\infty$.
		\begin{RHP}\label{Homogeneous RHP}
			Find a $2\times2$ matrix-valued function $N(z)=N(z;x,t)$ with the
			following properties:
			\begin{enumerate}
				\item $N$ is analytic in $\mathbb C\setminus\Gamma$.
				
				\item The boundary values satisfy $N_+(z)=N_-(z)V(z)$, $z\in\Gamma$.
				
				\item As $z\to\infty$, $N(z)=\mathcal O(z^{-1})$.
			\end{enumerate}
		\end{RHP}
		\par 
		On the upper spectral bands, the symmetries of the scattering data and the positivity of $r_1$ and $r_2$ give $V(\overline z)=V(z)^\dagger$, where the superscript $\dagger$ denotes the Hermitian transpose.
		On the real axis,
		\begin{equation*}
			\frac{V(z)+V(z)^\dagger}{2}
			=
			\begin{pmatrix}
				1-|\rho(z)|^2&0\\
				0&1
			\end{pmatrix}.
		\end{equation*}
		The scattering relation gives $1-|\rho(z)|^2>0$ for almost every $z\in\mathbb R$. The equality $|\rho(\pm1)|=1$ occurs only at two isolated points and does not affect the $L^2$ argument. Zhou's vanishing lemma \cite{Zhou1989RHP} therefore implies $N\equiv0$. Consequently, RHP~\ref{RHP2} has a unique solution, given by
		\begin{equation}\label{eq:short-RHP-solution-corrected}
			m(z)
			=
			m^{(0)}(z)
			+
			C_\Gamma
			\left[
			\bigl(\nu+m^{(0)}\bigr)w
			\right](z)
			=
			I+\frac{\sigma_1}{z}
			+
			\frac{1}{2\pi  i}
			\int_\Gamma
			\frac{
				\left(
				\nu(s)+I+\frac{\sigma_1}{s}
				\right)w(s)
			}{
				s-z
			}
			\,ds,
			\qquad
			z\in\mathbb C\setminus\Gamma.
		\end{equation}
		\par 
		Finally, we verify the regularity of the reconstructed potential.
		Differentiating the singular integral equation
		\begin{equation*}
			(I-\mathcal C_w)\nu
			=
			C_-[m^{(0)}w]
		\end{equation*}
		with respect to $x$ and $t$, and using the boundedness of $(I-\mathcal C_w)^{-1}$ on $L^2(\Gamma)$, together with the assumptions on $\rho(z)$ and its symmetry, we obtain
		$(s-s^{-1})\rho$,
		$(s-s^{-1})^2\rho$,
		$(s^2-s^{-2})\rho
		\in L^2(\mathbb R)$,
		and
		$(s-s^{-1})^2\rho$,
		$(s^2-s^{-2})\rho
		\in L^1(\mathbb R)$.
		It follows that
		$\partial_x\nu$,
		$\partial_x^2\nu$,
		$\partial_t\nu
		\in L^2(\Gamma)$.
		\par 
		Expanding \eqref{eq:short-RHP-solution-corrected} as $z\to\infty$
		gives
		\begin{equation*}
			u(x,t)
			=
			\gamma(\infty)
			-
			\frac{\gamma(\infty)}{2\pi i}
			\int_\Gamma
			\left[
			\left(
			\nu(s)+I+\frac{\sigma_1}{s}
			\right)w(s)
			\right]_{21}
			ds.
		\end{equation*}
		We split this expression as
		\begin{equation*}
			u(x,t)= \gamma(\infty) + u_{\mathbb R}(x,t)+u_\Gamma(x,t),
		\end{equation*}
		where $u_{\mathbb R}$ denotes the integral over $\mathbb R$ and $u_\Gamma$ contains the integrals over $\Gamma\setminus\mathbb R$.
		
		On the real axis,
		\begin{equation*}
			\left[
			\left(
			\nu+I+\frac{\sigma_1}{s}
			\right)w
			\right]_{21}
			=
			-
			\left(
			\nu_{21}+\frac{1}{s}
			\right)|\rho|^2
			+
			(\nu_{22}+1)\rho e^{if}.
		\end{equation*}
		Hence
		\begin{equation}\label{eq:u-real-part}
			u_{\mathbb R}(x,t)
			=
			-\frac{\gamma(\infty)}{2\pi i}
			\int_{\mathbb R}
			\left[
			-
			\left(
			\nu_{21}(s;x,t)+\frac{1}{s}
			\right)|\rho(s)|^2
			+
			\bigl(\nu_{22}(s;x,t)+1\bigr)
			\rho(s)e^{if(s;x,t)}
			\right]
			ds.
		\end{equation}
		
		Taking two derivatives of \eqref{eq:u-real-part} with respect to $x$ produces terms involving $\partial_x\nu$, $\partial_x^2\nu$, and the weighted reflection coefficients $(s-s^{-1})\rho$ and $(s-s^{-1})^2\rho$. These terms are absolutely integrable. Therefore $u_{\mathbb R}(\cdot,t)\in C^2(\mathbb R)$. Similarly, $u_{\mathbb R}(x,\cdot)\in C^1(\mathbb R^+)$.
		
		Since the remaining contour consists of finitely many compact spectral bands away from the origin, the boundedness of $r_1$ and $r_2$ yields the same regularity for $u_\Gamma$; hence $u(x,t) \in C^2(\mathbb{R}) \times C^1(\mathbb{R}^+)$, and the standard Lax-pair argument shows that $u$ is a classical solution of the dNLS equation.
	\end{proof}

	\appendix

	\section{Proofs of the properties of the Jost solutions}\label{AppA}
	In this appendix, we provide the proofs of several results concerning the Jost solutions stated in Section~\ref{Secfast}. In particular, we establish their analyticity, symmetry properties, large-$|z|$ asymptotics, and behavior near the branch points.
	
	\subsection{Proof of Proposition \ref{pro:M-large-z}}\label{app:proMl}
	We prove the result for $s=\ell$. The proof for $s=r$ is identical, with the direction of integration reversed.
	
	Set $M_0^\ell(x;z)=I$ and define recursively
	\begin{align*}
		M_{n+1}^\ell(x;z)
		={}&
		\int_{-\infty}^{x}
		e^{-i(x-y)p^\ell(z)\sigma_3}
		Y^\ell(z;y,0)^{-1}
		e^{-i(y-x_0^\ell)p_\infty^\ell\sigma_3}
		\\
		&\quad\times
		i\sigma_3\Delta Q^\ell(y)
		e^{i(y-x_0^\ell)p_\infty^\ell\sigma_3}
		Y^\ell(z;y,0)M_n^\ell(y;z)
		e^{i(x-y)p^\ell(z)\sigma_3}
		\,dy .
	\end{align*}
	For either column in its corresponding domain of analyticity, the exponential factors in the Volterra equation are bounded. Moreover, $Y^\ell$ and $(Y^\ell)^{-1}$ are uniformly bounded for sufficiently large $z$. It follows inductively that there exists a constant $\mathcal{C}>0$, independent of $n$ and of sufficiently large $z$, such that
	\[
	\|M_n^\ell(x;z)\|
	\leq
	\frac{1}{n!}
	\left(
	\mathcal{C}\int_{-\infty}^{x}
	\|\Delta Q^\ell(y)\|\,dy
	\right)^n .
	\]
	Indeed, the factor $1/n!$ follows from the ordered integration
	region $-\infty<y_n<\cdots<y_1<x$. Hence the series $I+\sum_{n=1}^{\infty}M_n^\ell(x;z)$
	converges absolutely and uniformly for sufficiently large $z$.
	Substitution of this series into \eqref{eq:direct_z_Volterra}
	shows that it solves the Volterra equation. By uniqueness,
	$M^\ell(x;z)
	=
	I+\sum_{n=1}^{\infty}M_n^\ell(x;z)$.
	\par 
	We next show that, for every fixed $n\geq1$, $M_n^\ell(x;z) \to 0$, $z\to\infty$.
	Since $Y^\ell(z;y,0)^{\pm1} = I+\mathcal{O}(z^{-1})$, we have
	\begin{align*}
		&Y^\ell(z;y,0)^{-1}
		e^{-i(y-x_0^\ell)p_\infty^\ell\sigma_3}
		i\sigma_3\Delta Q^\ell(y)
		e^{i(y-x_0^\ell)p_\infty^\ell\sigma_3}
		Y^\ell(z;y,0)
		\\
		&\qquad=
		e^{-i(y-x_0^\ell)p_\infty^\ell\sigma_3}
		i\sigma_3\Delta Q^\ell(y)
		e^{i(y-x_0^\ell)p_\infty^\ell\sigma_3}
		+
		\mathcal{O}\left(
		\frac{\|\Delta Q^\ell(y)\|}{|z|}
		\right).
	\end{align*}
	The first matrix on the right-hand side is off-diagonal. Consequently, after expanding the $n$-fold integral, every entry of its leading part is a finite sum of integrals of the form
	\[
	\int_{-\infty<y_n<\cdots<y_1<x}
	F(y_1,\ldots,y_n)
	e^{2ip^\ell(z)
		\left(
		\pm(x-y_1)
		\pm(y_1-y_2)
		\pm\cdots
		\pm(y_{n-1}-y_n)
		\right)}
	\,dy_n\cdots dy_1,
	\]
	where
	$
	|F(y_1,\ldots,y_n)|
	\leq
	C^n\prod_{j=1}^{n}
	\|\Delta Q^\ell(y_j)\|.
	$
	Thus $F$ is integrable on the ordered region. The phase is a nonconstant linear function of $y_1,\ldots,y_n$, and $|p^\ell(z)|\to\infty$. The multidimensional Riemann--Lebesgue lemma therefore implies that the leading part tends to zero. Every remaining term contains at least one $\mathcal{O}(z^{-1})$ factor and also tends to zero by the preceding integrable bound. Hence
	$
	M_n^\ell(x;z)=o(1)
	$
	for every fixed $n\geq1$.
	
	Finally, let $\varepsilon>0$. Choose $N$ sufficiently large that
	\[
	\sum_{n=N+1}^{\infty}
	\frac{1}{n!}
	\left(
	\mathcal{C}\int_{-\infty}^{x}
	\|\Delta Q^\ell(y)\|\,dy
	\right)^n
	<\frac{\varepsilon}{2}.
	\]
	This estimate is independent of $z$. Since each of the finitely many terms $M_1^\ell,\ldots,M_N^\ell$ tends to zero, for sufficiently large $z$,
	$
	\sum_{n=1}^{N}\|M_n^\ell(x;z)\|
	<\frac{\varepsilon}{2}.$ 
	Therefore,
	\[
	\|M^\ell(x;z)-I\|
	\leq
	\sum_{n=1}^{N}\|M_n^\ell(x;z)\|
	+
	\sum_{n=N+1}^{\infty}\|M_n^\ell(x;z)\|
	<\varepsilon.
	\]
	This proves
	$
	M^\ell(x;z)=I+o(1)$,
	$ z\to\infty.
	$
	\par 
	The large-$z$ behavior has already been established. We now prove the symmetry.
	\par 
	For convenience, recall the definition \eqref{def:H}. Then $M^s$ satisfies
	\[
	\partial_xM^s(x;z)
	=
	-ip^s(z)[\sigma_3,M^s(x;z)]
	+
	H^s(x;z)^{-1}i\sigma_3\Delta Q^s(x)
	H^s(x;z)M^s(x;z),
	\]
	together with
	$
	M^s(x;z)\to I$ for
	$x\to\infty^s.
	$
	
	Recall the symmetries corresponding to $p^s(z)$ and $Y^s(z;x,0)$
	\[
	p^s(z^{-1})=-p^s(z),
	\qquad
	Y^s(z;x,0)=z^{-1}Y^s(z^{-1};x,0)\sigma_1.
	\]
	Hence
	$
	H^s(x;z^{-1})
	=
	zH^s(x;z)\sigma_1$,
	and consequently
	\[
	\begin{aligned}
		H^s(x;z^{-1})^{-1}i\sigma_3\Delta Q^s(x)
		H^s(x;z^{-1})
		=
		\sigma_1H^s(x;z)^{-1}i\sigma_3\Delta Q^s(x)
		H^s(x;z)\sigma_1.
	\end{aligned}
	\]
	
	Define
	$
	\widetilde M^s(x;z)
	=
	\sigma_1M^s(x;z^{-1})\sigma_1.
	$
	Then
	\begin{align*}
		\partial_x\widetilde M^s(x;z)
		=
		-ip^s(z)[\sigma_3,\widetilde M^s(x;z)]
		+
		H^s(x;z)^{-1}i\sigma_3\Delta Q^s(x)
		H^s(x;z)\widetilde M^s(x;z).
	\end{align*}
	Moreover,
	$
	\widetilde M^s(x;z)\to I$,
	$ x\to\infty^s.
	$	Thus $\widetilde M^s(x;z)$ and $M^s(x;z)$ satisfy the same Volterra problem. By uniqueness,
	$
	M^s(x;z)
	=
	\sigma_1M^s(x;z^{-1})\sigma_1.
	$

	\subsection{Proof of Proposition \ref{pro3.3}, property \ref{prop:pro3.3-property-1} }\label{app:pro1}
	
	We prove the assertions for $\Psi_1^\ell(x;z)$. The other three columns can be treated in the same way.
	
	Taking the first column of \eqref{eq:direct_z_Volterra}, we obtain
	\begin{align}
		M^{\ell1}(x;z)
		={}e_1+
		\int_{-\infty}^{x}
		\begin{pmatrix}
			1 & 0\\
			0 & e^{2i(x-y)p^\ell(z)}
		\end{pmatrix}
		Y^\ell(z;y,0)^{-1}
		e^{-i(y-x_0^\ell)p_\infty^\ell\sigma_3}
		i\sigma_3\Delta Q^\ell(y)
		e^{i(y-x_0^\ell)p_\infty^\ell\sigma_3}
		Y^\ell(z;y,0)M^{\ell1}(y;z)\,dy,
		\label{eq:Volterra-first-column}
	\end{align}
	where $e_1=(1,0)^T$. 
	\par 
	Let $\mathcal{D}$ be a compact subset of $\mathbb{C}^+ \setminus \Sigma_1^\ell$, bounded away from $z=0$ and the endpoints of the cuts. By
	Lemma~\ref{lemfoin},
	$
	\operatorname{Im}p^\ell(z)>0$,
	$ z\in \mathcal{D}.
	$
	Since $y\leq x$, it follows that
	\[
	\left|e^{2i(x-y)p^\ell(z)}\right|
	=
	e^{-2(x-y)\operatorname{Im}p^\ell(z)}
	\leq1.
	\]
	Moreover, $p_\infty^\ell\in\mathbb{R}$, and the background matrices $Y^\ell(z;y,0)$ and $Y^\ell(z;y,0)^{-1}$ are uniformly bounded for $z\in \mathcal{D}$ and $y\in\mathbb{R}$. Let $C_{\mathcal{D}}>0$ denote a constant depending only on the compact set $\mathcal{D}$. Hence the kernel in \eqref{eq:Volterra-first-column} satisfies
	\begin{equation}	\label{eq:Volterra-kernel-bound}
		\left\|
		\begin{pmatrix}
			1 & 0\\
			0 & e^{2i(x-y)p^\ell(z)}
		\end{pmatrix}
		Y^\ell(z;y,0)^{-1}
		e^{-i(y-x_0^\ell)p_\infty^\ell\sigma_3}
		i\sigma_3\Delta Q^\ell(y)
		e^{i(y-x_0^\ell)p_\infty^\ell\sigma_3}
		Y^\ell(z;y,0)
		\right\|
		\leq C_\mathcal{D}\|\Delta Q^\ell(y)\|.
	\end{equation}

	Let $\mathcal{K}_z^\ell$ denote the integral operator in \eqref{eq:Volterra-first-column}. Iterating \eqref{eq:Volterra-kernel-bound} over the ordered simplex
	$
	-\infty<y_n<\cdots<y_1<x
	$
	gives
	\[
	\left\|
	(\mathcal{K}_z^\ell)^n e_1(x)
	\right\|
	\leq
	\frac{C_\mathcal{D}^n}{n!}
	\left(
	\int_{-\infty}^{x}
	\|\Delta Q^\ell(y)\|\,dy
	\right)^n.
	\label{eq:Volterra-iterate-bound}
	\]
	Consequently, for every $x_*\in\mathbb{R}$, the Neumann series
	$
	M^{\ell1}(x;z)
	=
	\sum_{n=0}^{\infty}
	(\mathcal{K}_z^\ell)^n e_1(x)
	$
	converges uniformly for $x\leq x_*$ and $z\in \mathcal{D}$. It therefore defines a solution of \eqref{eq:Volterra-first-column} belonging to $L^\infty((-\infty,x_*])$.
	
	To prove uniqueness, suppose that $M_{(1)}^{\ell1}$ and $M_{(2)}^{\ell1}$ are two solutions of \eqref{eq:Volterra-first-column} in $L^\infty((-\infty,x_*])$, and set
	\[
	f(x)
	=
	M_{(1)}^{\ell1}(x;z)
	-
	M_{(2)}^{\ell1}(x;z).
	\]
	Then $f=\mathcal{K}_z^\ell f$. Iterating this identity gives $f=(\mathcal{K}_z^\ell)^n f$, for $n\geq1$. By the same Volterra estimate as above, for every $x\leq x_*$,
	\[
	\|f\|_{L^\infty((-\infty,x_*])}
	\leq
	\frac{1}{n!}
	\left(
	C_\mathcal{D}
	\int_{-\infty}^{x_*}
	\|\Delta Q^\ell(y)\|\,dy
	\right)^n
	\|f\|_{L^\infty((-\infty,x_*])}.
	\]
	Therefore
	$
	\|f\|_{L^\infty((-\infty,x_*])}=0,
	$
	and thus
	$
	M_{(1)}^{\ell1} = M_{(2)}^{\ell1}.
	$
	This proves uniqueness. Each term of the Neumann series is analytic in $z\in \mathcal{D}$, and \eqref{eq:Volterra-iterate-bound} gives locally uniform convergence.
	Thus $M^{\ell1}(x;z)$ is analytic in $\mathbb{C}^+ \setminus \Sigma_1^\ell$. The same estimate, together with dominated convergence, also gives continuous boundary values away from $z=0$ and the endpoints of the cuts.
	
	It remains to show that the apparent jump across $\Sigma_0^\ell$ is removable. On $\Sigma_0^\ell$, one has
	$
	p_+^\ell(z)-p_-^\ell(z)=\Omega_1^\ell
	$
	and, at $t=0$,
	\[
	Y_+^\ell(z;x,0)
	=
	Y_-^\ell(z;x,0)
	e^{i(\Omega_0^\ell + x \Omega_1^{\ell} )\sigma_3}.
	\]
	Therefore, by \eqref{eq:direct_z_Volterra}, we have
	\begin{equation*}
		\begin{aligned}
			&e^{i(\Omega_0^\ell + x \Omega_1^{\ell} )\sigma_3} M^\ell_+(x;z) e^{-i(\Omega_0^\ell + x \Omega_1^{\ell} )\sigma_3}
			={}I+
			\int_{-\infty}^{x}
			e^{-i(x-y)p^\ell_-(z)\sigma_3}
			Y^\ell_-(z;y,0)^{-1}
			e^{-i(y-x_0^\ell)p_\infty^\ell\sigma_3}
			\nonumber\\
			&\qquad\qquad\qquad\times
			i\sigma_3\Delta Q^\ell(y)
			e^{i(y-x_0^\ell)p_\infty^\ell\sigma_3}
			Y^\ell_-(z;y,0) e^{i(\Omega_0^\ell + y \Omega_1^{\ell} )\sigma_3} M^\ell_+(y;z) e^{-i(\Omega_0^\ell + y \Omega_1^{\ell} )\sigma_3}
			e^{i(x-y)p^\ell_-(z)\sigma_3}
			\,dy .
		\end{aligned}
	\end{equation*}
	By uniqueness, 
	\begin{equation*}
		e^{i(\Omega_0^\ell + x \Omega_1^{\ell} )\sigma_3} M^\ell_+(x;z) e^{-i(\Omega_0^\ell + x \Omega_1^{\ell} )\sigma_3}
		= M^\ell_-(x;z) , \quad z \in \Sigma_0^\ell,
	\end{equation*}
	which implies that 
	$
	\Psi_{+}^\ell(x;z)
	=
	\Psi_{-}^\ell(x;z)$
	for $z\in\Sigma_0^\ell.
	$
	Hence the jump across $\Sigma_0^\ell$ is removable, and Morera's theorem yields that
	$
	\Psi_1^\ell(x;z)$ is analytic for $z\in\mathbb{C}^+\setminus\Sigma_1^\ell.$
	
	For $z\in\mathbb{R}\setminus\{0\}$ or for either boundary value in the interior of $\Sigma^\ell$, the quasi-momentum is real. The same estimates therefore hold for both columns, and the corresponding background factors are bounded in $x$. It follows that
	$
	\Psi^\ell(\cdot;z)
	\in L^\infty((-\infty,x_*]).
	$
	
	Finally, the remaining columns follow from the same argument. The
	relevant diagonal exponential factors are
	\[
	\begin{array}{c|c|c}
		\text{column}
		&
		\text{integration range}
		&
		\text{exponential factor}
		\\ \hline
		M^{\ell2}
		&
		y\leq x
		&
		e^{-2i(x-y)p^\ell(z)}
		\\
		M^{r1}
		&
		y\geq x
		&
		e^{2i(x-y)p^r(z)}
		\\
		M^{r2}
		&
		y\geq x
		&
		e^{-2i(x-y)p^r(z)}.
	\end{array}
	\]
	Using Lemma~\ref{lemfoin}, these factors are bounded precisely in the stated domains. Therefore $\Psi_2^\ell(x;z)$ and $\Psi_1^r(x;z)$ are analytic in $ \mathbb{C}^-\setminus\Sigma_2^\ell$ and $ \mathbb{C}^-\setminus\Sigma_2^r$, respectively, whereas $\Psi_2^r(x;z)$ is analytic in $\mathbb{C}^+\setminus\Sigma_1^r$.
	The corresponding right Jost solution belongs to $L^\infty([x_*,\infty))$, completing the proof.

	\subsection{Proof of Proposition \ref{pro3.3} property \ref{prop:pro3.3-property-2}}
	For $z\in\Sigma_1^s\cup\Sigma_2^s$, the boundary values of the quasi-momentum satisfy
	$
	p_+^s(z)+p_-^s(z)=0.
	$
	Moreover, from the jump conditions of the background matrix
	$Y^s(z;x,0)$, we have
	\begin{equation*}
		Y_+^s(z;x,0)
		=
		Y_-^s(z;x,0)
		\begin{cases}
			-i\sigma_1,
			& z\in\Sigma_1^s,
			\\[1ex]
			i\sigma_1,
			& z\in\Sigma_2^s.
		\end{cases}
	\end{equation*}
	
	We first compare the boundary values of $M^s$. For $z \in \Sigma_1^s$, substituting the above jump relations into the Volterra equation \eqref{eq:direct_z_Volterra}, one obtains
	\begin{align}
		-i\sigma_1M_+^s(x;z) i \sigma_1
		=
		I+
		&\int_{\infty^s}^{x}
		e^{-i(x-y)
			p_-^s(z)\sigma_3}
		Y_-^s(z;y,0)^{-1} e^{-i(y - x_0^s) p_{\infty}^s \sigma_3 }
		i\sigma_3\Delta Q^s(y) e^{i(y - x_0^s) p_{\infty}^s \sigma_3 }
		Y_-^s(z;y,0)
		\nonumber
		\\
		&\qquad\times
		(-i\sigma_1) M_+^s(y;z) i\sigma_1
		e^{i(x-y)
			p_-^s(z)\sigma_3}
		\,dy. \label{eq:integ}
	\end{align}
	The integral equation \eqref{eq:integ} is exactly the Volterra equation satisfied by $M_-^s(x;z)$. Therefore, by uniqueness,
	$
	M_+^s(x;z)
	=
	-i\sigma_1
	M_-^s(x;z)
	i\sigma_1$, for $z\in \Sigma_1^s.
	$
	It follows from \eqref{eq:direct_z_M_def} and the jump relation for $Y_\pm^s(z;x)$ that
	$
	\Psi_+^s(x;z)
	=
	\Psi_-^s(x;z)(-i\sigma_1)$, for  $z\in\Sigma_1^s.
	$
	The corresponding relation on $\Sigma_2^s$ follows in the same way.

	\subsection{Proof of Proposition \ref{pro3.3}, property \ref{prop:pro3.3-property-3}}
	We first prove the expansion as $z\to\infty$. It is sufficient to consider the left Jost solution. All the estimates below are understood columnwise in the corresponding domains of analyticity.
	
	Let $\mathcal{K}_z^\ell$ denote the integral operator appearing in \eqref{eq:direct_z_Volterra}, namely,
	\begin{equation*}
		(\mathcal{K}_z^\ell F)(x)
		=
		\int_{-\infty}^{x}
		e^{-i(x-y)p^\ell(z)\sigma_3}
		B^\ell(y;z)F(y;z)
		e^{i(x-y)p^\ell(z)\sigma_3}\,dy,
	\end{equation*}
	where
	\begin{equation*}
		B^\ell(y;z)
		=
		Y^\ell(z;y,0)^{-1}
		e^{-i(y-x_0^\ell)p_\infty^\ell\sigma_3}
		i\sigma_3\Delta Q^\ell(y)
		e^{i(y-x_0^\ell)p_\infty^\ell\sigma_3}
		Y^\ell(z;y,0).
	\end{equation*}
	By the existence and uniqueness result proved in property \ref{prop:pro3.3-property-1}, one has the convergent Neumann expansion
	\begin{equation*}
		M^\ell(x;z)
		=
		I+\sum_{m=1}^{\infty}(\mathcal{K}_z^\ell)^mI(x).
	\end{equation*}
	
	Since $Y^\ell(z;y,0)$ is the finite-gap background matrix, both $Y^\ell(z;y,0)$ and its inverse possess Laurent expansions of arbitrary order at infinity. Their coefficients, together with all their $y$-derivatives, are bounded on the real axis. Consequently,
	\begin{equation*}
		B^\ell(y;z)
		=
		B_0^\ell(y)+\sum_{j=1}^{n-1}\frac{B_j^\ell(y)}{z^j}
		+\mathcal{O}(z^{-n}),
	\end{equation*}
	where the expansion and the corresponding derivatives are controlled in $L^1((-\infty,x])$. The leading coefficient is
	\begin{equation*}
		B_0^\ell(y)
		=
		e^{-i(y-x_0^\ell)p_\infty^\ell\sigma_3}
		i\sigma_3\Delta Q^\ell(y)
		e^{i(y-x_0^\ell)p_\infty^\ell\sigma_3},
	\end{equation*}
	and is off-diagonal.
	
	We now examine the large-$z$ behavior of the successive Volterra iterates. Conjugation by the diagonal exponential leaves diagonal matrices unchanged, whereas the off-diagonal entries acquire the factors $e^{2i(x-y)p^\ell(z)}$ or $e^{-2i(x-y)p^\ell(z)}$. For a function $g$ such that $g,g'\in L^1((-\infty,x])$, integration by parts gives
	\begin{equation*}
		\int_{-\infty}^{x}
		e^{\pm2i(x-y)p^\ell(z)}g(y)\,dy
		=
		\mp\frac{g(x)}{2ip^\ell(z)}
		\pm\frac{1}{2ip^\ell(z)}
		\int_{-\infty}^{x}
		e^{\pm2i(x-y)p^\ell(z)}g'(y)\,dy.
	\end{equation*}
	The boundary contribution at $y=-\infty$ vanishes because $g,g'\in L^1((-\infty,x])$. Since
	$
	p^\ell(z)
	=
	\frac{z}{2}+p_\infty^\ell+\mathcal{O}(z^{-1}),
	$ for $z \to \infty$, 
	each such integration by parts produces one additional factor of order $z^{-1}$.
	
	The identity matrix is diagonal, while $B_0^\ell$ is off-diagonal. Therefore, in the first iteration, the leading term is off-diagonal and contains an oscillatory exponential. One integration by parts gives
	\begin{equation*}
		\mathcal{K}_z^\ell I
		=
		\mathcal{O}(z^{-1}), \qquad z\to\infty
	\end{equation*}
	Its leading coefficient is off-diagonal. In the second iteration, multiplication by $B_0^\ell$ transforms this off-diagonal leading coefficient into a diagonal one. No additional integration by parts is needed at this step, and hence
	\begin{equation*}
		(\mathcal{K}_z^\ell)^2I
		=
		\mathcal{O}(z^{-1}), \qquad z\to\infty
	\end{equation*}
	At the third iteration, the leading diagonal coefficient is again transformed into an off-diagonal one, and another integration by parts yields an additional factor $z^{-1}$. Continuing in this way, one obtains
	\begin{equation*}
		(\mathcal{K}_z^\ell)^mI
		=
		\mathcal{O}\left(z^{-\lceil m/2\rceil}\right),
		\qquad z\to\infty.
	\end{equation*}
	
	More precisely, repeated integration by parts, together with the Laurent expansion of $B^\ell(y;z)$, shows that the $m$th iterate admits an expansion through order $z^{-n}$ whose first possible power is $z^{-\lceil m/2\rceil}$. The terms $B_j^\ell$, $j\geq1$, already contain at least one additional power of $z^{-1}$ and therefore do not produce any lower-order contribution. The assumption
	$
	u(x)-u_0^\ell(x)\in W^{n,1}(\mathbb{R}^-)
	$
	guarantees that all integrations by parts required up to order $z^{-n}$ are justified.
	
	It follows that only the iterates with $m\leq2n-2$ can contribute to the coefficients of $z^{-1},\ldots,z^{-(n-1)}$. Moreover, the ordered-simplex estimates used in Subsection \ref{app:pro1} imply that the constants in these estimates are summable with respect to $m$. Hence
	\begin{equation*}
		\sum_{m=2n-1}^{\infty}
		(\mathcal{K}_z^\ell)^mI(x)
		=
		\mathcal{O}(z^{-n}),
	\end{equation*}
	and therefore
	\begin{equation*}
		M^\ell(x;z)
		=
		I+\sum_{j=1}^{n-1}\frac{\widehat{M}_j^\ell(x)}{z^j}
		+\mathcal{O}(z^{-n}).
	\end{equation*}
	The same argument gives the expansion of $M^r(x;z)$.
	\par 
	We now return to the Jost solutions. By \eqref{eq:direct_z_M_def},
	\begin{equation*}
		\Psi^s(x;z)e^{i(x-x_0^s)k(z)\sigma_3}
		={}
		e^{i(x-x_0^s)p_\infty^s\sigma_3}
		Y^s(z;x,0)M^s(x;z) 
		e^{-i(x-x_0^s)(p^s(z)-k(z))\sigma_3}.
	\end{equation*}
	Multiplying the preceding expansions therefore gives
	\begin{equation*}
		\Psi^s(x;z)e^{i(x-x_0^s)k(z)\sigma_3}
		=
		I+\sum_{j=1}^{n-1}\frac{C_j^s(x)}{z^j}
		+\mathcal{O}(z^{-n}),
		\qquad z\to\infty.
	\end{equation*}
	
	It remains to determine the structure of $C_1^s(x)$. Set
	\begin{equation}\label{eq:defW}
		W^s(x;z)
		=
		\Psi^s(x;z)e^{i(x-x_0^s)k(z)\sigma_3}.
	\end{equation}
	Using the spatial part of the Lax pair \eqref{eq:direct_z_lax}, we obtain
	\begin{equation}\label{eq:appwx}
		\partial_xW^s
		=
		i\sigma_3Q(x)W^s-ik(z)[\sigma_3,W^s].
	\end{equation}
	Substituting
	\begin{equation*}
		W^s(x;z)
		=
		I+\frac{C_1^s(x)}{z}+\mathcal{O}(z^{-2})
	\end{equation*}
	in \eqref{eq:appwx} and recalling $k(z)=\frac{z}{2}+\frac{1}{2z}$, the terms of order one give $[\sigma_3,C_1^s(x)]	= 2\sigma_3Q(x)$.
	It follows that
	$(C_1^s)_{12}
	=
	\overline{u(x)}$,
	and
	$(C_1^s)_{21}
	=
	u(x).$
	
	The coefficient matrix of the spatial Lax equation is traceless, so $\det\Psi^s(x;z)$ is independent of $x$. Consequently, all the coefficients in its asymptotic expansion as $z\to\infty$ are also independent of $x$. Indeed, the spatial part of the Lax pair can be written as
	$
	\partial_x\Psi^s(x;z)=U(x;z)\Psi^s(x;z)$,
	$\operatorname{tr}U(x;z)=0.
	$
	By Jacobi's formula,
	$
	\partial_x\det\Psi^s
	=
	\det\Psi^s\,
	\operatorname{tr}\left((\Psi^s)^{-1}\partial_x\Psi^s\right).
	$
	Substituting the spatial Lax equation and using the invariance of the trace under similarity transformations, we obtain
	$
	\operatorname{tr}\left((\Psi^s)^{-1}\partial_x\Psi^s\right)
	=
	\operatorname{tr}\left((\Psi^s)^{-1}U\Psi^s\right)
	=
	\operatorname{tr}U
	=
	0.
	$
	Therefore,
	$
	\partial_x\det\Psi^s(x;z)=0.
	$
	Evaluating this $x$-independent quantity by means of the Jost normalization and using $\det Y^s(z;x,0)=1-z^{-2}$, we obtain
	\begin{equation*}
		\det W^s(x;z)=\det\Psi^s(x;z)=1-z^{-2}.
	\end{equation*}
	On the other hand,
	\begin{equation*}
		\det\left(
		I+\frac{C_1^s(x)}{z}+\mathcal{O}(z^{-2})
		\right)
		=
		1+\frac{\operatorname{tr}C_1^s(x)}{z}
		+\mathcal{O}(z^{-2}).
	\end{equation*}
	Consequently, $\operatorname{tr}C_1^s(x)=0$. Thus, for some scalar function $A^s(x)$,
	\begin{equation*}
		C_1^s(x)
		=
		\begin{pmatrix}
			A^s(x) & \overline{u(x)} \\
			u(x) & -A^s(x)
		\end{pmatrix}.
	\end{equation*}
	
	Taking the diagonal part of the coefficient of $z^{-1}$ in the equation for $W^s$ \eqref{eq:appwx}, we obtain
	\begin{equation*}
		\partial_x\operatorname{diag}C_1^s(x)
		=
		\operatorname{diag}\left(
		i\sigma_3Q(x)C_1^s(x)
		\right)
		=
		\operatorname{diag}\left(
		i\sigma_3Q(x)C_1^s(x)
		\right)
		=
		\begin{pmatrix}
			i|u(x)|^2 & 0 \\
			0 & -i|u(x)|^2
		\end{pmatrix}.
	\end{equation*}
	We conclude that
	$
	\frac{d}{dx}A^s(x)
	=
	i|u(x)|^2.
	$
	\par 
	Finally, by using the symmetry of RHP \ref{RH5} and \eqref{eq:symMs}, we derive \eqref{eq:invsym}, i.e., 	$\Psi^s(x;z) = z^{-1}\Psi^s(x;z^{-1})\sigma_1$.
	Applying the large-$z$ expansion to $z^{-1}$ gives
	\begin{equation*}
		\Psi^s(x;z)e^{-i(x-x_0^s)k(z)\sigma_3}
		=
		\frac{\sigma_1}{z}
		+\sum_{j=1}^{n-1}z^{j-1}C_j^s(x)\sigma_1
		+\mathcal{O}(z^{n-1}),
		\qquad z\to0.
	\end{equation*}
	In particular,
	\begin{equation*}
		\Psi^s(x;z)e^{-i(x-x_0^s)k(z)\sigma_3}
		=
		\frac{\sigma_1}{z}
		+
		\begin{pmatrix}
			\overline{u(x)} & A^s(x) \\
			-A^s(x) & u(x)
		\end{pmatrix}
		+\mathcal{O}(z),
		\qquad z\to0.
	\end{equation*}

	\subsection{Proof of Proposition \ref{pro3.3}, property \ref{prop:pro3.3-property-4}}
	We first establish a uniform estimate for the matrix $\Psi_0^s(x;z)\Psi_0^s(y;z)^{-1}$. By \eqref{eq:direct_z_background} and
	\begin{equation}\label{eq:trF}
		Y^s(z;x,0)=\sqrt{1-z^{-2}}\,F^s(k(z);x,0),
	\end{equation}
	the scalar factors cancel, and hence
	\begin{equation}
		\Psi_0^s(x;z)\Psi_0^s(y;z)^{-1}=e^{i(x-x_0^s)p_\infty^s\sigma_3}F^s(k(z);x,0)e^{-i(x-y)p^s(z)\sigma_3}F^s(k(z);y,0)^{-1}e^{-i(y-x_0^s)p_\infty^s\sigma_3}.
	\end{equation}
	The two outer diagonal matrices are uniformly bounded and therefore do not affect the estimates below.
	
	Write the four theta-function quotients appearing in \eqref{solF} as $G_{ij}^s(x;z)$. Then
	\begin{equation*}
		F^s(k(z);x,0)=\frac{1}{2}\begin{pmatrix}
			(h^s+h^{s,-1})G_{11}^s(x;z) & -(h^s-h^{s,-1})G_{12}^s(x;z)\\
			-(h^s-h^{s,-1})G_{21}^s(x;z) & (h^s+h^{s,-1})G_{22}^s(x;z)
		\end{pmatrix},
	\end{equation*}
	where 
	\begin{equation}\label{eq:h-k}
		h^s=h^s(k(z)) = \left( \frac{z - 1}{z + 1} \right)^{1/2} \left[ \frac{(z - \eta_2^s)(z - \overline{\eta_2^s})}{(z - \eta_1^s)(z - \overline{\eta_1^s})} \right]^{1/4}.
	\end{equation}
	By the transformation \eqref{eq:trF}, we see that $ Y^s(z;x,0) $ has no singularity at $z = \pm 1$.
	Since $\det F^s=1$, a direct computation gives
	\begin{equation}\label{eq:background-propagator-decomposition}
		F^s(k(z);x,0)e^{-i(x-y)p^s(z)\sigma_3}F^s(k(z);y,0)^{-1}
		=\frac{h^s(k(z))^{-2}}{4}\mathcal{F}_1^s(x,y;z)+\frac{1}{2}\mathcal{F}_0^s(x,y;z)+\frac{h^s(k(z))^2}{4}\sigma_3\mathcal{F}_1^s(x,y;z)\sigma_3,
	\end{equation}
	where
	\begin{equation*}
		\begin{aligned}
			\mathcal{F}_0^s(x,y;z)={}&
			\begin{pmatrix}
				G_{11}^s(x;z)G_{22}^s(y;z) & 0\\
				0 & G_{21}^s(x;z)G_{12}^s(y;z)
			\end{pmatrix}
			e^{-i(x-y)p^s(z)}\\
			&+
			\begin{pmatrix}
				G_{12}^s(x;z)G_{21}^s(y;z) & 0\\
				0 & G_{22}^s(x;z)G_{11}^s(y;z)
			\end{pmatrix}
			e^{i(x-y)p^s(z)},
		\end{aligned}
	\end{equation*}
	which is uniformly bounded near the band endpoints and
	\begin{equation}\label{eq:K1-endpoint}
		\begin{aligned}
			\mathcal{F}_1^s(x,y;z)={}&
			\begin{pmatrix}G_{11}^s(x;z)\\ G_{21}^s(x;z)\end{pmatrix}
			\begin{pmatrix}G_{22}^s(y;z)&-G_{12}^s(y;z)\end{pmatrix}e^{-i(x-y)p^s(z)}\\
			&-
			\begin{pmatrix}G_{12}^s(x;z)\\ G_{22}^s(x;z)\end{pmatrix}
			\begin{pmatrix}G_{21}^s(y;z)&-G_{11}^s(y;z)\end{pmatrix}e^{i(x-y)p^s(z)}.
		\end{aligned}
	\end{equation}
	\par 
	Recall the Abel map $J(k)$ in the $k$-plane defined by \eqref{Abelmap}. Under the Joukowski transformation, the corresponding function is (see Remark \ref{rem:relation-Abel-maps})
	$
	\mathcal{J}(z) = J\Big( \frac{z + z^{-1}}{2} \Big),
	$
	with endpoint values
	\begin{equation*}
		\mathcal{J}_+(\eta_1) = J_+(\re \eta_1)= -\frac{1}{2}, \quad 
		\mathcal{J}_+(\overline{\eta_1}) = \frac{1}{2}, \quad 
		\mathcal{J}_+(\eta_2) = J_+(\re \eta_2) = -\frac{1}{2} + \frac{\tau}{2}, \quad 
		\mathcal{J}_+(\overline{\eta_2}) = \frac{1}{2} + \frac{\tau}{2}.
	\end{equation*}
	Note that the superscript ``$+$'' used here for the endpoint value of $\mathcal J$ does not mean that $\mathcal J_+(\zeta)=J_{z+}(\zeta)$. In fact, one has $\mathcal J_+(\zeta)=J_{z-}(\zeta)$, since the integration path defining $\mathcal J_+(\zeta)$ lies essentially on the upper sheet of the Riemann surface in the $k$-plane, which corresponds to the region outside the unit circle in the $z$-plane.
	\par 
	Since both $dp^s$ and the differential defining the Abel map have square-root singularities at the branch points,
	\begin{equation}\label{eq:p-J-endpoint}
		p^s_+(z)-p^s_+(\zeta)=\mathcal O(|z-\zeta|^{1/2}),\qquad \mathcal{J}_+^s(z)-\mathcal{J}_+^s(\zeta)=\mathcal O(|z-\zeta|^{1/2}),
	\end{equation}
	where we clarify once again that the integration path defining $p_+^s(z)$ lies inside the unit circle and is oriented along the positive side of the contour, as shown in Figure \ref{fig:jump}.

	\begin{lem}\label{lem:regularized-background-endpoint}
		Choose $\varepsilon>0$ sufficiently small so that the $\varepsilon$-neighborhoods of the four endpoints $\eta_1^s$, $\eta_2^s$, $\overline{\eta_1^s}$, and $\overline{\eta_2^s}$ are pairwise disjoint. For $s\in\{\ell,r\}$, define the regularized Jost and background solutions near the band endpoints by
		\begin{equation*}
			\widetilde{\Psi}_{0}^s(x;z):=
			\begin{cases}
				h^s(k(z))^{-1}\Psi_{0}^s(x;z),
				&\min\left\{|z-\eta_1^s|,|z-\overline{\eta_1^s}|\right\}<\varepsilon,\\[1ex]
				h^s(k(z))\Psi_{0}^s(x;z),
				&\min\left\{|z-\eta_2^s|,|z-\overline{\eta_2^s}|\right\}<\varepsilon.
			\end{cases}
		\end{equation*}
		\begin{equation}\label{eq:regularized-Psi-definitions}
			\widetilde{\Psi}^s(x;z):=
			\begin{cases}
				h^s(k(z))^{-1}\Psi^s(x;z),
				&\min\left\{|z-\eta_1^s|,|z-\overline{\eta_1^s}|\right\}<\varepsilon,\\[1ex]
				h^s(k(z))\Psi^s(x;z),
				&\min\left\{|z-\eta_2^s|,|z-\overline{\eta_2^s}|\right\}<\varepsilon.
			\end{cases}
		\end{equation}
		\par 
		For every fixed $x\in\mathbb R$, the matrix $\widetilde{\Psi}_0^s(x;\zeta)$ is independent of the choice of the boundary value and satisfy
		\begin{equation}\label{eq:regularized-background-endpoint}
			\begin{aligned}
				h_{\pm}^{s}(k(z))^{-1} \Psi_{0 \pm}^{s}(x ; z) &= \widetilde{\Psi}_{0}^{s}(x ; \zeta) + \mathcal{O}\left(|z-\zeta|^{1 / 2}\right), \quad \zeta \in\{\eta_{1}^{s}, \overline{\eta_{1}^{s}}\}, \\
				h_{\pm}^{s}(k(z)) \Psi_{0 \pm}^{s}(x ; z) &= \widetilde{\Psi}_{0}^{s}(x ; \zeta) + \mathcal{O}\left(|z-\zeta|^{1 / 2}\right), \quad \zeta \in\{\eta_{2}^{s}, \overline{\eta_{2}^{s}}\}.
			\end{aligned}
		\end{equation}
	\end{lem}
	
	\begin{proof}
		We prove the result for $\zeta=\eta_2^s$. The remaining cases are treated similarly. By \eqref{eq:direct_z_background} and \eqref{eq:trF},
		\begin{equation*}
			\Psi_{0\pm}^s(x;z)=e^{i(x-x_0^s)p_\infty^s\sigma_3}\sqrt{1-z^{-2}}\,F_\pm^s(k(z);x,0)e^{-i(x-x_0^s)p_\pm^s(z)\sigma_3}.
		\end{equation*}
		Using the explicit formula for $F^s$, we obtain
		\begin{equation}\label{eq:hF-regularized-eta2}
			h_\pm^s(k(z))F_\pm^s(k(z);x,0)=\frac12
			\begin{pmatrix}
				\bigl(1+h_\pm^s(k(z))^2\bigr)G_{11 \pm}^s(x;z) & \bigl(1-h_\pm^s(k(z))^2\bigr)G_{12 \pm}^s(x;z)\\
				\bigl(1-h_\pm^s(k(z))^2\bigr)G_{21 \pm}^s(x;z) & \bigl(1+h_\pm^s(k(z))^2\bigr)G_{22 \pm}^s(x;z)
			\end{pmatrix}.
		\end{equation}
		Since $\eta_2^s$ is a simple branch point,
		$
		h_\pm^s(k(z))^2=\mathcal O\left(|z-\eta_2^s|^{1/2}\right).
		$
		Moreover, by \eqref{eq:p-J-endpoint} and for each boundary value separately,
		\begin{equation*}
			\mathcal J_\pm^s(z)=\mathcal J_\pm^s(\eta_2^s)+\mathcal O\left(|z-\eta_2^s|^{1/2}\right),\qquad
			p_\pm^s(z)=p_\pm^s(\eta_2^s)+\mathcal O\left(|z-\eta_2^s|^{1/2}\right),
		\end{equation*}
		where $p_+^s(\eta_2^s)=-p_-^s(\eta_2^s)$. Since the theta-function quotients are analytic functions of the Abel map near the endpoint,
		\begin{equation*}
			G_{ij \pm}^s(x;z)=G_{ij \pm}^s(x;\eta_2^s)+\mathcal O\left(|z-\eta_2^s|^{1/2}\right),\qquad i,j=1,2.
		\end{equation*}
		Also, for fixed $x$,
		\begin{equation*}
			e^{-i(x-x_0^s)p_\pm^s(z)\sigma_3}=e^{-i(x-x_0^s)p_\pm^s(\eta_2^s)\sigma_3}+\mathcal O\left(|z-\eta_2^s|^{1/2}\right).
		\end{equation*}
		It follows from \eqref{eq:hF-regularized-eta2} that
		\begin{equation}\label{eq:regularized-background-separate-limits}
			h_\pm^s(k(z))\Psi_{0\pm}^s(x;z)=\widetilde{\Psi}_{0\pm}^s(x;\eta_2^s)+\mathcal O\left(|z-\eta_2^s|^{1/2}\right).
		\end{equation}
		
		It remains to prove that the two endpoint limits in \eqref{eq:regularized-background-separate-limits} coincide. The endpoint identities for the theta-function quotients give
		\begin{equation}\label{eq:G-endpoint-column-relation}
			G_{12 \pm}^s(x;\eta_2^s)=e^{-2i(x-x_0^s)p_\pm^s(\eta_2^s)}G_{11 \pm}^s(x;\eta_2^s),\qquad
			G_{22 \pm}^s(x;\eta_2^s)=e^{-2i(x-x_0^s)p_\pm^s(\eta_2^s)}G_{21 \pm}^s(x;\eta_2^s).
		\end{equation}
		These identities follow from the endpoint values of the Abel map, the quasi-periodicity of the theta function, and the relation $x_0^s=-\Omega_0^s/\Omega_1^s$, with the endpoint constants and the values $p_\pm^s(\eta_2^s)$ chosen consistently on the two sides. After multiplication by the right diagonal phase factor, the first column of the endpoint matrix is proportional to
		\begin{equation*}
			\begin{pmatrix}
				G_{11 \pm}^s(x;\eta_2^s)\\
				G_{21 \pm}^s(x;\eta_2^s)
			\end{pmatrix}
			e^{-i(x-x_0^s)p_\pm^s(\eta_2^s)},
		\end{equation*}
		while, by \eqref{eq:G-endpoint-column-relation}, the second column equals
		\begin{equation*}
			\begin{pmatrix}
				G_{12 \pm}^s(x;\eta_2^s)\\
				G_{22 \pm}^s(x;\eta_2^s)
			\end{pmatrix}
			e^{i(x-x_0^s)p_\pm^s(\eta_2^s)}
			=
			\begin{pmatrix}
				G_{11 \pm}^s(x;\eta_2^s)\\
				G_{21 \pm}^s(x;\eta_2^s)
			\end{pmatrix}
			e^{-i(x-x_0^s)p_\pm^s(\eta_2^s)}.
		\end{equation*}
		Thus 
		\begin{equation}\label{eq:regularized-background-sigma1}
			\widetilde{\Psi}_{0\pm}^s(x;\eta_2^s)\sigma_1=\widetilde{\Psi}_{0\pm}^s(x;\eta_2^s).
		\end{equation}
		
		On $\Sigma_1^s$, we have
		$
		p_+^s(z)=-p_-^s(z)$ and $Y_+^s(z;x,0)=Y_-^s(z;x,0)(-i\sigma_1),
		$
		which gives
		$
		\Psi_{0 +}^s(x;z)=\Psi_{0 -}^s(x;z)(-i\sigma_1).
		$
		Together with \eqref{eq:h-upper-jump}, this yields
		$
		h_+^s(k(z))\Psi_{0 +}^s(x;z)=h_-^s(k(z))\Psi_{0 -}^s(x;z)\sigma_1.
		$
		Letting $z\to\eta_2^s$ and using \eqref{eq:regularized-background-sigma1}, we find
		\begin{equation*}
			\widetilde{\Psi}_{0 +}^s(x;\eta_2^s)=\widetilde{\Psi}_{0 -}^s(x;\eta_2^s)\sigma_1=\widetilde{\Psi}_{0 -}^s(x;\eta_2^s).
		\end{equation*}
		Denoting this common endpoint value by $\widetilde{\Psi}_0^s(x;\eta_2^s)$ proves \eqref{eq:regularized-background-endpoint}.
	\end{proof}
	\par 
	For $\zeta \in \{ \eta_1^s, \overline{\eta_1^s}, \eta_2^s, \overline{\eta_2^s} \}$, we obtain the exact endpoint relations
	\begin{equation*}
		G_{12 \pm}^s(x;\zeta)
		=
		c_{\zeta \pm}^s e^{-2ixp_{ \pm}^s(\zeta)}G_{11 \pm}^s(x;\zeta),
		\qquad
		G_{22 \pm}^s(x;\zeta)
		=
		c_{\zeta \pm}^s e^{-2ixp_{ \pm}^s(\zeta)}G_{21 \pm}^s(x;\zeta),
	\end{equation*}
	where
	$
	p_{ \pm}^s(\eta_1^s)=p_{ \pm}^s(\overline{\eta_1^s})=0$,
	$
	p_{ \pm}^s(\eta_2^s)=p_{ \pm}^s(\overline{\eta_2^s})
	= \pm \frac{\Omega_1^s}{2},
	$
	and the endpoint-dependent constants are given by
	\[
	c_{\eta_1^s \pm}^s=c_{\overline{\eta_1^s } \pm}^s=1,
	\qquad
	c_{\eta_2^s \pm}^s=c_{\overline{\eta_2^s } \pm}^s
	=e^{\mp i\Omega_0^s}.
	\]
	
	Moreover, recall the estimates in \eqref{eq:p-J-endpoint}. The theta-function quotients are analytic functions of $\mathcal J^s(z)$ in a neighborhood of $\mathcal J^s(\zeta)$. Expanding them about $\mathcal J^s(\zeta)$ gives 
	\begin{equation}\label{eq:H-endpoint-relation-1}
		\begin{aligned}
			G_{12}^s(x;z)-c_\zeta^s e^{-2ixp^s(\zeta)}G_{11}^s(x;z)
			&=
			\mathcal O\bigl(|z-\zeta|^{1/2}\bigr),\\
			G_{22}^s(x;z)-c_\zeta^s e^{-2ixp^s(\zeta)}G_{21}^s(x;z)
			&=
			\mathcal O\bigl(|z-\zeta|^{1/2}\bigr).
		\end{aligned}
	\end{equation}
	Here and below, the boundary-value subscripts on $ c^s_{\zeta\pm} $, $p^s_{\pm}$ and $ G^s_{ij\pm}$ are suppressed when a fixed side is understood.
	These estimates are uniform in $x$, since the theta-function quotients are uniformly bounded in $x$.
	
	Substituting \eqref{eq:H-endpoint-relation-1} into \eqref{eq:K1-endpoint}, we obtain
	\begin{equation*}
		\begin{aligned}
			\mathcal{F}_1^s(x,y;z)={}&c_\zeta^se^{-i(x+y)p^s(\zeta)}
			\begin{pmatrix}G_{11}^s(x;z)\\ G_{21}^s(x;z)\end{pmatrix}
			\begin{pmatrix}G_{21}^s(y;z)&-G_{11}^s(y;z)\end{pmatrix}\\
			&\times\left[e^{-i(x-y)(p^s(z)-p^s(\zeta))}-e^{i(x-y)(p^s(z)-p^s(\zeta))}\right]
			+\mathcal O(|z-\zeta|^{1/2}).
		\end{aligned}
	\end{equation*}
	For boundary values on the spectral contour, $p^s(z)-p^s(\zeta)$ is real, and therefore
	\begin{equation*}
		\left|e^{-i(x-y)(p^s(z)-p^s(\zeta))}-e^{i(x-y)(p^s(z)-p^s(\zeta))}\right|
		\leq 2|x-y|\,|p^s(z)-p^s(\zeta)|
		\leq \mathcal{C}|x-y||z-\zeta|^{1/2}.
	\end{equation*}
	Consequently,
	\begin{equation}\label{eq:K1-estimate}
		\|\mathcal{F}_1^s(x,y;z)\|\leq \mathcal{C}(1+|x-y|)|z-\zeta|^{1/2}.
	\end{equation}
	\par 
	Thus, at each endpoint, the singular one of $h^s(k(z))^2$ and $h^s(k(z))^{-2}$ is cancelled by the factor $|z-\zeta|^{1/2}$ in \eqref{eq:K1-estimate}. Since $\mathcal{F}_0^s$ is uniformly bounded, \eqref{eq:background-propagator-decomposition} yields
	\begin{equation}\label{eq:background-propagator-bound}
		\|\Psi_0^s(x;z)\Psi_0^s(y;z)^{-1}\|\leq \mathcal{C}(1+|x-y|),
	\end{equation}
	where $\mathcal{C}$ is independent of $x$, $y$, and $z$ near the band endpoints. Away from the endpoints, the same estimate follows from the boundedness of the background matrices and their inverses.
	\par 
	Moreover, for $\zeta\in\{\eta_2^s, \overline{\eta_2}^s\}$, since the quantities in \eqref{eq:background-propagator-decomposition} are analytic in the local parameter $(z-\zeta)^{1/2}$, expanding the same expression one order further gives
	\begin{equation}\label{eq:background-transition-difference}
		\left\|
		\Psi_0^s(x;z)\Psi_0^s(y;z)^{-1}
		-
		\Psi_0^s(x;\zeta)\Psi_0^s(y;\zeta)^{-1}
		\right\|
		\leq
		\mathcal{C}
		\left(1+|x-y|^2\right)
		|z-\zeta|^{1/2},
		\qquad z\to\zeta.
	\end{equation}
	Indeed, since $p^s(z)-p^s(\zeta)$ is real on the spectral contour, we have
	\begin{equation*}
		e^{\pm i(x-y)p^s(z)}
		={}e^{\pm i(x-y)p^s(\zeta)}
		\left[
		1\pm i(x-y)\bigl(p^s(z)-p^s(\zeta)\bigr)
		\right]
		+\mathcal O\left(|x-y|^2|z-\zeta|\right).
	\end{equation*}
	It follows directly from the definition of $\mathcal F_0^s$ that
	\begin{equation}\label{eq:F0-endpoint-difference}
		\left\|
		\mathcal F_0^s(x,y;z)-\mathcal F_0^s(x,y;\zeta)
		\right\|
		\leq
		\mathcal C(1+|x-y|)|z-\zeta|^{1/2}.
	\end{equation}
	
	For $\mathcal F_1^s$, the endpoint identities for the theta-function quotients imply the cancellation of its constant term. Expanding \eqref{eq:K1-endpoint} one order further in $(z-\zeta)^{1/2}$ therefore yields
	\begin{equation}\label{eq:F1-normalized-difference}
		\left\|
		(z-\zeta)^{-1/2}\mathcal F_1^s(x,y;z)
		-
		\lim_{w\to\zeta}
		(w-\zeta)^{-1/2}\mathcal F_1^s(x,y;w)
		\right\|
		\leq
		\mathcal C(1+|x-y|^2)|z-\zeta|^{1/2}.
	\end{equation}
	In fact, the first-order variation of the theta-function quotients contributes at most $\mathcal O((1+|x-y|)|z-\zeta|)$ to $\mathcal F_1^s$, while the quadratic remainder in the exponential expansion contributes $\mathcal O(|x-y|^2|z-\zeta|)$. Then \eqref{eq:F1-normalized-difference} gives
	\begin{equation}\label{eq:hfa1}
		\left\|
		h^s(k(z))^{-2}\mathcal F_1^s(x,y;z)
		-
		\lim_{w\to\zeta}
		h^s(k(w))^{-2}\mathcal F_1^s(x,y;w)
		\right\|
		\leq
		\mathcal C(1+|x-y|^2)|z-\zeta|^{1/2}.
	\end{equation}
	Furthermore, by \eqref{eq:K1-estimate},
	\begin{equation}\label{eq:hfa2}
		\left\|
		h^s(k(z))^2\sigma_3\mathcal F_1^s(x,y;z)\sigma_3
		\right\|
		\leq
		\mathcal C(1+|x-y|)|z-\zeta|.
	\end{equation}
	Using \eqref{eq:F0-endpoint-difference}, \eqref{eq:hfa1} and \eqref{eq:hfa2} in \eqref{eq:background-propagator-decomposition}, we have the inequality \eqref{eq:background-transition-difference}. At $\zeta\in\{\eta_1^s,\overline{\eta_1^s}\}$, the same proof applies with the roles of $h^s(k(z))^2$ and $h^s(k(z))^{-2}$ interchanged.
	\par 
	For the left Jost solution, the variation-of-parameters equation is
	\begin{equation}\label{eq:Psi-left-property-4}
		\Psi^\ell(x;z)=\Psi_0^\ell(x;z)+\int_{-\infty}^{x}\Psi_0^\ell(x;z)\Psi_0^\ell(y;z)^{-1}i\sigma_3\Delta Q^\ell(y)\Psi^\ell(y;z)\,dy.
	\end{equation}
	Define $\Psi_{(0)}^\ell(x;z)=\Psi_0^\ell(x;z)$, and, for $n\geq0$,
	\begin{equation*}
		\Psi_{(n+1)}^\ell(x;z)=\int_{-\infty}^{x}\Psi_0^\ell(x;z)\Psi_0^\ell(y;z)^{-1}i\sigma_3\Delta Q^\ell(y)\Psi_{(n)}^\ell(y;z)\,dy.
	\end{equation*}
	As in Subsection \ref{app:proMl},	using \eqref{eq:background-propagator-bound} on the ordered region $-\infty<y_n<\cdots<y_1<x$, we obtain
	\begin{equation*}
		\|\Psi_{(n)}^\ell(x;z)\|
		\leq \sup_{y<x_*}\|\Psi_0^\ell(y;z)\|
		\frac{1}{n!}
		\left[\mathcal{C}(1+|x_*|)\int_{-\infty}^{x_*}(1+|y|)\|\Delta Q^\ell(y)\|\,dy\right]^n.
	\end{equation*}
	Since $\|\Delta Q^\ell(y)\|\leq \mathcal{C}|u(y)-u_0^\ell(y)|$ and $u-u_0^\ell\in L^{1,2}((-\infty,x_*])\subset L^{1,1}((-\infty,x_*])$, the Neumann series converges absolutely. Hence
	\begin{equation*}
		\|\Psi^\ell(x;z)\|\leq \sup_{y<x_*}\|\Psi_0^\ell(y;z)\|
		\exp\left\{\mathcal{C}(1+|x_*|)\|u-u_0^\ell\|_{L^{1,1}((-\infty,x_*])}\right\},\qquad x<x_*.
	\end{equation*}
	\par 
	Since $L^{1,2}([x_*,\infty))\subset L^{1,1}([x_*,\infty))$, the same argument on the right gives
	\begin{equation*}
		\|\Psi^r(x;z)\|\leq \sup_{y>x_*}\|\Psi_0^r(y;z)\|
		\exp\left\{\mathcal{C}(1+|x_*|)\|u-u_0^r\|_{L^{1,1}([x_*,\infty))}\right\},\qquad x>x_*.
	\end{equation*}

	We finally refine the endpoint behavior. We first consider the left Jost solution near $\zeta=\eta_2^\ell$. Multiplying \eqref{eq:Psi-left-property-4} by $h_\pm^\ell(k(z))$ gives
	\begin{equation}\label{eq:regularized-left-Volterra}
		h_\pm^\ell(k(z))\Psi_\pm^\ell(x;z)=h_\pm^\ell(k(z))\Psi_{0 \pm}^\ell(x;z)+\int_{-\infty}^{x}\Psi_{0 \pm}^\ell(x;z)\Psi_{0 \pm}^\ell(y;z)^{-1}i\sigma_3\Delta Q^\ell(y)h_\pm^\ell(k(z))\Psi_\pm^\ell(y;z)\,dy.
	\end{equation}
	Then $\widetilde{\Psi}^\ell(x;\zeta)$ is the unique solution of the limiting Volterra equation
	\begin{equation}\label{eq:limiting-left-Volterra}
		\widetilde{\Psi}^\ell(x;\zeta)=\widetilde{\Psi}_0^\ell(x;\zeta)+\int_{-\infty}^{x}\Psi_0^\ell(x;\zeta)\Psi_0^\ell(y;\zeta)^{-1}i\sigma_3\Delta Q^\ell(y)\widetilde{\Psi}^\ell(y;\zeta)\,dy.
	\end{equation}
	The existence and uniqueness of this solution follow from the same Neumann-series argument used in Subsection \ref{app:proMl}, since $u-u_0^\ell\in L^{1,2}(\mathbb R^-)$. Here, since $\Psi_{0,+}^\ell(x;z)\Psi_{0,+}^\ell(y;z)^{-1} = \Psi_{0,-}^\ell(x;z)\Psi_{0,-}^\ell(y;z)^{-1}$, the two boundary values give rise to the same limiting Volterra equation. Therefore, by the uniqueness of the solution to \eqref{eq:limiting-left-Volterra}, we obtain $\widetilde{\Psi}^\ell(x;\zeta):=
	\widetilde{\Psi}_+^\ell(x;\zeta)
	=
	\widetilde{\Psi}_-^\ell(x;\zeta)$.
	\par 
	Subtracting \eqref{eq:limiting-left-Volterra} from \eqref{eq:regularized-left-Volterra}, we obtain
	\begin{equation*}
		\begin{aligned}
			&h_\pm^\ell(k(z))\Psi_\pm^\ell(x;z)-\widetilde{\Psi}^\ell(x;\zeta)\\
			={}&h_\pm^\ell(k(z))\Psi_{0 \pm}^\ell(x;z)-\widetilde{\Psi}_0^\ell(x;\zeta)\\
			&+\int_{-\infty}^{x}\Psi_{0 \pm}^\ell(x;z)\Psi_{0 \pm}^\ell(y;z)^{-1}i\sigma_3\Delta Q^\ell(y)\left(h_\pm^\ell(k(z))\Psi_\pm^\ell(y;z)-\widetilde{\Psi}^\ell(y;\zeta)\right)\,dy\\
			&+\int_{-\infty}^{x}\left(\Psi_{0 \pm}^\ell(x;z)\Psi_{0 \pm}^\ell(y;z)^{-1}-\Psi_0^\ell(x;\zeta)\Psi_0^\ell(y;\zeta)^{-1}\right)i\sigma_3\Delta Q^\ell(y)\widetilde{\Psi}^\ell(y;\zeta)\,dy.
		\end{aligned}
	\end{equation*}
	Using \eqref{eq:regularized-background-endpoint}, \eqref{eq:background-propagator-bound}, and \eqref{eq:background-transition-difference}, we find, for every fixed $x$,
	\begin{equation*}
		\begin{aligned}
			&\left\|h_\pm^\ell(k(z))\Psi_\pm^\ell(x;z)-\widetilde{\Psi}^\ell(x;\zeta)\right\|\\
			\leq{}&\mathcal C_x|z-\zeta|^{1/2}+\mathcal C\int_{-\infty}^{x}(1+|x-y|)\|\Delta Q^\ell(y)\|\left\|h_\pm^\ell(k(z))\Psi_\pm^\ell(y;z)-\widetilde{\Psi}^\ell(y;\zeta)\right\|\,dy\\
			&+\mathcal C|z-\zeta|^{1/2}\int_{-\infty}^{x}(1+|x-y|^2)\|\Delta Q^\ell(y)\|\|\widetilde{\Psi}^\ell(y;\zeta)\|\,dy.
		\end{aligned}
	\end{equation*}
	Here and below, $\mathcal{C}_x>0$ denotes a constant that may depend on the fixed value of $x$, but is independent of $z$.
	The solution $\widetilde{\Psi}^\ell(y;\zeta)$ is bounded for $y\leq x$. Since, for fixed $x$, $1+|x-y|^2\leq\mathcal C_x(1+|y|^2)$, the assumption $u-u_0^\ell\in L^{1,2}(\mathbb R^-)$ implies $\int_{-\infty}^{x}(1+|x-y|^2)\|\Delta Q^\ell(y)\|\|\widetilde{\Psi}^\ell(y;\zeta)\|\,dy<\infty$. Consequently,
	\begin{equation*}
		\begin{aligned}
			&\left\|h_\pm^\ell(k(z))\Psi_\pm^\ell(x;z)-\widetilde{\Psi}^\ell(x;\zeta)\right\|\\
			\leq{}&\mathcal C_x|z-\zeta|^{1/2}+\mathcal C\int_{-\infty}^{x}(1+|x-y|)\|\Delta Q^\ell(y)\|\left\|h_\pm^\ell(k(z))\Psi_\pm^\ell(y;z)-\widetilde{\Psi}^\ell(y;\zeta)\right\|\,dy.
		\end{aligned}
	\end{equation*}
	Applying the same Volterra iteration estimate as in Subsection \ref{app:proMl} yields
	\begin{equation*}
		\left\|h_\pm^\ell(k(z))\Psi_\pm^\ell(x;z)-\widetilde{\Psi}^\ell(x;\zeta)\right\|
		\leq\mathcal C_x|z-\zeta|^{1/2}\exp\left\{\mathcal C_x\int_{-\infty}^{x}(1+|y|)\|\Delta Q^\ell(y)\|\,dy\right\}.
	\end{equation*}
	The exponential factor is independent of $z$, and hence
	\begin{equation*}
		h_\pm^\ell(k(z))\Psi_\pm^\ell(x;z)=\widetilde{\Psi}^\ell(x;\zeta)+\mathcal O\left(|z-\zeta|^{1/2}\right),\qquad z\to\zeta.
	\end{equation*}

	For $\zeta\in\{\eta_1^\ell,\overline{\eta_1^\ell}\}$, the same argument applies after multiplying the Volterra equation by $h_\pm^\ell(k(z))^{-1}$. The right Jost solution is treated similarly, with the interval $(-\infty,x)$ replaced by $(x,\infty)$ and the assumption $u-u_0^r\in L^{1,2}(\mathbb R^+)$. Consequently, for $s\in\{\ell,r\}$,
	\begin{equation}\label{eq:Psi-regularized-endpoint-expansion}
		\begin{aligned}
			h_\pm^s(k(z))^{-1}\Psi_\pm^s(x;z)&=\widetilde{\Psi}^s(x;\zeta)+\mathcal O\left(|z-\zeta|^{1/2}\right),\qquad \zeta\in\{\eta_1^s,\overline{\eta_1^s}\},\\
			h_\pm^s(k(z))\Psi_\pm^s(x;z)&=\widetilde{\Psi}^s(x;\zeta)+\mathcal O\left(|z-\zeta|^{1/2}\right),\qquad \zeta\in\{\eta_2^s,\overline{\eta_2^s}\}.
		\end{aligned}
	\end{equation}
	Since $h^s(k(z))$ and $h^s(k(z))^{-1}$ have fourth-root behavior at the corresponding endpoints, \eqref{eq:Psi-regularized-endpoint-expansion} gives
	\begin{equation*}
		\Psi^s(x;z)=\mathcal O\left(|z-\zeta|^{-1/4}\right),\qquad z\to\zeta,\qquad z\in\Sigma_1^s\cup\Sigma_2^s.
	\end{equation*}
	Away from the spectral bands, the same estimate holds columnwise in the corresponding slit neighborhoods of the endpoints.

	\subsection{Large-$z$ behavior of $b(z)$ on the real line} \label{app:b-large-z}
	In this subsection, we prove the large-$z$ behavior of the scattering coefficient $b(z)$. Assume that
	\begin{equation*}
		u\in W_{\mathrm{loc}}^{4,1}(\mathbb R),
		\qquad
		u-u_0^\ell\in W^{4,1}(\mathbb{R}^-),
		\qquad
		u-u_0^r\in W^{4,1}(\mathbb{R}^+).
	\end{equation*}
	By property \ref{prop:pro3.3-property-3} of Proposition \ref{pro3.3} and the definition \eqref{eq:defW}, for $s\in\{\ell,r\}$,
	\begin{equation}\label{eq:b-W-expansion}
		W^s(x;z)
		=
		\Psi^s(x;z)
		e^{i(x-x_0^s)k(z)\sigma_3}
		=
		I+\frac{C_1^s(x)}{z}
		+\frac{C_2^s(x)}{z^2}
		+\frac{C_3^s(x)}{z^3}
		+\mathcal{O}(z^{-4}),
		\qquad z\to\infty.
	\end{equation}
	Write the first column of $C_j^s$ as
	\begin{equation*}
		C_j^s(x)e_1
		=
		\begin{pmatrix}
			\xi_j^s(x)\\
			\chi_j^s(x)
		\end{pmatrix},
		\qquad
		\xi_0^s=1,
		\qquad
		\chi_0^s=0.
	\end{equation*}
	It follows from \eqref{eq:b-W-expansion} that the first column of $W^s(x;z)$ is given by
	\begin{equation}\label{eq:b-first-column-W}
		W_1^s(x;z)
		=
		\begin{pmatrix}
			1+\dfrac{\xi_1^s}{z}
			+\dfrac{\xi_2^s}{z^2}
			+\dfrac{\xi_3^s}{z^3}
			+\mathcal{O}(z^{-4})
			\\[2ex]
			\dfrac{\chi_1^s}{z}
			+\dfrac{\chi_2^s}{z^2}
			+\dfrac{\chi_3^s}{z^3}
			+\mathcal{O}(z^{-4})
		\end{pmatrix}.
	\end{equation}
	
	We first derive the relations among these coefficients. By \eqref{eq:appwx},
	\begin{equation*}
		\partial_xW^s
		=
		i\sigma_3QW^s
		-ik(z)[\sigma_3,W^s],
		\qquad
		k(z)=\frac{1}{2}\left(z+z^{-1}\right).
	\end{equation*}
	Taking the first column and substituting \eqref{eq:b-first-column-W}, comparison of the coefficients of equal powers of $z$ gives
	\begin{equation}\label{eq:b-coefficient-recursion}
		\chi_1^s=u,
		\qquad
		\partial_x\xi_j^s
		=
		i\overline{u}\,\chi_j^s, \quad j =1,2,3,
		\qquad
		\chi_{j+1}^s
		=
		-i\partial_x\chi_j^s
		+u\xi_j^s
		-\chi_{j-1}^s, \quad 
		j = 1,2.
	\end{equation}
	In particular,
	\begin{equation}\label{eq:b-beta1}
		\chi_1^\ell=\chi_1^r=u.
	\end{equation}
	Moreover, the second identity in \eqref{eq:b-coefficient-recursion} implies
	\begin{equation}\label{eq:b-alpha1-difference}
		\partial_x
		\left(\xi_1^\ell-\xi_1^r\right)
		=
		i\overline{u}
		\left(\chi_1^\ell-\chi_1^r\right)
		=
		0.
	\end{equation}
	Taking $j=1$ in \eqref{eq:b-coefficient-recursion}, we obtain
	$
	\chi_2^s
	=
	-iu_x+u\xi_1^s,
	$
	and hence
	\begin{equation}\label{eq:b-beta2-difference}
		\chi_2^\ell-\chi_2^r
		=
		u\left(\xi_1^\ell-\xi_1^r\right).
	\end{equation}
	Similarly, taking $j=2$ in \eqref{eq:b-coefficient-recursion} and using \eqref{eq:b-alpha1-difference}, we find
	\begin{align}
		\chi_3^\ell-\chi_3^r
		=
		-i\partial_x
		\left(\chi_2^\ell-\chi_2^r\right)
		+u\left(\xi_2^\ell-\xi_2^r\right)
		-\left(\chi_1^\ell-\chi_1^r\right)
		=
		-iu_x\left(\xi_1^\ell-\xi_1^r\right)
		+u\left(\xi_2^\ell-\xi_2^r\right).
		\label{eq:b-beta3-difference}
	\end{align}
	\par 
	By \eqref{eq:defab} and since
	$\Psi_1^s(x;z)
	=
	e^{-i(x-x_0^s)k(z)}W_1^s(x;z)$, we have
	\begin{equation}\label{eq:b-determinant-W}
		b(z)
		=
		\frac{
			e^{-i(2x-x_0^\ell-x_0^r)k(z)}
		}{
			1-z^{-2}
		}
		\det\left[
		W_1^r(x;z),
		W_1^\ell(x;z)
		\right].
	\end{equation}
	Expanding the determinant in \eqref{eq:b-determinant-W} gives
	\begin{equation*}
		\det\left[
		W_1^r,W_1^\ell
		\right]
		=
		\frac{
			\chi_1^\ell-\chi_1^r
		}{z}+
		\frac{
			\chi_2^\ell-\chi_2^r
			+\xi_1^r\chi_1^\ell
			-\xi_1^\ell\chi_1^r
		}{z^2}+
		\frac{
			\chi_3^\ell-\chi_3^r
			+\xi_1^r\chi_2^\ell
			+\xi_2^r\chi_1^\ell
			-\xi_1^\ell\chi_2^r
			-\xi_2^\ell\chi_1^r
		}{z^3}
		+\mathcal{O}(z^{-4}).
	\end{equation*}
	
	We show that all three displayed coefficients vanish. By \eqref{eq:b-beta1}, the coefficient of $z^{-1}$ is zero. Using \eqref{eq:b-beta1} and \eqref{eq:b-beta2-difference}, the coefficient of $z^{-2}$ becomes
	\begin{equation*}
		\chi_2^\ell-\chi_2^r
		+\xi_1^r\chi_1^\ell
		-\xi_1^\ell\chi_1^r=
		u\left(\xi_1^\ell-\xi_1^r\right)
		+u\xi_1^r-u\xi_1^\ell
		=
		0.
	\end{equation*}
	Furthermore,
	\begin{equation}\label{eq:b-middle-cancellation}
		\xi_1^r\chi_2^\ell
		-\xi_1^\ell\chi_2^r
		=
		\xi_1^r
		\left(-iu_x+u\xi_1^\ell\right)
		-\xi_1^\ell
		\left(-iu_x+u\xi_1^r\right)=
		iu_x
		\left(\xi_1^\ell-\xi_1^r\right).
	\end{equation}
	Therefore, by \eqref{eq:b-beta1}, \eqref{eq:b-beta3-difference}, and \eqref{eq:b-middle-cancellation}, the coefficient of $z^{-3}$ is
	\begin{align*}
		&\chi_3^\ell-\chi_3^r
		+\xi_1^r\chi_2^\ell
		+\xi_2^r\chi_1^\ell
		-\xi_1^\ell\chi_2^r
		-\xi_2^\ell\chi_1^r
		\\
		&\quad=
		-iu_x\left(\xi_1^\ell-\xi_1^r\right)
		+u\left(\xi_2^\ell-\xi_2^r\right)
		+iu_x\left(\xi_1^\ell-\xi_1^r\right)
		+u\left(\xi_2^r-\xi_2^\ell\right)
		=
		0.
	\end{align*}
	Consequently,
	\begin{equation*}
		\det\left[
		W_1^r(x;z),W_1^\ell(x;z)
		\right]
		=
		\mathcal{O}(z^{-4}),
		\qquad
		z\in\mathbb{R},
		\qquad
		|z|\to\infty.
	\end{equation*}
	It follows that
	\begin{equation*}
		b(z)e^{-\frac{i}{2}(x_0^\ell-x_0^r)(z - z^{-1})}
		=
		\mathcal{O}(z^{-4}),
		\qquad
		z\in\mathbb{R},
		\qquad
		|z|\to\infty.
	\end{equation*}

	\section{Classical elliptic travelling wave}\label{subsec:classical-elliptic}
	In this section, we recall the classical travelling-wave reduction for the dNLS equation. 
	\par
	Let $v,\varpi_0,x_0\in\mathbb R$ and set
	\begin{equation}\label{eq:classical-travelling-wave}
		u_0^{\mathrm{cl}}(x,t)
		=
		e^{-i\varpi_0t}
		e^{-i\left(\frac{v}{2}x + \frac{v^2}{4}t\right)}
		\psi(\xi),
		\qquad
		\xi=x + vt-x_0.
	\end{equation}
	Substitution \eqref{eq:classical-travelling-wave} into \eqref{eq:dNLS} gives the equation
	\begin{equation}\label{eq:classical-profile-ode}
		\psi''+(\varpi_0+2)\psi-2|\psi|^2\psi=0.
	\end{equation}
	\par
	Write
	$
	\psi(\xi)=\phi(\xi)e^{i\vartheta(\xi)}$,
	with $\phi(\xi)\geq0.
	$
	Separating the real and imaginary parts of \eqref{eq:classical-profile-ode}, we obtain
	\begin{equation*}
		\phi''-\phi(\vartheta')^2+(\varpi_0+2)\phi-2\phi^3=0,
		\qquad
		2\phi'\vartheta'+\phi\vartheta''=0.
	\end{equation*}
	Then we have
	\begin{equation}\label{eq:classical-phase}
		\vartheta(\xi)=j_0\int_0^\xi\frac{ds}{\phi^2(s)}, \qquad 
		(\phi')^2+\frac{j_0^2}{\phi^2}+(\varpi_0+2)\phi^2-\phi^4=C_0,
	\end{equation}
	where $j_0, C_0 \in \mathbb{R}$.
	With
	$
	\mathfrak{F}(\xi):=\phi^2(\xi),
	$
	the second equation in \eqref{eq:classical-phase} becomes
	\begin{equation}\label{eq:classical-cubic}
		(\mathfrak{F}')^2
		=
		4\left(\mathfrak{F}^3-(\varpi_0+2)\mathfrak{F}^2+C_0\mathfrak{F}-j_0^2\right).
	\end{equation}
	Suppose that the cubic polynomial in \eqref{eq:classical-cubic} has three real roots
	$
	\rho_1>\rho_2>\rho_3\geq0.
	$
	Then
	\begin{equation*}
		\varpi_0+2=\rho_1+\rho_2+\rho_3,
		\qquad
		C_0=\rho_1\rho_2+\rho_1\rho_3+\rho_2\rho_3,
		\qquad
		j_0^2=\rho_1\rho_2\rho_3,
	\end{equation*}
	and the periodic branch is characterized by $\rho_3\leq \mathfrak{F}\leq\rho_2$. It is given explicitly by
	\begin{equation}\label{eq:classical-dn-profile}
		\mathfrak{F}(\xi)
		=
		\rho_1-(\rho_1-\rho_3)
		\dn^2\left(
		\sqrt{\rho_1-\rho_3}\,\xi-K(m_1);m_1
		\right),
		\qquad
		m_1=\frac{\rho_2-\rho_3}{\rho_1-\rho_3},
	\end{equation}
	which coincides with \eqref{eq:modu0}. The modulus $|u_0^{\mathrm{cl}}|^2=\mathfrak{F}$ is periodic with period
	$
	L=\frac{2K(m_1)}{\sqrt{\rho_1-\rho_3}}.
	$
	\par
	We next relate the turning points $\rho_j$ to the branch points of the spectral curve.
	\par 
	Introduce
	$
	\widehat{Q}(\xi)=
	\begin{pmatrix}
		0&\overline{\psi(\xi)}\\
		\psi(\xi)&0
	\end{pmatrix}.
	$
	If $\Phi_0$ is written in the form
	\begin{equation*}
		\Phi_0(x,t;k)
		=
		e^{\frac{i}{8} (4\varpi_0 t + 2 v x + v^2 t ) \sigma_3 }\mathcal{P}(\xi;\lambda)
		e^{-2i\mathcal R(\lambda)t\sigma_3},
		\qquad
		\lambda=k+\frac{v}{4},
	\end{equation*}
	then the spatial part of the Lax pair becomes
	$
	\mathcal{P}_\xi
	=
	i\sigma_3(\widehat{Q}-\lambda I)\mathcal{P}.
	$
	The time part reduces to the algebraic equation
	\begin{equation}\label{eq:stationary-time-operator}
		\left(2\lambda\sigma_3\widehat{Q}
		-2\lambda^2\sigma_3
		-(\widehat{Q}^2-I)\sigma_3
		+i(\widehat{Q})_\xi
		+\frac{\varpi_0}{2}\sigma_3\right) \mathcal{P}
		=
		2\mathcal R(\lambda)\mathcal{P}\sigma_3.
	\end{equation}
	Using \eqref{eq:classical-phase}, the determinant of \eqref{eq:stationary-time-operator} gives
	\begin{equation}\label{eq:classical-spectral-curve}
		\mathcal R(\lambda)^2
		=
		\lambda^4
		-\frac{\varpi_0+2}{2}\lambda^2
		-j_0\lambda
		+\frac{(\varpi_0+2)^2-4C_0}{16}.
	\end{equation}
	Thus the spectral curve obtained from the classical travelling-wave reduction is a two-sheeted curve with four branch points.
	\par
	In the $k$-plane used in the Riemann--Hilbert construction, these branch points are
	$
	-1$, $\re\eta_2$, $\re\eta_1$ and $1.
	$
	After the shift $\lambda=k+v/4$, with $v=-\re(\eta_1+\eta_2)$, their sum is zero, consistently with the absence of the cubic term in \eqref{eq:classical-spectral-curve}. Comparing the coefficients of the two quartic polynomials and substituting them into the cubic in \eqref{eq:classical-cubic}, one obtains the factorization
	\begin{equation*}
		\begin{aligned}
			&\mathfrak{F}^3-(\varpi_0+2)\mathfrak{F}^2+C_0\mathfrak{F}-j_0^2\\
			&\quad=
			\left[\mathfrak{F}-\frac{(2+\re\eta_1-\re\eta_2)^2}{4}\right]
			\left[\mathfrak{F}-\frac{(2-\re\eta_1+\re\eta_2)^2}{4}\right]
			\left[\mathfrak{F}-\frac{(\re\eta_1+\re\eta_2)^2}{4}\right].
		\end{aligned}
	\end{equation*}
	Consequently, the turning points are
	\begin{equation}\label{eq:rho-eta-relations}
		\rho_1=\frac{(2+\re\eta_1-\re\eta_2)^2}{4},
		\qquad
		\rho_2=\frac{(2-\re\eta_1+\re\eta_2)^2}{4},
		\qquad
		\rho_3=\frac{(\re\eta_1+\re\eta_2)^2}{4}.
	\end{equation}
	Since $0<\arg\eta_1<\arg\eta_2<\pi$, we have $\re\eta_1>\re\eta_2$, and hence $\rho_1>\rho_2>\rho_3\geq0$. Moreover, \eqref{eq:rho-eta-relations} also gives the elliptic modulus in \eqref{eq:classical-dn-profile} is exactly
	$
	m_1
	=
	\frac{(1-\re\eta_1)(1+\re\eta_2)}{(1+\re\eta_1)(1-\re\eta_2)}.
	$

	\section{An explicit example with pure-step elliptic backgrounds}
	\label{app:pure-step-elliptic-example}
	Fix two genus-one backgrounds determined by the spectral endpoints
	$\eta_1^s,\eta_2^s\in\mathbb T\cap\mathbb C^+$, $s\in\{\ell,r\}$, and assume that the eight band endpoints are pairwise distinct. For simplicity, we take
	$x_0^\ell=x_0^r=0$ and consider the pure-step initial datum
	\begin{equation}\label{eq:pure-step-initial-data}
		u_{\mathrm{ps}}(x)
		=
		\begin{cases}
			u_0^\ell(x),&x\leq0,\\
			u_0^r(x),&x>0.
		\end{cases}
	\end{equation}
	Since $x=t=x_0^s=0$, the expression \eqref{solF} can be simplified to
	\begin{equation*}
		F^s(k(z);0,0)
		=
		\frac12
		\begin{pmatrix}
			h^s+(h^s)^{-1}&(h^s)^{-1}-h^s\\
			(h^s)^{-1}-h^s&h^s+(h^s)^{-1}
		\end{pmatrix},
	\end{equation*}
	where $h^s=h^s(k(z))$ given by \eqref{eq:h-k}. For the datum \eqref{eq:pure-step-initial-data}, the Jost solutions coincide with the corresponding background solutions on their normalization half-lines:
	\begin{equation*}
		\Psi^\ell(x;z)=\Psi_0^\ell(x;z),\quad x\leq0,
		\qquad
		\Psi^r(x;z)=\Psi_0^r(x;z),\quad x\geq0.
	\end{equation*}
	Evaluating the scattering relation \eqref{eq:scat} at $x=0$ gives
	\begin{equation*}
		S(z)
		=
		\Psi_0^r(0;z)^{-1}\Psi_0^\ell(0;z)
		=
		F^r(k(z);0,0)^{-1}F^\ell(k(z);0,0)
		=\frac12
		\begin{pmatrix}
			\Lambda(z)+\Lambda(z)^{-1}&\Lambda(z)-\Lambda(z)^{-1}\\
			\Lambda(z)-\Lambda(z)^{-1}&\Lambda(z)+\Lambda(z)^{-1}
		\end{pmatrix},
	\end{equation*}
	where
	\begin{equation*}
		\Lambda(z)
		:=
		\frac{h^r(k(z))}{h^\ell(k(z))}
		=
		\left[
		\frac{
			(z-\eta_2^r)(z-\overline{\eta_2^r})
			(z-\eta_1^\ell)(z-\overline{\eta_1^\ell})
		}{
			(z-\eta_1^r)(z-\overline{\eta_1^r})
			(z-\eta_2^\ell)(z-\overline{\eta_2^\ell})
		}
		\right]^{1/4},
	\end{equation*}
	with the normalization $\Lambda(z)\to1$ as $z\to\infty$. Consequently,
	\begin{equation*}
		a(z)=\frac12\left(\Lambda(z)+\Lambda(z)^{-1}\right),
		\qquad
		b(z)=\frac12\left(\Lambda(z)-\Lambda(z)^{-1}\right).
	\end{equation*}
	\par 
	By using the fourth-root endpoint singularities of $a(z)$ and $b(z)$, we obtain the following propositions.
	\begin{pro}\label{pro:pure-step-gamma-a12-endpoints}
		Assume that $a(z)$ has no zeros in $\mathbb C^+\setminus(\Sigma_1^\ell\cup\Sigma_1^r)$. The scalar function $\gamma(z)$ defined by \eqref{eq:gamma} has the endpoint behavior
		\begin{equation*}
			\begin{aligned}
				\gamma(z)&=(z-\zeta)^{-1/4}\gamma_\zeta(z),
				&&\zeta\in\left\{\eta_1^\ell,\eta_2^\ell,\overline{\eta_1^r},\overline{\eta_2^r}\right\},\\
				\gamma(z)&=(z-\zeta)^{1/4}\gamma_\zeta(z),
				&&\zeta\in\left\{\eta_1^r,\eta_2^r,\overline{\eta_1^\ell},\overline{\eta_2^\ell}\right\}.
			\end{aligned}
		\end{equation*}
		Here $\gamma_\zeta$ is analytic and nonvanishing in each component of a slit neighborhood of $\zeta$ and admits a finite nonzero non-tangential limit at $\zeta$ from each component.
	\end{pro}

	\begin{pro}\label{pro:pure-step-r12-endpoints}
		$r_1(z)$ has square-root zeros at the endpoints $\eta_j^\ell$ of $\Sigma_1^\ell$, while $r_2(z)$ has square-root zeros at the endpoints $\eta_j^r$ of $\Sigma_1^r$, for $j=1,2$.
	\end{pro}

	\begin{rmk}
		The datum \eqref{eq:pure-step-initial-data} is generally discontinuous at $x=0$. Although the two half-line perturbations vanish identically, the condition $u_{\mathrm{ps}}\in W_{\mathrm{loc}}^{4,1}(\mathbb R)$ fails unless the left and right backgrounds match at the interface through the required derivatives. Thus the decay $b(z)=\mathcal O(z^{-1})$ as $z\to\infty$ in this example does not contradict the fourth-order decay obtained under the additional assumption $u\in W_{\mathrm{loc}}^{4,1}(\mathbb R)$.
	\end{rmk}

	\section*{Acknowledgments}
	The author would like to thank Xiaodong Zhu and Zechuan Zhang for their helpful discussions and valuable suggestions.

	\bibliographystyle{amsplain}

\end{document}